\documentclass[a4paper,twoside]{article}  

\usepackage{amssymb,amsmath,amsthm,dsfont,amsfonts,color,latexsym}

\usepackage{hyperref}

\usepackage{mathrsfs} 

\definecolor{blue}{rgb}{0.00,0.00,1.00}
\definecolor{red}{rgb}{1.00,0.00,0.00}

\allowdisplaybreaks[4]
\renewcommand{\baselinestretch}{1.2}
\def\bq{\begin{equation}}
\def\eq{\end{equation}}
\def\ba{\begin{array}{ccc}}
\def\bal{\begin{array}{lll}}
\def\ea{\end{array}}
\def\bsp{\begin{split}}
\def\esp{\end{split}}

\def\({\left(}\def\){\right)}
\def\[{\left[}\def\]{\right]}

    \def \R   {\mathbb{R}}
    
    \def\C    {\mathbb{C}}
    
    \def\i    {\mathrm{i}}

    \def\S    {\mathbb{S}}
    \def\eps  {\epsilon}

    \def \div  {{\rm div}}
    \def \divx  {{\rm div}_x}

    \def\Tdx   {\nabla_x}

       \def\bq{\begin{equation}}
       \def\eq{\end{equation}}
       \def\be{\begin{equation}}
       \def\ee{\end{equation}}
       \def\bma#1\ema{{\allowdisplaybreaks\begin{align}#1\end{align}}}
       \def\bmas#1\emas{{\allowdisplaybreaks\begin{align*}#1\end{align*}}}
       \def\bln#1\eln{{\allowdisplaybreaks\begin{aligned}#1\end{aligned}}}
       \def\nnm{\notag}
       \def\bgr#1\egr{\allowdisplaybreaks\begin{gather}#1\end{gather}}
       \def\bgrs#1\egrs{\allowdisplaybreaks\begin{gather*}#1\end{gather*}}

       \theoremstyle{plain}
       \newtheorem{lem}{\bf Lemma}[section]
       \newtheorem{thm}[lem]{\textbf{Theorem}}

       \newtheorem{remark}[lem]{\bf Remark}

\begin{document}


\title{ Incompressible Navier-Stokes Limit of the Relativistic Boltzmann Equation with the Optimal Convergence Rate}

\author{ Yanchao Li$^{1,2}$,\, Mingying Zhong$^{1}$ \\[2mm]
 \emph
    {\small\it $^1$School of  Mathematics,
    Guangxi University,  China.}\\
    {\small\it  \& Center for Applied Mathematical of Guangxi (Guangxi University), Guangxi University,  China.}\\
    {\small\it E-mail:\ zhongmingying@gxu.edu.cn}\\
     {\small\it $^2$Post-doctoral Research Station of the First-level Discipline in Mathematics, Guangxi University,  China.}\\
    {\small\it E-mail:\ yanchaoli@gxu.edu.cn}}
\date{ }

\pagestyle{myheadings}
\markboth{Relativistic Boltzmann Equation}%
{Y.-C. Li, M.-Y. Zhong}

 \maketitle

 \thispagestyle{empty}

\begin{abstract}\noindent{
We study the diffusion limit of the classical solution to the relativistic Boltzmann equation with the
initial data near a global relativistic Maxwellian. By using spectral analysis,  we establish the convergence of the
classical solution to the relativistic Boltzmann equation to that of the incompressible Navier-Stokes system,
and give for the first time the optimal convergence rate and a precise estimate of the initial layer. Moreover,
through a refined analysis of the linear collision operator $L$, combined with Duhamel's principle, we obtain
the existence and uniqueness of a global strong solution to the relativistic Boltzmann equation, as well as a
convergence rate independent of the speed of light $\mathbf{c}.$
}


\end{abstract}

\tableofcontents


\section{Introduction}
\label{sect1}
\setcounter{equation}{0}

Derived by Lichnerowicz and Marrot \cite{Lichnerowicz-1} in 1940, the relativistic Boltzmann equation is one of the important equations in relativistic kinetic theory. It describes the evolution of the relevant particles in a relativistic gas. We consider the rescaled relativistic Boltzmann equation \cite{Wang-2}:
\be
\partial_tF_{\epsilon}+\frac{1}{\epsilon}\tilde{v}\cdot\nabla_xF_{\epsilon}=\frac{1}{\epsilon^2}Q\(F_{\epsilon},F_{\epsilon}\),\label{rbe1}
\ee
where $\epsilon>0$ is the Knudsen number related to the mean free path, $F_{\epsilon}=F_{\epsilon}(t,x,v)$ is the distribution function with $(t,x,v)\in\mathbb{R}^{+}\times\mathbb{R}^3\times\mathbb{R}^3$. The relativistic velocity $\tilde{v}$ is defined by
$$
\tilde{v}=\frac{\mathbf{c}v}{v_0},\quad v_0=\sqrt{\mathbf{c}^2+|v|^2},
$$
where $\mathbf{c}\ge 1$ is the speed of light. Throughout this paper, we assume $\epsilon\in(0,1)$. The collision operator $Q(f,h)$ is given by
\bq\label{rBQ}
Q(f,h)=\frac{1}{2}\int_{\mathbb{R}^3}\int_{\mathbb{S}^2}v_M\sigma(g,\theta)\big(f(v')h(u')+f(u')h(v')-f(v)h(u)-f(u)h(v)\big)dud\omega,
\eq
where $v_M=v_M(v,u)$ is M{\o}ller velocity given by
\be\label{rbdl-vmgs}
v_M=\frac{\mathbf{c}g\sqrt{s}}{4u_0v_0},\quad g^2=s-4\mathbf{c}^2,\quad s=2(u_0v_0-u\cdot v+\mathbf{c}^2),
\ee
and $u,v$, $u',v'$  satisfy the conservation law of momentum and energy
$$
 u+v=u'+v',\quad u_0+v_0=u'_0+v'_0.
$$
The function $\sigma(g,\theta)$ is called the differential cross section or the scattering kernel, and it measures the interactions between particles.
The scattering angle $\theta$ is defined by
$$
\cos\theta=\frac{(v^{\mu}-u^{\mu})(v'_{\mu}-u'_{\mu})}{g^2}
$$
with $v^{\mu}=(v_0,v)$ and $v'_{\mu}=(-v'_0,v')$. In the following, we consider the hard ball model:
$$\sigma(g,\theta)=1.$$

There has been significant progress made on the well-posedness and long time behavior of solutions to the relativistic Boltzmann for fixed $\epsilon$. The background of relativistic Boltzmann equation is mentioned in \cite{Cercignani-1}.  The existence and uniqueness of the weak solution to the relativistic Boltzmann equation has been proved in \cite{Dudynski-1,Dudynski-2}. The global solution of relativistic Boltzmann equation near a relativistic Maxwellian is studied in \cite{Glassey-2,Glassey-3} for hard potentials, and in \cite{Duan-1,STRIN-1} for soft potentials. The global existence of solutions near vacuum was investigated in \cite{Glassey-1}. It is shown in \cite{Yang-1} that the global solutions to the relativistic Boltzmann and Landau equations tend to the equilibriums at $(1+t)^{-\frac{3}{4}}$ in $L^2$-norm by using compensating function and  energy method.  The spectrum of the relativistic Boltzmann equation with hard sphere and hard potential was constructed in \cite{ZHAO-1}. The pointwise behavior of the Green's function of the relativistic Boltzmann equation for hard sphere was verified in \cite{Li-2}. The Dirichlet boundary value problem of steady-state relativistic Boltzmann equation in half-line with hard potential model was studied in \cite{Wang-3}. The global well-posedness of the relativistic Boltzmann equation with hard potentials and diffuse reflection boundary condition in bounded domains was studied in \cite{Wang-4}.


The fluid limit to the Boltzmann equation is a classical problem with pioneer work by Bardos-Golse-Levermore in \cite{Bardos-1}. Later, \cite{Golse-1} made substantial progress concerning the convergence of renormalized solutions to Leray solutions of the Navier-Stokes system. Within the perturbative setting, spectral analysis provides a powerful tool for investigating fluid dynamical behavior. As an early example, Ellis and Pinsky \cite{Ellis-1} examined the linear compressible Euler limit of the linear Boltzmann equation, obtaining the convergence rate away from the initial layer. Such an initial layer emerges in fluid limits because of the mismatch in the initial data, especially owing to rapid oscillations of the system's eigenmodes. For the Boltzmann equation, the incompressible Navier-Stokes limit together with an estimate on the initial layer was first established by Bardos and Ukai \cite{Bardos-3} for the linear case. Their analysis in \cite{Bardos-3} relies on a semigroup estimate involving highly oscillatory eigenmodes, derived in \cite{Ukai-1}, which concerns the incompressible limit of the compressible Euler equations.

The fluid dynamical limit of the solution to the relativistic Boltzmann equation near relativistic Maxwellian was also studied in \cite{Wang-2,Yang-3}. In \cite{Wang-2}, the author justify rigorously the validity of the two independent limits from the special relativistic Boltzmann equation to the classical Euler equations without assuming any dependence between the Knudsen number $\epsilon$ and the light speed $\mathbf{c}$. In \cite{Yang-3}, based on the spectrum analysis, the author established a global convergence result of the solution to the relativistic Boltzmann equation towards a solution to the relativistic Euler equations. In contrast to the works on \cite{Bardos-1,Bardos-2,Bardos-3,Guo-4,Masi-1,Nishida-1,Wang-2,Yang-3}, the convergence and the optimal convergence rate of  the classical solution to the relativistic Boltzmann equation towards Navier--Stokes type limit and the estimation of the initial layer  have not been given.

In this paper, we study the diffusion limit of the strong solution to the rescaled relativistic Boltzmann equation \eqref{rbe1} with initial data near the normalized Maxwellian $M$, where (cf. \cite{STRIN-1,Wang-2})
\bq\label{mmv-dif}
M=p_0e^{-\frac{\mathbf{c}v_0}{k_0T}},\quad p_0=\frac{1}{4\pi\mathbf{c}k_0TK_2(\frac{\mathbf{c}^2}{k_0T})},
\eq
where $T$ is the temperature, $k_0$ is Boltzmann's constant and $K_j(\cdot)$ is the Bessel function
\bq\label{rbdl-kjzdif}
K_j(z)=\(\frac{z}{2}\)^j\frac{\Gamma(\frac{1}{2})}{\Gamma(j+\frac{1}{2})}\int_1^{\infty}e^{-zt}(t^2-1)^{j-\frac{1}{2}}dt,\quad j\geq0,~ z>0.
\eq

Hence, we denote the perturbation of $F_{\epsilon}$ as
$$
F_{\epsilon}=M+\epsilon\sqrt{M}f_{\epsilon}.
$$
Then the Cauchy problem of the relativistic Boltzmann equation \eqref{rbe1} for $f_{\epsilon}$ is
\bma
&\partial_tf_{\epsilon}+\frac{1}{\epsilon}\tilde{v}\cdot\nabla_xf_{\epsilon}-\frac{1}{\epsilon^2}Lf_{\epsilon}=\frac{1}{\epsilon}\Gamma(f_{\epsilon},f_{\epsilon}),\label{rbe2}\\
&f_{\epsilon}(0,x,v)=f_{0}(x,v),\label{Pm3}
\ema
where the linearized collision operator $Lf_{\epsilon}$ and nonlinear term $\Gamma\(f_{\epsilon},f_{\epsilon}\)$ are defined by
\bq\label{rbepz-1}
\left\{\bln
&Lf_{\epsilon}=\frac{1}{\sqrt{M}}\big[Q(M,\sqrt{M}f_{\epsilon})+Q(\sqrt{M}f_{\epsilon},M)\big], \\
&\Gamma\(f_{\epsilon},f_{\epsilon}\)=\frac{1}{\sqrt{M}}Q(\sqrt{M}f_{\epsilon},\sqrt{M}f_{\epsilon}).
\eln\right.
\eq
The linearized collision operator $L$ can be written as
\bq\label{rbenu-k}
\left\{\bln
&(Lh)(v)=(Kh)(v)-\nu(v)h(v),\\
&\nu(v)=\int_{\R^3}\int_{\S^2}v_MM(u)d\omega du,\\
&Kh=\int_{\R^3}\int_{\S^2}v_M\big(h(v') M^{\frac12}(u')+h(u')M^{\frac12}(v')-h(u)M^{\frac12}(v)\big) M^{\frac12}(u)d\omega du\\
&\quad~ =\int_{\R^3}k(v,u)h(u)du,
\eln\right.
\eq
where $\nu(v)$, the collision frequency, is a real function, and $K$ is a self-adjoint compact operator on $L^2(\R^3_v)$ with a real symmetric integral kernel $k(v,u)$. In addition, $\nu(v)$ satisfies
\bq\label{rbdl-nu1}
\nu_0w(v)\leq\nu(v)\leq\nu_1w(v),
\eq
where $\nu_0, \nu_1$ are positive constants depended on $k_0$ and $T$, and
\bmas
w(v)=
\left\{\bal
1+|v|,\quad&|v|\leq\mathbf{c},\\
\mathbf{c},\quad&|v|\geq\mathbf{c}.
\ea\right.
\emas

The null space of the operator $L$, denoted by $N_0$, is a subspace spanned by the orthogonal basis $\{\chi_j,~j=0,1,2,3,4\}$ with
\bq\label{chii-1}
\chi_0=\sqrt{M},\quad \chi_j=p_1v_j\sqrt{M} ~~ (j=1,2,3),\quad \chi_4=p_2(v_0-p_3)\sqrt{M},
\eq
where (cf. \cite{Cao-1})
\bq\label{rbdl-p-123}
\left\{\bln
p_1& =\bigg(k_0T\frac{K_3(\frac{\mathbf{c}^2}{k_0T})}{K_2(\frac{\mathbf{c}^2}{k_0T})}\bigg)^{-\frac{1}{2}}, \quad
p_3 =\mathbf{c}\frac{K_3(\frac{\mathbf{c}^2}{k_0T})}{K_2(\frac{\mathbf{c}^2}{k_0T})}-\frac{k_0T}{\mathbf{c}},\\
p_2&  = \bigg( \mathbf{c}^2\bigg(1-\frac{K_3(\frac{\mathbf{c}^2}{k_0T})^2}{K_2(\frac{\mathbf{c}^2}{k_0T})^2} \bigg) + 5k_0T \frac{K_3(\frac{\mathbf{c}^2}{k_0T})}{K_2(\frac{\mathbf{c}^2}{k_0T})}-\frac{(k_0T)^2}{\mathbf{c}^2}\bigg)^{-\frac{1}{2}}.
\eln\right.
\eq
It holds that (cf. Lemma \ref{p123})
$$p_1\sim\frac{1}{\sqrt{k_0T}}, \quad p_2\sim \frac{\mathbf{c}}{k_0T}, \quad  p_3\sim \mathbf{c}.$$
Here $q_1\sim q_2$ means that there exist positive constants $C_1,C_2$ independent of $\mathbf{c}$ such that $C_1q_1\le q_2\le C_2q_1.$

Let $L^2(\R^3)$ be a Hilbert space of complex-value functions $f(v)$ on $\R^3$ with the inner product and the norm
$$
(f_1,f_2)=\int_{\R^3}f_1(v)\overline{f_2(v)}dv,\quad \|f\|=\(\int_{\R^3}|f(v)|^2dv\)^{\frac{1}{2}}.
$$
Introduce the macro-micro decomposition as follows
\bq\label{mami}
f=P_0f+P_1f,\quad  P_0f=\sum^4_{i=0}(f,\chi_i)\chi_i.
\eq
The linearized operator $L$ is non-positive and moreover, $L$ is locally coercive in the sense that there is a constant $\mu>0$ such that (cf. \cite{Glassey-1,Glassey-2,STRIN-1})
\bq
(Lf,f)\leq-\mu\|P_1f\|^2,\quad f\in D(L),\label{LF-F}
\eq
where $D(L)$ is the domain of $L$ given by
$$
D(L)=\{f\in L^2(\R^3)\,|\,\nu(v)f\in L^2(\R^3)\}.
$$

This paper aims to prove the convergence and establish the convergence rate of strong solution $f_{\epsilon}$ towards $u$, where $u=n\chi_0+m\cdot \chi'+q\chi_4$  with $\chi'=(\chi_1,\chi_2,\chi_3)$, and $(n,m,q)(t,x)$ is the solution of the following incompressible Navier--Stokes (iNS) system:
\be
\left\{\bln
&\divx m=0,\quad bn+aq=0,\\
&\partial_tm-\kappa_1\Delta_xm+\nabla_xp=\eta_1\divx\(m\otimes m\),\\
&\partial_tq-\kappa_0\Delta_xq=\eta_0\divx\(qm\),
\eln\right.\label{rNS-1}
\ee
where $p$ is the pressure, and the initial data $(n,m,q)(0)$ satisfies
\bq\label{rbdl-innmq}
\left\{\bln
&m(0)=(f_0,\chi')-\Delta_x^{-1}\nabla_x\div_x(f_0,\chi'),\\
&n(0)=-\frac ab q(0)=-\frac{a}{\sqrt{a^2+b^2}}(f_0,E_0).
\eln\right.
\eq
Here, the parameters $a,b>0$ (See Lemma \ref{rbdl-ap-lp1}) and $E_0\in N_0$ are defined by
\bma
&a=(\tilde{v}_1\chi_1,\chi_4)=\frac{p_2}{3p_1}\bigg( \mathbf{c}+\frac{2k_0T}{\mathbf{c}}\frac{K_2(\frac{\mathbf{c}^2}{k_0T})}{K_3(\frac{\mathbf{c}^2}{k_0T})}\bigg) ,\label{rbdl-a}\\
&b=(\tilde{v}_1\chi_0,\chi_1)=\frac{2k_0Tp_1}{3} ,\label{rbdl-b}\\
& E_0=-\frac{a}{\sqrt{a^2+b^2}}\chi_0+ \frac{b}{\sqrt{a^2+b^2}}\chi_4,\label{rbdl-e0}
\ema
and the heat coefficient  $\kappa_0 >0$ and viscosity coefficient $ \kappa_1>0$ are given by
\bmas
&\kappa_0=-(L^{-1}P_1(\tilde{v}_1E_0),\tilde{v}_1E_0),\quad \kappa_1=-(L^{-1}P_1(\tilde{v}_1\chi_2),\tilde{v}_1\chi_2),\\
 &\eta_0= -p_1(P_1(v_1E_0),\tilde{v}_1E_0),\quad \eta_1=-p_1(P_1(v_1\chi_2),\tilde{v}_1\chi_2).
\emas
Note that (cf. Lemma \ref{rbdl-ap-lp1})
$$ a\sim  \frac{\mathbf{c}^2}{\sqrt{k_0T}},\quad b\sim \sqrt{k_0T} , \quad \kappa_j\sim1,\quad \eta_j\sim -\sqrt{k_0T},\quad j=0,1.$$

In general, the convergence is not uniform near $t=0$ because of the appearance of an initial layer. However, we can show that if the initial data $f_0$ satisfies
\bq\label{rbf0}
\left\{\bln
&f_0(x,v)=n_0(x)\chi_0+m_0(x)\cdot\chi'+q_0(x)\chi_4,\\
&\divx m_0=0,\quad bn_0+aq_0=0,
\eln\right.
\eq
then the uniform convergence is up to $t=0$.


\noindent\textbf{Notations:}\ \ Before state the main results in this paper, we list some notations. Throughout this paper, $C$ denote a generic positive constant. For any $\alpha=(\alpha_1,\alpha_2,\alpha_3)\in\mathbb{N}^3$, we denote
$$
\partial^{\alpha}_{x}=\partial^{\alpha_1}_{x_1}\partial^{\alpha_2}_{x_2}\partial^{\alpha_3}_{x_3}.
$$
The Fourier transform of $f=f(x,v)$ is defined by
\bq
\hat{f}(\xi,v)=\mathcal{F}f(x,v)=\frac{1}{(2\pi)^{\frac{3}{2}}}\int_{\mathbb{R}^3}e^{-\mathrm{i}x\cdot\xi}f(x,v)dx,
\eq
where and throughout this paper we denote $\mathrm{i}=\sqrt{-1}$.

Set the Sobolev space $H^N=\{f\in L^2\(\mathbb{R}^3_x\times\R^3_v\)|\|f\|_{H^N}<\infty\}$ equipped with the norms
$$
\|f\|_{H^N}=\sum_{|\alpha|\le N}\|\partial^{\alpha}_{x}f\|_{L^2(\mathbb{R}^3_x\times\R^3_v)}.
$$
Denote the Banach space $L^{\infty}=L^{\infty}_x(L^2_v)$ equipped with the norm
$$
L^{\infty}=L^{\infty} (\R^3_x,L^2(\R^3_v) ),\quad \|f\|_{L^{\infty}}=\sup_{x\in\R^3}\(\int_{\R^3}|f(x,v)|^2dv\)^{\frac{1}{2}}.
$$
For $q\geq1$ and $k\ge 1$, define
\bmas
L^{2,q}=L^2(\mathbb{R}^3_v,L^q(\mathbb{R}^3_x)),&\quad \|f\|_{L^{2,q}}=\bigg(\int_{\mathbb{R}^3}\(\int_{\mathbb{R}^3}|f(x,v)|^qdx\)^{\frac{2}{q}}dv\bigg)^{\frac{1}{2}},\\
W^{k,q}=L^2(\mathbb{R}^3_v,W^{k,q}(\mathbb{R}^3_x)),&\quad \|f\|_{W^{k,q}}=\bigg(\sum_{|\alpha|\leq k}\int_{\R^3}\(\int_{\R^3}\big|\partial_x^{\alpha}f(x,v)\big|^qdx\)^{\frac{2}{q}}dv\bigg)^{\frac{1}{2}}.
\emas
For simplicity, we denote $L^2=L^{2,2}$.

For $k\geq0$, define
$$
\|f\|_{L^{\infty}_{v,k}(H^N_x)}=\sup_{v\in\R^3}(1+|v|)^{k}\|f(\cdot,v)\|_{H^N_x}.
$$

Now we are ready to state main results in this paper.

\begin{thm}\label{thm-1}
(1) Let $N\geq2$. For any $\epsilon\in(0,1)$, there exists a small constant $\delta_0>0$ independent of $\mathbf{c}$ such that if $\|f_0\|_{L^{\infty}_{v,3}(H^N_x)}+\|f_0\|_{L^{2,1}}\leq\delta_0$, then the relativistic Boltzmann equation \eqref{rbe2}--\eqref{Pm3} admits a unique global solution $f_{\epsilon}=f_{\epsilon}(t,x,v)$ satisfying the following time-decay estimate:
\bq\label{thm-1-1}
\|f_{\epsilon}(t)\|_{L^{\infty}_{v,3}(H^N_x)}\leq C\delta_0(1+t)^{-\frac{3}{4}},
\eq
where $C>0$ is a constant independent of $\mathbf{c}$ and $\epsilon$.

(2) There exists a small constant $\delta_0>0$ such that if $\|f_0\|_{H^N}+\|f_0\|_{L^{2,1}}\leq\delta_0$, then the iNS system \eqref{rNS-1} admits a unique global solution $(n,m,q)(t,x)\in L^{\infty}_t(H^N_x)$. Moreover, $u(t,x,v)=n(t,x)\chi_0+m(t,x)\cdot \chi'+q(t,x)\chi_4$ has the following time-decay rate:
$$
\|u(t)\|_{H^N}\leq C\delta_0(1+t)^{-\frac{3}{4}},
$$
where $C>0$ is a constant.
\end{thm}


\begin{thm}\label{thm-2}
Let $f_{\epsilon}=f_{\epsilon}(t,x,v)$ be the global solution to the relativistic Boltzmann equation \eqref{rbe2}--\eqref{Pm3}, and let $(n,m,q)=(n,m,q)(t,x)$ be the global solution to the iNS system \eqref{rNS-1}. Then, there exists a small constant $\delta_0>0$ independent of $\mathbf{c}$ such that if $\|f_0\|_{L^{\infty}_{v,3}(H^4_x)}+\|f_0\|_{L^{2,1}}\leq\delta_0$, then we have
\bq
\|f_{\epsilon}(t)-u(t)\|_{L^{\infty}}\leq C\delta_0\bigg(\epsilon|\ln\epsilon|^2(1+t)^{-\frac{3}{4}}+\(1+\frac{t}{\epsilon}\)^{-1}\bigg),\label{thm-2-1}
\eq
where $u(t,x,v)=n(t,x)\chi_0+m(t,x)\cdot \chi'+q(t,x)\chi_4$.

Moreover, if the initial data $f_0$ satisfies \eqref{rbf0} and $\|f_0\|_{L^{\infty}_{v,3}(H^4_x)}+\|f_0\|_{L^{2,1}}\leq\delta_0$, then we have
\bq
\|f_{\epsilon}(t)-u(t)\|_{L^{\infty}}\leq C\delta_0\epsilon|\ln\epsilon|(1+t)^{-\frac{3}{4}}.\label{thm-2-2}
\eq
\end{thm}

\begin{remark}\label{dl-rmkil}
Define the oscillation part of $f_{\epsilon}$ as
\be
u_{\epsilon}^{osc}(t,x,v)=\sum_{j=-1,1}\mathcal{F}^{-1}\(e^{\frac{-\mathrm{i}|\xi|\mu_j t}{\epsilon}-A_j|\xi|^2t}(P_0\hat{f}_0,E_j(\xi)\big)E _j(\xi)\),\label{dlxg-uosc}
\ee
where $\mu_j$, $A_j$ and $E_j(\xi)$, $j=-1,1$ are defined by \eqref{sp4}. Then, under the first assumption of Theorem \ref{thm-2}, we have
$$
\|f_{\epsilon}(t)-u(t)-u_{\epsilon}^{osc}(t)-e^{\frac{tB_{\epsilon}}{\epsilon^2}}P_1f_0\|_{L^{\infty}}\leq C\delta_0\epsilon|\ln\epsilon|(1+t)^{-\frac{3}{4}}.
$$
Hence, $u_{\epsilon}^{osc}(t)$ and $e^{\frac{tB_{\epsilon}}{\epsilon^2}}P_1f_0$ are the essential components for generating the initial layer.
\end{remark}

\begin{remark}
The coefficients $|\ln\epsilon|^2$ and $|\ln\epsilon|$ in Theorem \ref{thm-2} come from the time convolutions with the initial layer. 
Precisely, under the first assumption of Theorem \ref{thm-2},  the terms related to $|\ln\epsilon|^2$ consist of two parts: the first part  associated with $|\ln\epsilon|$ mainly arises from the convolution of $e^{\frac{tB_{\eps}}{\eps^2}}$ and nonlinear term, namely (See \eqref{dlth2-3-1-2-1})
$$
\int^t_0\(1+\frac{t-s}{\epsilon}\)^{-1}(1+s)^{-\frac{3}{2}}ds\leq C\epsilon|\ln\epsilon|(1+t)^{-1},
$$
the second part  associated with $|\ln\epsilon|^2$  mainly arises from the convolution of $V(t)$ and the initial layer, such as (See \eqref{dlth2-4-7-6})
\bmas
&\quad\(\int_0^{\frac{t}{2}}+\int^t_{\frac{t}{2}}\)(1+t-s)^{-\frac{3}{4}}(t-s)^{-\frac{1}{2}}\(1+\frac{s}{\epsilon}\)^{-1}(1+s)^{-\frac{3}{4}}ds\\
&\leq C\epsilon |\ln\epsilon|(1+t)^{-\frac{3}{4}}t^{-\frac12}+\cdot\cdot\cdot\leq C\epsilon|\ln\epsilon|^2(1+t)^{-1}+\cdot\cdot\cdot.
\emas
However, under the second assumption of Theorem \ref{thm-2}, the initial layer vanishes, so the estimation  only involve the parts associated with $|\ln\epsilon|$ (See \eqref{dlth-2-2-1j1}--\eqref{dlth-2-2-5}).
\end{remark}

Before the rest of the introduction, we will briefly present the main ideas and the approach of the analysis in the proof. Based on the spectral analysis \cite{ZHAO-1} and the ideas inspired by \cite{Bardos-3,Li-1}, we prove the convergence rates given in Theorem \ref{thm-2} of diffusion limit of the relativistic Boltzmann equation. First of all, the solution $f_{\epsilon}=f_{\epsilon}(t,x,v)$ to the relativistic Boltzmann equation \eqref{rbe2}--\eqref{Pm3} can be represented by
$$
f_{\epsilon}(t)=e^{\frac{tB_{\epsilon}}{\epsilon^2}}f_0+\frac{1}{\epsilon}\int_0^te^{\frac{(t-s)B_{\epsilon}}{\epsilon^2}}\Gamma(f_{\epsilon},f_{\epsilon})ds,
$$
and the solution $u(t,x,v)=n(t,x)\chi_0+m(t,x)\cdot \chi'+q(t,x)\chi_4$ to the iNS system \eqref{rNS-1} can be represented by
\bmas
u(t)=&V(t)P_0f_0+\int^t_0V(t-s) P_{0}(\tilde{v}\cdot\nabla_xL^{-1}\Gamma(u,u)) ds,
\emas
where $V(t)$ is defined by \eqref{lns4j1-1}.

The linear relativistic Boltzmann operator $B_{\eps}(\xi)=L-\i\eps(\tilde{v}\cdot\xi)$ satisfies that $B_{\epsilon}(\xi)=B(\eps\xi)$. Thus, we can show that there exist five eigenvalue $\lambda_j(\epsilon|\xi|)$, $j=-1,0,1,2,3$ for $\epsilon|\xi|$ being small and they satisfy (See Lemma \ref{3dmvpbsp10}):
\bq\label{bjsp-1}
\lambda_j(\epsilon|\xi|)=-\mathrm{i}\epsilon|\xi| \mu_j-\epsilon^2|\xi|^2A_j+O(\epsilon^3|\xi|^3),\quad \epsilon|\xi|\leq r_0,
\eq
where $\mu_j$ and $A_j$ are constants defined by \eqref{sp4}. With the help of the expansion \eqref{bjsp-1}, we can rewrite $e^{\frac{tB(\epsilon\xi)}{\epsilon^2}}$ as
\bmas
e^{\frac{tB(\epsilon\xi)}{\epsilon^2}}\hat{f}_0=&\sum^3_{j=-1}e^{\frac{-\mathrm{i}|\xi|\mu_jt}{\epsilon}-|\xi|^2A_jt+O(\epsilon|\xi|^3)t}\(\big(P_0\hat{f}_0,E_j(\xi)\big)E _j(\xi)+O(\epsilon|\xi|)\)\\
& +S_2(t,\xi,\epsilon)\hat{f}_0,
\emas
where $E_j(\xi)$ $(j=-1,0,1,2,3)$ are given in \eqref{sp4}, and $S_2(t,\xi,\epsilon)$ is given in Lemma \ref{3dmvpbsp10} and satisfy $\|S_2(t,\xi,\epsilon)\hat{f}_0\|\le Ce^{-\frac{\tau_1t}{\eps^2}}\|\hat{f}_0\|$. By using the following key estimate (See Lemma \ref{3dmvpbfas4}):
$$
\bigg\|\mathcal{F}^{-1}\(e^{\frac{-\mathrm{i}|\xi|\mu_{j}t}{\epsilon}}\alpha(\omega)(1+|\xi|)^{-3}\)\bigg\|_{L^{\infty}_x}\leq C\(\frac{t}{\epsilon}\)^{-1}, \quad j=\pm1,
$$
with $\alpha(\omega)$ is a smooth function for $\omega=\frac{\xi}{|\xi|}\in\S^2$, we can establish the optimal convergence rate of the semigroup $e^{\frac{tB_{\epsilon}}{\epsilon^2}}$ to its first and second order fluid limits in  $L^{\infty}$ norm as given in Lemma \ref{3dmvpbfas6}.

By using estimates of the convergence rates for the fluid limits of the linear relativistic Boltzmann equation, we establish the optimal convergence rate from the strong solution $f_{\epsilon}=f_{\epsilon}(t,x,v)$ of the nonlinear relativistic Boltzmann equation towards the solution $u=u(t,x,v)$ of the iNS system. Consequently, we obtain precise structure of the initial layer.

The rest of this paper is organized as follows. In Section \ref{sect2}, we present results on the spectral analysis of the linear operator associated with the linearized relativistic Boltzmann equation. In Section \ref{sect3}, we establish first- and second-order fluid approximations for the solution of the linearized relativistic Boltzmann equation. In Section \ref{sect4}, we prove the convergence of the global solution of the original nonlinear relativistic Boltzmann equation to the solution of the nonlinear iNS system and establish its convergence rate.

%
%
%
%

\section{Spectral analysis}
\label{sect2}
\setcounter{equation}{0}

In this section, we are concerned with the spectral analysis of the linear relativistic Boltzmann operator $B(\epsilon\xi)$ defined by \eqref{bep1}, which will be applied to study diffusion limit of the solution to the relativistic Boltzmann equation \eqref{rbe2}--\eqref{Pm3}. From \eqref{rbe2}--\eqref{Pm3}, we have the following linearized relativistic Boltzmann equation:
\bq\label{Bm1}
\left\{\bal
\epsilon^2\partial_tf_{\epsilon}=B_{\epsilon}f_{\epsilon},\quad t>0,\\
f_{\epsilon}(0,x,v)=f_{0}(x,v),
\ea\right.
\eq
where
$$
B_{\epsilon}f=Lf-\epsilon \tilde{v}\cdot\nabla_xf.
$$

Taking Fourier transform to \eqref{Bm1} to get
\bq
\left\{\bal
\epsilon^2\partial_t\hat{f}_{\epsilon}=B(\epsilon\xi)\hat{f}_{\epsilon},\quad t>0,\\
\hat{f}_{\epsilon}(0,\xi,v)=\hat{f}_{0}(\xi,v),
\ea\right.
\eq
where
\bq
B(\epsilon\xi)=L-\mathrm{i}\epsilon\tilde{v}\cdot \xi.\label{bep1}
\eq

Let $\zeta=\epsilon\xi$, we have the following results about the spectrum structure and semigroup of the operator $B(\zeta)=L-\mathrm{i}\tilde{v}\cdot\zeta$.
\begin{lem}[\cite{ZHAO-1}]\label{rbesp1}
The operator $B(\zeta)$ generates a continuous contraction semigroup on $L^2(\R^3_v)$, which satisfies
\bq
\|e^{tB(\zeta)}f\|\leq\|f\|,\quad \forall t>0, \quad f\in L^2(\mathbb{R}^3).
\eq
\end{lem}

\begin{lem}[\cite{ZHAO-1}]\label{rbesp2-si1}
For all $\zeta\in\R^3$, the following conditions hold.

{\rm (1)} $\sigma_{ess}(B(\zeta))\subset\{\lambda\in\mathbb{C} \,|\,\mathrm{Re}\lambda\leq-\nu_0 \}\ \mathrm{and}\ \sigma(B(\zeta))\cap \{\lambda\in\mathbb{C} \,|-\nu_0<\mathrm{Re}\lambda\leq0 \}\subset\sigma_d(B(\zeta)).$

{\rm (2)} If $\lambda(\zeta)$ is an eigenvalue of $B(\zeta)$, then $\mathrm{Re} \lambda(\zeta)<0$ for any $|\zeta|\neq0$ and $\lambda(\zeta)=0$ iff $|\zeta|=0$.

\end{lem}

Let $A(\zeta)$ be the operator obtained by dropping $K$ from $B(\zeta)$, we have
\bma
&A(\zeta)f=(-\mathrm{i}\zeta\cdot\tilde{v}-\nu(v))f,\quad f\in D(A(\zeta)),\label{hataex-1}\\
&D(A(\zeta))=\big\{f\in L^2(\R^3)\,|\,\nu(v)f\in L^2(\R^3)\big\}.
\ema
For $\mathrm{Re}\lambda>-\nu_0$, we decompose  $B(\zeta)$ into
\bq
\lambda-B(\zeta)=\lambda-A(\zeta)-K=\(I-K(\lambda-A(\zeta))^{-1}\)(\lambda-A(\zeta)),\label{lamdaba}
\eq
and estimate the right hand terms of \eqref{lamdaba} as follows.

\begin{lem}\label{rbesplamdaa}
The following conditions hold.

{\rm (1)} For $\delta>0$, we have
\bq
\sup_{\mathrm{Re}\lambda>-\nu_0+\delta,\mathrm{Im}\lambda\in\R}\|K(\lambda-A(\zeta))^{-1}\|^{2}\leq C\delta^{-\frac{8}{5}}(1+|\zeta|)^{-\frac{2}{5}}.\label{rbesplamdaa-im}
\eq

{\rm (2)} For any $\delta,r_0>0$ and $|\mathrm{Im}\lambda|\ge 2\mathbf{c}r_0$, we have
\bq
\sup_{\mathrm{Re}\lambda>-\nu_0+\delta,|\zeta|\leq r_0}\|K(\lambda-A(\zeta))^{-1}\|^{2}\leq C\delta^{-\frac{8}{5}}(1+|\mathrm{Im}\lambda|)^{-\frac{2}{5}}.\label{rbesplamdaa-ze}
\eq
Here, $C>0$ is a constant independent of $\mathbf{c}$.
\end{lem}
\begin{proof}
Firstly, we want to show \eqref{rbesplamdaa-im}. By \eqref{hataex-1}, we have
\bma
\|K(\lambda-A(\zeta))^{-1}f\|^2&=\int_{\R^3}\(\int_{\R^3}k(v,u)(\lambda+\nu(u)+\mathrm{i}\zeta\cdot\tilde{u})^{-1}f(u)du\)^2dv\nnm\\
&\leq2\int_{\R^3}\(\int_{\{|u|\leq R\}}k(v,u)(\lambda+\nu(u)+\mathrm{i}\zeta\cdot\tilde{u})^{-1}f(u)du\)^2dv\nnm\\
&\quad+2\int_{\R^3}\(\int_{\{|u|> R\}}k(v,u)(\lambda+\nu(u)+\mathrm{i}\zeta\cdot\tilde{u})^{-1}f(u)du\)^2dv\nnm\\
&=:I_1+I_2.
\ema
For $I_1$, we have by Lemma \ref{rbdl-svku1alfa} that
\bma
I_1&\leq2\int_{\R^3}\(\int_{\{|u|\leq R\}}|\lambda+\nu(u)+\mathrm{i}\zeta\cdot\tilde{u}|^{-2}du\)\(\int_{\{|u|\leq R\}}|k(v,u)|^2|f(u)|^2du\)dv\nnm\\
&\leq C\|f\|^2\int_{\{|u|\leq R\}}|\lambda+\nu(u)+\mathrm{i}\zeta\cdot\tilde{u}|^{-2}du\nnm\\
&=C\|f\|^2\int_{\{|u|\leq R\}}\frac{1}{|\lambda+\nu(u)+\mathrm{i}|\zeta|\tilde{u}_1|^2}du\nnm\\
&\leq C\|f\|^2\int_{\{|u|\leq R\}}\frac{1}{(\mathrm{Re}\lambda+\nu_0)^2+(\mathrm{Im}\lambda+|\zeta|\tilde{u}_1)^2}du\nnm\\
&=C\|f\|^2\int_0^{2\pi}d\theta\int_0^R \int_0^{\pi}\frac{r^2\sin\varphi}{(\mathrm{Re}\lambda+\nu_0)^2+(\mathrm{Im}\lambda+|\zeta|\frac{\mathbf{c}r}{\sqrt{\mathbf{c}^2+r^2}}\cos\varphi)^2}d \varphi  dr\nnm\\
&=C\|f\|^2\int_0^{2\pi}d\theta\int_0^R\int_{-1}^{1}\frac{r^2}{(\mathrm{Re}\lambda+\nu_0)^2+(\mathrm{Im}\lambda+|\zeta|\frac{\mathbf{c}rt}{\sqrt{\mathbf{c}^2+r^2}})^2}dt dr\nnm\\
&=C|\zeta|^{-1}\|f\|^2\int_0^{2\pi}d\theta\int_0^R\frac{1}{\mathbf{c}}r\sqrt{\mathbf{c}^2+r^2}\int_{\mathrm{Im}\lambda-\frac{|\zeta|r\mathbf{c}}{\sqrt{\mathbf{c}^2+r^2}}}^{\mathrm{Im}\lambda+\frac{|\zeta|r\mathbf{c}}{\sqrt{\mathbf{c}^2+r^2}}}\frac{1}{(\mathrm{Re}\lambda+\nu_0)^2+z^2}dz dr\nnm\\
&\leq C\delta^{-1}|\zeta|^{-1}\|f\|^2\int_0^Rr(1+r)dr\leq CR^3 \delta^{-1}|\zeta|^{-1}\|f\|^2.
\ema

For $I_2$, we have by Lemma \ref{rbdl-svku1alfa} that
\bma
I_2&\le C\delta^{-2}\int_{\R^3}\(\int_{\{|u|> R\}}|k(v,u)| du\)\(\int_{\{|u|>R\}}|k(v,u)| |f(u)|^2 du\)dv\nnm\\
&\le C\delta^{-2}\int_{\{|u|> R\}} |f(u)|^2\(\int_{\R^3}|k(v,u)|(1+|v|)^{-1} dv\)du\nnm\\
&\leq C\delta^{-2}R^{-2}\|f\|^2.
\ema
By taking $R=\(\frac{|\zeta|}{\delta}\)^{\frac{1}{5}}$, we can obtain \eqref{rbesplamdaa-im}.

Now, we will prove \eqref{rbesplamdaa-ze}. If $|\zeta|<r_0$, $u\leq R$ and $|\mathrm{Im}\lambda|\geq2\mathbf{c}r_0$, we have
\bq
|\mathrm{Im}\lambda+\tilde{u}\cdot\zeta| \geq|\mathrm{Im}\lambda|-|\tilde{u}||\zeta|\geq|\mathrm{Im}\lambda|-\frac{\mathbf{c}Rr_0}{\sqrt{ \mathbf{c}^2 +R^2}}\geq\frac{|\mathrm{Im}\lambda|}{2},\quad  R>\mathbf{c}.
\eq
Thus,
\bma
I_1&\leq C\|f\|^2\int_{\{|u|\leq R\}}|\lambda+\nu(u)+\mathrm{i}\zeta\cdot\tilde{u}|^{-2}du\nnm\\
&\leq C(\delta^2+|\mathrm{Im}\lambda|^2)^{-1}R^3\|f\|^2\nnm\\
&\leq C\delta^{-1}|\mathrm{Im}\lambda|^{-1}R^3\|f\|^2.
\ema
By taking $ R=\(\frac{|\mathrm{Im}\lambda|}{\delta}\)^{\frac{1}{5}}$, we can verify \eqref{rbesplamdaa-ze}. The proof of the lemma is completed.
\end{proof}

\begin{lem}\label{rbesp2-si2}
{\rm (1)} For any $r_1>0$, there exists $\alpha=\alpha(r_1)>0$ such that when $|\zeta|>r_1$,
\bq
\sigma(B(\zeta))\subset\{\lambda\in \mathbb{C}\,|\,\mathrm{Re}\lambda<-\alpha\}.
\eq

{\rm (2)} For any $\delta>0$ and all $\zeta\in\R^3$, there exists $y_1=y_1(\delta)>0$ such that
\bq
\rho(B(\zeta))\supset\{\lambda\in\mathbb{C}\,|\,\mathrm{Re}\lambda\geq-\nu_0+\delta,\,|\mathrm{Im}\lambda|\geq y_1\}\cup\{\lambda\in\mathbb{C}\,|\,\mathrm{Re}\lambda>0\}.
\eq
\end{lem}
\begin{proof}
By using Lemma \ref{rbesplamdaa} and the same argument as Lemma
2.6 in \cite{ZHAO-1}, we can prove the lemma. The detail of the proof is omitted for brevity.
\end{proof}

\begin{lem}[\cite{ZHAO-1}]\label{3dmvpbsp10}
{\rm (1)} There exists a constant $r_0>0$ such that for $|\zeta|\leq r_0$,
$$\sigma(B(\zeta))\cap\Big\{\lambda\in\mathbb{C}\,|\,\mathrm{Re}\lambda\geq-\frac{\nu_0}{2}\Big\}=\{\lambda_j(|\zeta|),~j=-1,0,1,2,3\}.$$
The eigenvalues $\lambda_j(|\zeta|)$, $j=-1,0,1,2,3$ are $C^{\infty}$ functions of $|\zeta|$, and satisfy the following expansions:
\bq\label{rbdl-sp-lamze}
\lambda_j(|\zeta|)=-\mathrm{i} \mu_j|\zeta|-|\zeta|^2A_j+O(|\zeta|^3),\quad|\zeta|\leq r_0,
\eq
where
\bq\label{sp4}
\left\{\bln
&\mu_{\pm1}=\mp \sqrt{a^2+b^2},\quad \mu_k=0,\quad k=0,2,3,\\
&A_{j}=-\(L^{-1}P_1(\tilde{v}\cdot\omega)E_j(\zeta),(\tilde{v}\cdot\omega)E_j(\zeta)\),\quad j=-1,0,1,2,3,\\
&E_{\pm1}(\zeta)=\frac{b}{\sqrt{2a^2+2b^2}}\chi_0\mp \frac{1}{\sqrt{2}} \frac{\zeta\cdot\chi'}{|\zeta|}+\frac{a}{\sqrt{2a^2+2b^2}}\chi_4,\\
&E_0(\zeta)=-\frac{a}{\sqrt{a^2+b^2}}\chi_0+\frac{b}{\sqrt{a^2+b^2}}\chi_4,\\
&E_l(\zeta)=W^l\cdot\chi',\quad l=2,3,\quad \chi'=(\chi_1,\chi_2,\chi_3),
\eln\right.
\eq
with $\omega=\zeta/|\zeta|$, and $W^l~(l=2,3)$ are orthonormal vectors satisfying $W^l\cdot\omega=0$.

{\rm (2)} The eigenfunctions $\psi_j(\zeta)=\psi_j(\zeta,v)$, $j=-1,0,1,2,3$ satisfy
\bq\label{egf1}
\left\{\bln
&\big(\psi_j(\zeta),\overline{\psi_k}(\zeta)\big)=\delta_{jk},\quad -1\leq j,k\leq3,\\
&\psi_j(\zeta)=P_0\psi_j(\zeta)+P_1\psi_j(\zeta),\\
&P_0\psi_j(\zeta)=E_j(\zeta)+O(|\zeta|),\\
&P_1\psi_j(\zeta)=\mathrm{i} L^{-1}P_1(\tilde{v}\cdot\zeta) E_j(\zeta)+O(|\zeta|^2).
\eln\right.
\eq

{\rm (3)} The semigroup $S(t,\zeta)=e^{\frac{tB(\zeta)}{\epsilon^2}}$ with $\zeta\in\mathbb{R}^3$ has the following decomposition:
\bq\label{se1}
S(t,\zeta)f=S_1(t,\zeta)f+S_2(t,\zeta)f,\quad f\in L^2(\mathbb{R}^3), \quad t>0,
\eq
where
\bq\label{se2}
S_1(t,\zeta)f=\sum^{3}_{j=-1}e^{\frac{t\lambda_j(|\zeta|)}{\epsilon^2}}\big(f,\overline{\psi_j(\zeta)}\big)\psi_j(\zeta),\quad |\zeta|\leq r_0,
\eq
with $\big(\lambda_j(|\zeta|),\psi_j(\zeta)\big)$ being the eigenvalue and eigenfunction of the operator $B(\zeta)$ for $|\zeta|\leq r_0$, and
$S_2(t,\zeta)f=S(t,\zeta)f-S_1(t,\zeta)f$
satisfies for two constants $\tau_1>0$ and $C>0$ independent of $\xi$, $\mathbf{c}$ and $\epsilon$ that
\bq\label{se3}
\|S_2(t,\zeta)f\|\leq Ce^{-\frac{\tau_1 t}{\epsilon^2}}\|f\|,\quad t>0.
\eq
\end{lem}

\begin{remark}From Lemma \ref{rbdl-aj-est}, it holds that
\be \mu_{\pm1}\sim \mp \frac{\mathbf{c}^2}{\sqrt{k_0T}}, \quad A_j\sim 1,\quad j=0,2,3,\quad A_{\pm1}\sim \mathbf{c}^4.  \label{Aj}\ee
\end{remark}

%
%
%
%

%
%
%
%

\section{Fluid approximation of semigroup}\setcounter{equation}{0}
\label{sect3}
In this section, we give the first and second order fluid approximations of the semigroup $e^{\frac{tB_{\epsilon}}{\epsilon^2}}$, which will be used to prove the convergence and establish the convergence rate of the solution to the relativistic Boltzmann equation \eqref{rbe2}--\eqref{Pm3} towards the solution to the iNS system \eqref{rNS-1}.

For any $f_0\in L^2$, set
\bq
e^{\frac{tB_{\epsilon}}{\epsilon^2}}f_0=\(\mathcal{F}^{-1}e^{\frac{tB(\epsilon\xi)}{\epsilon^2}}\mathcal{F}\)f_0.
\eq
it hold that
$$
\|e^{\frac{tB_{\epsilon}}{\epsilon^2}}f_0\|^2_{H^k}=\int_{\mathbb{R}^3}(1+|\xi|^2)^k\|e^{\frac{tB(\epsilon\xi)}{\epsilon^2}}\hat{f}_0\|^2d\xi\leq \int_{\mathbb{R}^3}(1+|\xi|^2)^k\|\hat{f}_0\|^2d\xi=\|f_0\|^2_{H^k}.
$$
This means that the operator $\frac{B_{\epsilon}}{\epsilon^2}$ generates a strongly continuous contraction semigroup $e^{\frac{tB_{\epsilon}}{\epsilon^2}}$ in $H^k$.  Therefore, $f(t,x,v)=e^{\frac{tB_{\epsilon}}{\epsilon^2}}f_0$ is a global solution to the linearized relativistic Boltzmann equation \eqref{Bm1} for any $f_0\in H^k$.

\subsection{Semigroup of the linear iNS system}
In this subsection, we are going to study the solution to the linear iNS system. Consider the following linear iNS system for $(n,m,q) $:
\bma
&\divx m=0,\quad bn+aq=0,\label{lrNS-1}\\
&\partial_tm-A_2\Delta_xm+\nabla_xp=H_1,\label{lrNS-2}\\
&\partial_tq-A_0\Delta_xq= \frac{b}{\sqrt{a^2+b^2}}H_2,\label{lrNS-3}
\ema
where $H_1= (H^1_1,H^2_1,H^3_1 )$ and $H_2$ are given functions, $p$ is the pressure satisfying $p=\Delta_x^{-1}\div_xH_1$, and the initial data $(n,m,q)(0)$ satisfies \eqref{rbdl-innmq}. For any $U_0=U_0(x,v)\in N_0$, we define
\bq\label{lns4j1}
V(t,\xi)\hat{U}_0=\sum_{j=0,2,3}e^{-A_j|\xi|^2t}\big(\hat{U}_0,E_j(\xi)\big)E_j(\xi),
\eq
where $A_j>0$ and $E_j(\xi)\in N_0$, $j=0,2,3$ are defined by \eqref{sp4}. Note that $A_0=\kappa_0$ and $A_2=A_3=\kappa_1$.

Set
\bq\label{lns4j1-1}
V(t)U_0=\(\mathcal{F}^{-1}V(t,\xi)\mathcal{F}\)U_0.
\eq
Then, we can represent the solution to the iNS system \eqref{lrNS-1}--\eqref{lrNS-3} by the semigroup $V(t)$ as follows.
\begin{lem}\label{3dmvpbfas1}
For any $f_0\in L^2$ and $H \in  L^1_t(L^2)$, define
$$
U(t,x,v)=V(t)P_0f_0+\int^t_0V(t-s)H(s)ds.
$$
Let $(n,m,q)=((U,\chi_0),(U,\chi'),(U,\chi_4))$. Then $(n,m,q)(t,x)\in L^{\infty}_t(L^2_x)$ is an unique global solution to the linear iNS system \eqref{lrNS-1}--\eqref{lrNS-3} with the initial data $(n,m,q)(0)$ given by \eqref{rbdl-innmq} and the inhomogeneous terms $H_1,H_2$ given by
$$ (H_1)_\bot=(H, \chi')_\bot , \quad H_2= (H,E_0).$$
Here $Y_\bot=Y-\Delta_x^{-1}\Tdx\div_x Y$ with $Y=(Y_1,Y_2,Y_3)(x)$.
\end{lem}
\begin{proof}
By taking Fourier transform to \eqref{lrNS-1}--\eqref{lrNS-3}, we have
\bma
&\mathrm{i}\xi\cdot\hat{m}=0,\quad b\hat{n}+a\hat{q}=0,\label{lrnsf1}\\
&\partial_t\hat{m}+A_2|\xi|^2\hat{m} =\mathbb{O}_1\hat{H}_1,\label{lrnsf2}\\
&\partial_t\hat{q}+A_0|\xi|^2\hat{q}= \frac{b}{\sqrt{a^2+b^2}}\hat{H}_2,\label{lrnsf3}
\ema
where $\mathbb{O}_1=|\xi|^{-2}\xi\times \xi\times $, and the initial data $(\hat{n},\hat{m},\hat{q})(0)$ satisfies
\bq\label{lnsi1}
\hat{m}(0)=\mathbb{O}_1 (P_0\hat{f}_0,\chi' ),\quad b\hat{q}(0)-a\hat{n}(0) =(P_0\hat{f}_0,b\chi_4-a\chi_0 ).
\eq

By \eqref{lrnsf3} and using the relation $q=\frac{b}{a^2+b^2}(bq-an)$, we obtain
\be
\partial_t\(\frac{b\hat{q}-a\hat{n}}{\sqrt{a^2+b^2}}\)+A_0|\xi|^2\frac{b\hat{q}-a\hat{n}}{\sqrt{a^2+b^2}}= \hat{H}_2.\label{lnsj2}
\ee
It follows from \eqref{lnsi1}--\eqref{lnsj2} that
\bma
\frac{b\hat{q}-a\hat{n}}{\sqrt{a^2+b^2}}&=e^{-A_0|\xi|^2t}\frac{b\hat{q}(0)-a\hat{n}(0)}{\sqrt{a^2+b^2}}+\int^t_0e^{-A_0|\xi|^2(t-s)} \hat{H}_2(s)ds\nnm\\
&=e^{-A_0|\xi|^2t}\big(P_0\hat{f}_0,E_{0}(\xi)\big)\big(E_0(\xi),E_0(\xi)\big)\nnm\\
&\quad+\int^t_0e^{-A_0|\xi|^2(t-s)}\big(\hat{H}(s),E_0(\xi)\big)\big(E_0(\xi),E_0(\xi)\big)ds.\label{lnsj3}
\ema

By \eqref{lrnsf1}--\eqref{lrnsf2}, we have
\bma\label{lnsj5}
\hat{m}(t,\xi)&=e^{-A_2|\xi|^2t} \hat{m}(0)+\int^t_0e^{-A_2|\xi|^2(t-s)}\mathbb{O}_1\hat{H}_1(s)ds\nnm\\
&=\sum_{j=2,3}e^{-A_j|\xi|^2t}\big(P_0\hat{f}_0,E_{j}(\xi)\big)\big(E_j(\xi),\chi'\big)\nnm\\
&\quad+\sum_{j=2,3}\int^t_0e^{-A_j|\xi|^2(t-s)}\big(\hat{H}(s),E_j(\xi)\big)\big(E_j(\xi),\chi'\big)ds,
\ema
where $\chi'=(\chi_1,\chi_2,\chi_3)$. Noting that $(E_0(\xi),\chi')=0$ and $(E_j(\xi),\chi_0)=(E_j(\xi),\chi_4)=0$, $j=2,3,$ we can prove the lemma by using \eqref{lnsj3}--\eqref{lnsj5}.
\end{proof}

\subsection{Fluid approximation of $e^{\frac{t B_{\epsilon}}{\epsilon^2}}$}
We have the time decay rates of the semigroups $e^{\frac{tB_{\epsilon}}{\epsilon^2}}$ and $V(t)$ as follows.
\begin{lem}\label{3dmvpbfas2}
For any $\epsilon\in(0,1)$, $\alpha\in\mathbb{R}^3$ and any $f_0\in L^2$, we have
\bma
&\|P_0\partial^{\alpha}_xe^{\frac{tB_{\epsilon}}{\epsilon^2}}f_0\|_{L^2}\leq C(1+t)^{-\frac{3+2m}{4}}(\|\partial^{\alpha}_xf_0\|_{L^2}+\|\partial^{\alpha'}_xf_0\|_{L^{2,1}}),\label{ns321}\\
&\|P_1\partial^{\alpha}_xe^{\frac{tB_{\epsilon}}{\epsilon^2}}f_0\|_{L^2}\leq
C\(\epsilon(1+t)^{-\frac{5+2m}{4}}+e^{-\frac{\tau_1 t}{\epsilon^2}}\)(\|\partial^{\alpha}_xf_0\|_{H^1}+\|\partial^{\alpha'}_xf_0\|_{L^{2,1}}),\label{ns322}
\ema
where $\alpha'\leq\alpha$, $m=|\alpha-\alpha'|$, $\tau_1>0$ and $C>0$ are two constants independent of $\epsilon$ and $\mathbf{c}$.

Moreover, if $P_0f_0=0$, then
\bma
&\|P_0\partial^{\alpha}_xe^{\frac{tB_{\epsilon}}{\epsilon^2}}f_0\|_{L^2}\leq C\(\epsilon(1+t)^{-\frac{5+2m}{4}}+e^{-\frac{\tau_1 t}{\epsilon^2}}\)(\|\partial^{\alpha}_xf_0\|_{H^1}+\|\partial^{\alpha'}_xf_0\|_{L^{2,1}}),\label{ns325}\\
&\|P_1\partial^{\alpha}_xe^{\frac{tB_{\epsilon}}{\epsilon^2}}f_0\|_{L^2}\leq
C\(\epsilon^2(1+t)^{-\frac{7+2m}{4}}+e^{-\frac{\tau_1 t}{\epsilon^2}}\)(\|\partial^{\alpha}_xf_0\|_{H^2}+\|\partial^{\alpha'}_xf_0\|_{L^{2,1}}).\label{ns326}
\ema

In particular, it holds that for $P_0f_0=0$,
\be
\|\partial^{\alpha}_xe^{\frac{tB_{\epsilon}}{\epsilon^2}}f_0\|_{L^2}\leq C\(\epsilon(1+t)^{-\frac{5+2m}{4}}+\epsilon t^{-\frac{1}{2}}e^{-\frac{\tau_2t}{2}}+e^{-\frac{\tau_1 t}{\epsilon^2}}\)(\|\partial^{\alpha}_xf_0\|_{L^2}+\|\partial^{\alpha'}_xf_0\|_{L^{2,1}}).\label{ns325-re}
\ee
\end{lem}
\begin{proof}
By \eqref{se1}, we have that for $j=0,1$,
\bma
\|P_j\partial^{\alpha}_xe^{\frac{tB_{\epsilon}}{\epsilon^2}}f_0\|^2_{L^2}&=\int_{\mathbb{R}^3}\|P_j\xi^{\alpha}e^{\frac{tB(\epsilon\xi)}{\epsilon^2}}\hat{f}_0\|^2d\xi\nonumber\\
&\leq\int_{\{|\xi|\leq\frac{r_0}{\epsilon}\}}\|\xi^{\alpha}P_jS_1(t,\epsilon\xi)\hat{f}_0\|^2d\xi+\int_{\mathbb{R}^3}\|\xi^{\alpha}S_2(t,\epsilon\xi)\hat{f}_0\|^2d\xi,\label{lnsp1}
\ema
where
\be
\int_{\mathbb{R}^3}\|\xi^{\alpha}S_2(t,\epsilon\xi)\hat{f}_0\|^2d\xi \leq Ce^{-\frac{2\tau_1 t}{\epsilon^2}}\int_{\mathbb{R}^3}(\xi^{\alpha})^2\|\hat{f}_0\|^2d\xi\leq Ce^{-\frac{2\tau_1 t}{\epsilon^2}}\|\partial^{\alpha}_xf_0\|^2_{L^2}.\label{lnsp3}
\ee
By \eqref{se2}, we have for $\epsilon|\xi|\leq r_0$,
\be\label{s111}
S_1(t,\epsilon\xi)\hat{f}_0=\sum^{3}_{j=-1}e^{\frac{-\mathrm{i}|\xi|\mu_j t}{\epsilon}-|\xi|^2A_j t+O(\epsilon|\xi|^3)t}\big(\hat{f}_0,  \overline{\psi_j(\eps\xi)} \big) \psi_j(\eps\xi).
\ee
Thus,
\bma
\int_{\{|\xi|\leq\frac{r_0}{\epsilon}\}}\|\xi^{\alpha}P_0S_1(t,\epsilon\xi)\hat{f}_0\|^2d\xi
&\leq C\int_{\{|\xi|\leq\frac{r_0}{\epsilon}\}}e^{-2\tau_2|\xi|^2t}(\xi^{\alpha})^2\|\hat{f}_0\|^2d\xi\nnm\\
&\leq C\sup_{\xi\le1}\|\xi^{\alpha'}\hat{f}_0\|^2\int_{\{|\xi|\leq1\}}e^{-2\tau_2|\xi|^2t}|\xi|^{2|\alpha-\alpha'|}d\xi\nnm\\
&\quad+Ce^{-2\tau_2t}\int_{\{|\xi|\geq1\}}(\xi^{\alpha})^2\|\hat{f}_0\|^2d\xi\nnm\\
&\leq C(1+t)^{-\frac{3+2m}{2}}\big(\|\partial^{\alpha}_xf_0\|^2_{L^2}+\|\partial^{\alpha'}_xf_0\|^2_{L^{2,1}}\big),\label{lnsp4}
\ema
where $m=|\alpha-\alpha'|$ and $\tau_2=\min\{\frac{A_j}{2},j=-1,0,1,2,3\}$.

By \eqref{s111}, we can obtain
\bma
\int_{\{|\xi|\leq\frac{r_0}{\epsilon}\}}\|\xi^{\alpha}P_1S_1(t,\epsilon\xi)\hat{f}_0\|^2d\xi&\leq C\epsilon^{2}\int_{\{|\xi|\leq\frac{r_0}{\epsilon}\}}e^{-2\tau_2|\xi|^2t}|\xi|^2(\xi^{\alpha})^2\|\hat{f}_0\|^2d\xi\nnm\\
&\leq C\epsilon^2(1+t)^{-\frac{5+2m}{2}}\big(\|\partial^{\alpha}_xf_0\|^2_{H^1}+\|\partial^{\alpha'}_xf_0\|^2_{L^{2,1}}\big),\label{lnsp5}
\ema
where we have used the fact that $P_1\psi_j(\eps\xi)=O(\eps|\xi|)$. Combining \eqref{lnsp1}--\eqref{lnsp5}, we prove \eqref{ns321}--\eqref{ns322}.

If $P_0f_0=0$, then it holds that for $\epsilon|\xi|\leq r_0$,
\be\label{s222}
S_1(t,\epsilon\xi)\hat{f}_0=\mathrm{i}\sum^{3}_{j=-1}e^{\frac{-\mathrm{i}|\xi|\mu_j t}{\epsilon}-|\xi|^2A_j t+O(\epsilon|\xi|^3)t} \big(\hat{f}_0, P_1\overline{\psi_j(\eps\xi)} \big) \psi_j(\eps\xi) .
\ee
Thus, 
\bma
\int_{\{|\xi|\leq\frac{r_0}{\epsilon}\}}\|\xi^{\alpha}P_0S_1(t,\epsilon\xi)\hat{f}_0\|^2d\xi&\leq C\epsilon^{2}\int_{\{|\xi|\leq\frac{r_0}{\epsilon}\}}e^{-2\tau_2|\xi|^2t}|\xi|^2(\xi^{\alpha})^2\|\hat{f}_0\|^2d\xi\nnm\\
&\leq C\epsilon^{2}(1+t)^{-\frac{5+2m}{2}}\big(\|\partial^{\alpha}_xf_0\|^2_{H^1}+\|\partial^{\alpha'}_xf_0\|^2_{L^{2,1}}\big),\label{lnsp8}\\
\int_{\{|\xi|\leq\frac{r_0}{\epsilon}\}}\|\xi^{\alpha}P_1S_1(t,\epsilon\xi)\hat{f}_0\|^2d\xi&\leq C\int_{\{|\xi|\leq\frac{r_0}{\epsilon}\}}\epsilon^{4}e^{-2\tau_2|\xi|^2t}|\xi|^4(\xi^{\alpha})^2\|\hat{f}_0\|^2d\xi\nnm\\
&\leq C\epsilon^4(1+t)^{-\frac{7+2m}{2}}\big(\|\partial^{\alpha}_xf_0\|^2_{H^2}+\|\partial^{\alpha'}_xf_0\|^2_{L^{2,1}}\big),\label{lnsp9}
\ema
where we have used the fact that $P_1\psi_j(\eps\xi)=O(\eps|\xi|)$. Combining \eqref{lnsp1}--\eqref{lnsp3} and \eqref{lnsp8}--\eqref{lnsp9}, we obtain \eqref{ns325}--\eqref{ns326}.

Finally, we prove \eqref{ns325-re}. Indeed, we can rewrite \eqref{lnsp8} as
\bmas
 \int_{\{|\xi|\leq\frac{r_0}{\epsilon}\}}\|\xi^{\alpha}S_1(t,\epsilon\xi)\hat{f}_0\|^2d\xi
&\leq C\epsilon^{2}\sup_{\xi\le 1}\|\xi^{\alpha'}\hat{f}_0\|^2\int_{\{|\xi|\leq1\}}e^{-2\tau_2|\xi|^2t}|\xi|^{2+2|\alpha-\alpha'|}d\xi\nnm\\
&\quad+C\epsilon^{2}e^{-\tau_2t}\sup_{|\xi|\geq1}(e^{-\tau_2|\xi|^2t}|\xi|^2)\int_{\{|\xi|\geq1\}}(\xi^{\alpha})^2\|\hat{f}_0\|^2d\xi\nnm\\
&\leq C\epsilon^{2}\((1+t)^{-\frac{5+2m}{2}}+t^{-1}e^{-\tau_2t}\)\big(\|\partial^{\alpha}_xf_0\|^2_{L^2}+\|\partial^{\alpha'}_xf_0\|^2_{L^{2,1}}\big).
\emas
This and \eqref{lnsp3} implies that \eqref{ns325-re} holds. The proof of the lemma is completed.
\end{proof}


\begin{lem}\label{3dmvpbfas3}
For any $\alpha\in\mathbb{N}^3$ and any $u_0\in N_0$, we have
\bma
\|\partial^{\alpha}_xV(t)u_0\|_{L^2}&\leq C(1+t)^{-\frac{3+2m}{4}}(\|\partial^{\alpha}_xu_0\|_{L^2}+\|\partial^{\alpha'}_xu_0\|_{L^{2,1}}),\\
\|\partial^{\alpha}_xV(t)u_0\|_{L^{\infty}}&\leq C(1+t)^{-\frac{3}{4}}\beta_m(t)(\|\partial^{\alpha'}_xu_0\|_{L^{\infty}}+\|\partial^{\alpha'}_xu_0\|_{L^2}),
\ema
where $\alpha'\leq\alpha$, $m=|\alpha-\alpha'|$, $C>0$ is a constant and $\beta_0(t)=\ln\(2+\frac{1}{t}\)$, $\beta_m(t)=t^{-\frac{m}{2}}$, $m\geq1$. 
\end{lem}
\begin{proof}
By \eqref{lns4j1}, we have
\bma
V(t,\xi)\hat{u}_0&=\frac{1}{a^2+b^2}e^{-A_0|\xi|^2t} (a\hat{U}_0-b\hat{U}_4)(a\chi_0-b\chi_4)\nnm\\
&\quad+e^{-A_2|\xi|^2t}\sum_{j=1}^3  \hat{U}_j\(\chi_j-\frac{(\xi\cdot\chi')}{|\xi|^2}\xi_j\),\label{lnsp10}
\ema
where $\hat{U}_j=(\hat{u}_0,\chi_j),~j=0,1,2,3,4$.
Then, we can estimate $\partial^{\alpha}_xV(t)u_0$ precisely by using an same argument as that of Lemma 3.5 in \cite{Li-1}. Hence, we omit the detail of the proof for brevity.
\end{proof}


We now prepare two Lemmas \ref{3dmvpbfas4}--\ref{3dmvpbfas5} that will be used to study the fluid dynamical approximation of the semigroup $e^{\frac{tB_{\epsilon}}{\epsilon^2}}$.
\begin{lem}[\cite{Yang-2}]\label{3dmvpbfas4}
For any functions $\phi(r)$ satisfying $\big|\phi^{(k)}(r)\big|\leq C(1+r)^{-2-k-\delta}$ for any $\delta>0$ and $k=0,1$, we have
$$
\bigg|\int_{\mathbb{R}^3}e^{\mathrm{i}x\cdot\xi}e^{\mathrm{i}\vartheta|\xi|}\alpha(\omega)\phi(|\xi|)d\xi\bigg|\leq C|\vartheta|^{-1},
$$
where $\alpha(\omega)$ is a smooth function for $\omega=\frac{\xi}{|\xi|}\in\S^2$ and $\vartheta\in\R$.
\end{lem}

\begin{lem}[\cite{Li-1}]\label{3dmvpbfas5}
For any $f_0\in N_0$, we have
\bq\label{f-sS-2}
\|S_2(t,\epsilon\xi)f_0\|\leq C\big(\epsilon|\xi|1_{\{\epsilon|\xi|\leq r_0\}}+1_{\{\epsilon|\xi|\geq r_0\}}\big)e^{-\frac{\tau_1 t}{\epsilon^2}}\|f_0\|.
\eq
\end{lem}


Now we are going to estimate the first and second order expansions of the semigroup $e^{\frac{tB_{\epsilon}}{\epsilon^2}}$.
\begin{lem}\label{3dmvpbfas6}
For any $\epsilon\in (0,1)$ and any $f_0\in L^2,$ we have
\be
\big\|e^{\frac{tB_{\epsilon}}{\epsilon^2}}f_0-V(t)P_0f_0\big\|_{L^{\infty}} \leq C\bigg(\epsilon(1+t)^{-2}+\(1+\frac{t}{\epsilon}\)^{-1}\bigg)\(\|f_0\|_{H^{3}}+\|f_0\|_{W^{3,1}}\),\label{lsp2j1}
\ee
where $V(t)$ is given in \eqref{lns4j1}, and $C>0$ is a constant independent of $\epsilon$ and $\mathbf{c}$. Moreover, if $f_0$ satisfies \eqref{rbf0}, then
\bq\label{lsp2j2}
\big\|e^{\frac{tB_{\epsilon}}{\epsilon^2}}f_0-V(t)P_0f_0\big\|_{L^{\infty}}\leq C\epsilon(1+t)^{-2}\(\|f_0\|_{H^{3}}+\|f_0\|_{W^{3,1}}\).
\eq
\end{lem}
\begin{proof}
First, we prove \eqref{lsp2j1} as follows. By Lemma \ref{3dmvpbsp10}, we have
\bma\label{F8-1-1}
\big\|e^{\frac{tB_{\epsilon}}{\epsilon^2}}f_0-V(t)P_0f_0\big\|
&=\bigg\|\int_{\mathbb{R}^3}e^{\mathrm{i}x\cdot\xi}\(e^{\frac{tB(\epsilon\xi)}{\epsilon^2}}\hat{f}_0-V(t,\xi)P_0\hat{f}_0\)d\xi\bigg\|\nonumber\\
&\leq \bigg\|\int_{\{|\xi|\leq\frac{r_0}{\epsilon}\}}e^{\mathrm{i}x\cdot\xi}(S_1(t,\epsilon\xi)\hat{f}_0-V(t,\xi)P_0\hat{f}_0)d\xi\bigg\|\nonumber\\
&\quad+\int_{\{|\xi|\geq\frac{r_0}{\epsilon}\}}\|V(t,\xi)P_0\hat{f}_0\|d\xi+\int_{\mathbb{R}^3}\|S_2(t,\epsilon\xi)\hat{f}_0\|d\xi\nonumber\\
&=:I_1+I_2+I_3.
\ema
We estimates $I_j$, $j=1,2,3$ as follows. By Lemma \ref{3dmvpbsp10}, we have
\bq
S_1(t,\epsilon\xi)\hat{f}_0=\sum^3_{j=-1}e^{\frac{-\mathrm{i}|\xi|\mu_j t}{\epsilon}-A_j|\xi|^2t+O(\epsilon|\xi|^3)t}\(\big(P_0\hat{f}_0,E_j(\xi)\big)E _j(\xi)+O(\epsilon|\xi|)\),
\eq
which leads to
\bma\label{F6}
I_1&\leq \sum^3_{j=-1}\int_{\{|\xi|\leq\frac{r_0}{\epsilon}\}}\Big\|e^{\frac{-\mathrm{i}|\xi|\mu_jt}{\epsilon}-A_j|\xi|^2t+O(\epsilon|\xi|^3)t}\(\big(P_0\hat{f}_0,E_j(\xi)\big)E _j(\xi)+O(\epsilon|\xi|)\)\nonumber\\
&\qquad -e^{\frac{-\mathrm{i}|\xi|\mu_jt}{\epsilon}-A_j|\xi|^2t}\big(P_0\hat{f}_0,E_{j}(\xi)\big)E_{j}(\xi)\Big\|d\xi\nonumber\\
&\quad+\sum_{j=-1,1}\bigg\|\int_{\{|\xi|\leq\frac{r_0}{\epsilon}\}}e^{\mathrm{i}x\cdot\xi}e^{\frac{-\mathrm{i}|\xi|\mu_jt}{\epsilon}-A_j|\xi|^2t}\big(P_0\hat{f}_0,E_j(\xi)\big)E_j(\xi)d\xi\bigg\|\nonumber\\
&=:I_{11}+I_{12}.
\ema

For $I_{11}$, it follows from \eqref{sp4} that
\bma\label{F7}
I_{11}&\leq C\epsilon\int_{\{|\xi|\leq\frac{r_0}{\epsilon}\}}e^{-2\tau_2|\xi|^2t}(|\xi|^3t\|\hat{f}_0\|+|\xi|\|\hat{f}_0\|)d\xi\nnm\\
&\leq C\epsilon\sup_{|\xi|\leq1}\|\hat{f}_0\|\int_{\{|\xi|\leq1\}}e^{-2\tau_2|\xi|^2t}(|\xi|+|\xi|^3t)d\xi\nnm\\
&\quad+  C\epsilon\bigg(\int_{\{|\xi|>1\}}e^{-4\tau_2|\xi|^2t}\frac{(1+|\xi|^2t)^2}{|\xi|^{4}} d\xi\bigg)^{\frac{1}{2}}\bigg(\int_{\{|\xi|>1\}}|\xi|^6\|\hat{f}_0\|^2d\xi\bigg)^{\frac{1}{2}}\nnm\\
&\leq C\epsilon(1+t)^{-2}\(\|f_0\|_{H^3}+\|f_0\|_{L^{2,1}}\),
\ema
where $\tau_2=\min\{\frac{A_j}{2},j=-1,0,1,2,3\}>0$ is a constant independent of $\mathbf{c}$ (cf. Lemma \ref{rbdl-aj-est}).

Then, we estimate $I_{12}$ as follows.
\bma
I_{12}&=\sum_{j=-1,1}\bigg\|\bigg(\int_{\R^3}-\int_{\{|\xi|>\frac{r_0}{\epsilon}\}}\bigg)e^{\mathrm{i}x\cdot\xi}e^{\frac{-\mathrm{i}|\xi|\mu_jt}{\epsilon}-A_j|\xi|^2t}\big(P_0\hat{f}_0,E_j(\xi)\big)E_j(\xi)d\xi\bigg\|\nnm\\
&\leq\sum_{j=-1,1}\bigg\|\int_{\R^3}e^{\mathrm{i}x\cdot\xi}\hat{H}_j(t,\xi)d\xi\bigg\|+\sum_{j=-1,1}\bigg\|\int_{\{|\xi|>\frac{r_0}{\epsilon}\}}e^{\mathrm{i}x\cdot\xi}\hat{H}_j(t,\xi)d\xi\bigg\|\nnm\\
&=:I_{12}^1+I_{12}^2,\label{F7-1}
\ema
where
$$
\hat{H}_j(t,\xi)=e^{\frac{-\mathrm{i}|\xi|\mu_jt}{\epsilon}-A_j|\xi|^2t}\big(P_0\hat{f}_0,E_j(\xi)\big)E_j(\xi),\quad j=\pm1.
$$
For $I_{12}^1$ and $I_{12}^2$, we have
\bma
I_{12}^1&\leq C\int_{\R^3}e^{-2\tau_2|\xi|^2t}\|P_0\hat{f}_0\|d\xi\nnm\\
&\leq C\(\int_{\R^3}e^{-2\tau_2|\xi|^2t}\frac{1}{(1+|\xi|^2)^2}d\xi\)^{\frac{1}{2}}\(\int_{\R^3}(1+|\xi|^2)^2\|\hat{f}_0\|^2d\xi\)^{\frac{1}{2}}\nnm\\
&\leq C\|f_0\|_{H^{2}},\label{F7-2}\\
I_{12}^2&\leq Ce^{-\frac{2\tau_2r_0^2t}{\epsilon^2}}\int_{\{|\xi|>\frac{r_0}{\epsilon}\}}\|P_0\hat{f}_0\|d\xi\nnm\\
&\leq Ce^{-\frac{2\tau_2r_0^2t}{\epsilon^2}}\bigg(\int_{\R^3}\frac{1}{(1+|\xi|^2)^2}d\xi\bigg)^{\frac{1}{2}}\bigg(\int_{\R^3}(1+|\xi|^2)^2\|\hat{f}_0\|^2d\xi\bigg)^{\frac{1}{2}}\nnm\\
&\leq Ce^{-\frac{2\tau_2r_0^2t}{\epsilon^2}}\|f_0\|_{H^2}.\label{F7-3}
\ema

By \eqref{sp4}, we have
\bma
\big(P_0\hat{f}_0,E_j \big)E_{j}&=\frac{1}{2a^2+2b^2}\big(b\hat{n}_0-j \sqrt{a^2+b^2}\hat{m}_0\cdot\omega+a\hat{q}_0\big)\big(b\chi_0+a\chi_4\big)\nnm\\
&\quad-j\frac{1}{2\sqrt{a^2+b^2}}\big( b \hat{n}_0-j\sqrt{a^2+b^2}\hat{m}_0\cdot\omega+a\hat{q}_0\big)\omega\cdot\chi',\label{F7-4}
\ema
where $j=\pm1$, and
$$
\big(\hat{n}_0,\hat{m}_0,\hat{q}_0\big)=\big((\hat{f}_0,\chi_0),(\hat{f}_0,\chi'),(\hat{f}_0,\chi_4)\big).
$$
Thus,
\bma
&\big(\hat{H}_j,b\chi_0+a\chi_4\big) =\frac{b}{2}\hat{G}_{j}^{1}\hat{\mathcal{C}}_j\hat{F}_0+\frac{a}{2}\hat{G}_{j}^{1}\hat{\mathcal{C}}_j\hat{F}_2-j \frac{\sqrt{a^2+b^2}}{2} \hat{G}_{j}^{2}\hat{\mathcal{C}}_j\cdot\hat{F}_1,\label{F7-5}\\
&\big(\hat{H}_j,\chi'\big)=-j\frac{b}{2\sqrt{a^2+b^2}}\hat{G}_{j}^{2}\hat{\mathcal{C}}_j\hat{F}_0-j\frac{a}{2\sqrt{a^2+b^2}}\hat{G}_{j}^{2}\hat{\mathcal{C}}_j\hat{F}_2+\frac{1}{2}\hat{G}_{j}^{3}\hat{\mathcal{C}}_j\cdot\hat{F}_1,\label{F7-6}
\ema
where $j=-1,1$, and
\bq\label{F7-6-1}
\left\{\bln
&\hat{G}_{j}^{k}(t,\xi)=e^{\frac{-\mathrm{i}|\xi|\mu_jt}{\epsilon}}\alpha_k(\omega)(1+|\xi|)^{-3}, \quad  k=1,2,3,\\
&\hat{\mathcal{C}}_j(t,\xi)=e^{-A_j|\xi|^2t},\\
&(\hat{F}_0,\hat{F}_1,\hat{F}_2)=(1+|\xi|)^3(\hat{n}_0,\hat{m}_0,\hat{q}_0),\\
&\alpha_1(\omega)=1,\quad \alpha_2(\omega)=\omega,\quad \alpha_3(\omega)=\omega\otimes\omega.
\eln\right.
\eq
Then, we have
\bma
I_{12}^1&\leq C\big\| (H_j(t),b\chi_0+a\chi_4 )\big\|_{L^{\infty}_x}+C\big\|(H_j(t),\chi')\big\|_{L^{\infty}_x}\nnm\\
&\leq   C\(\|G_{j}^{1}\|_{L^{\infty}_x}+\|G_{j}^{2}\|_{L^{\infty}_x}+\|G_{j}^{3}\|_{L^{\infty}_x}\)\|\mathcal{C}_j\|_{L^1_x}\|(F_0,F_1,F_2)\|_{L^{1}_x}\nnm\\
&\leq C\(\frac{t}{\epsilon}\)^{-1}\|f_0\|_{W^{3,1}},\label{F7-7}
\ema
where we have used Lemma \ref{3dmvpbfas4} and
$$
\|\mathcal{C}_j\|_{L^1_x}=C\int_{\R^3}t^{-\frac{3}{2}}e^{-\frac{|x|^2}{4A_jt}}dx\leq C.
$$

By combining \eqref{F7-2}--\eqref{F7-3} and \eqref{F7-7}, we have
\bq
I_{12}\leq C\(1+\frac{t}{\epsilon}\)^{-1}\(\|f_0\|_{H^3}+\|f_0\|_{W^{3,1}}\),\label{F8}
\eq
which together with \eqref{F6}--\eqref{F7} implies that
\bq\label{F9}
I_1\leq C\bigg(\epsilon(1+t)^{-2}+\bigg(1+\frac{t}{\epsilon}\bigg)^{-1}\bigg)\(\|f_0\|_{H^3}+\|f_0\|_{W^{3,1}}\).
\eq

By Lemma \ref{3dmvpbsp10} and \eqref{lnsp10}, we have
\bma
I_2&\le C\eps \int_{\{|\xi|\geq\frac{r_0}{\epsilon}\}}e^{-2\tau_2|\xi|^2t}|\xi|\|P_0\hat{f}_0\|d\xi \nnm\\
&\leq C\eps e^{-\frac{2\tau_2r_0^2t}{\epsilon^2}}\bigg(\int_{\{|\xi|\geq\frac{r_0}{\epsilon}\}}\frac{1}{(1+|\xi|^2)^2}d\xi\bigg)^{\frac{1}{2}}\bigg(\int_{\{|\xi|\geq\frac{r_0}{\epsilon}\}}(1+|\xi|^2)^3 \| \hat{f}_0\|^2d\xi\bigg)^{\frac{1}{2}}\nonumber\\
&\leq C\eps e^{-\frac{2\tau_2r_0^2t}{\epsilon^2}}\|f_0\|_{H^3},\label{F10}\\
I_3&\leq Ce^{-\frac{\tau_1 t}{\epsilon^2}}\bigg(\int_{\R^3}\frac{1}{(1+|\xi|^2)^2}d\xi\bigg)^{\frac{1}{2}}\bigg(\int_{\R^3}(1+|\xi|^2)^2\|\hat{f}_0\|^2d\xi\bigg)^{\frac{1}{2}}\nnm\\
&\leq Ce^{-\frac{\tau_1 t}{\epsilon^2}}\|f_0\|_{H^2}.\label{F11}
\ema
Therefore, it follows from \eqref{F9}--\eqref{F11} that
\be\label{F12}
\big\|e^{\frac{tB_{\epsilon}}{\epsilon^2}}f_0-V(t)P_0f_0\big\|_{L^{\infty}}\leq C\bigg(\epsilon(1+t)^{-2}+\(1+\frac{t}{\epsilon}\)^{-1}\bigg)\(\|f_0\|_{H^3}+\|f_0\|_{W^{3,1}}\).
\ee
Thus, we proved \eqref{lsp2j1}.

Finally, we turn to show \eqref{lsp2j2}. If $f_0$ satisfies \eqref{rbf0}, then we have
$$
\big(P_0\hat{f}_0,E_j(\xi)\big)=0,\quad j=\pm1,
$$
which implies that $I_{12}=0$. The term $I_{11}$ satisfies \eqref{F7}. It follows from Lemma \ref{3dmvpbfas5}  that
\bma
I_3&\leq C\epsilon\int_{\{|\xi|\leq\frac{r_0}{\epsilon}\}}e^{-\frac{\tau_1 t}{\epsilon^2}}|\xi|\|\hat{f}_0\|d\xi +C\epsilon\int_{\{|\xi|\geq\frac{r_0}{\epsilon}\}}e^{-\frac{\tau_1 t}{\epsilon^2}}|\xi|\|\hat{f}_0\|d\xi\nnm\\
&\leq C\eps e^{-\frac{\tau_1 t}{\epsilon^2}}\|f_0\|_{H^3}. \label{I3}
\ema
Thus, by combining \eqref{F7}, \eqref{F10} and \eqref{I3},  we can obtain \eqref{lsp2j2}. The proof of the lemma is completed.
\end{proof}

\begin{remark}
From Lemma \ref{3dmvpbfas6}, we have
$$
\|e^{\frac{tB_{\epsilon}}{\epsilon^2}}P_0f_0-V(t)P_0f_0-u_{\epsilon}^{osc}(t)\|_{L^{\infty}}\leq C\epsilon(1+t)^{-2}\(\|f_0\|_{H^{3}}+\|f_0\|_{W^{3,1}}\),
$$
where $u_{\epsilon}^{osc}(t)=u_{\epsilon}^{osc}(t,x,v)$ is the high oscillation part of $e^{\frac{tB_{\epsilon}}{\epsilon^2}}f_0$ defined by \eqref{dlxg-uosc}.
\end{remark}

\begin{lem}\label{3dmvpbfas7}
For any $\epsilon\in(0,1)$ and any $f_0\in L^2$ satisfying $P_0f_0=0$, we have
\bma
&\quad\bigg\|\frac{1}{\epsilon}e^{\frac{tB_{\epsilon}}{\epsilon^2}}f_0-V(t)P_{0}(\tilde{v}\cdot\nabla_xL^{-1}f_0)\bigg\|_{L^{\infty}}\nnm\\
&\leq C\bigg(\epsilon(1+t)^{-\frac{5}{2}}+\(1+\frac{t}{\epsilon}\)^{-1}+\frac{1}{\epsilon}e^{-\frac{\tau_1 t}{\epsilon^2}}\bigg)\(\|f_0\|_{H^{4}}+\|f_0\|_{W^{4,1}}\).
\ema
where $V(t)$ is given in \eqref{lns4j1},  and $C>0$ is a constant independent of $\epsilon$.
\end{lem}

\begin{proof}
By Lemma \ref{3dmvpbsp10}, we obtain
\bma
&\quad\bigg\|\frac{1}{\epsilon}e^{\frac{tB_{\epsilon}}{\epsilon^2}}f_0-V(t)P_{0}(\tilde{v}\cdot\nabla_xL^{-1}f_0)\bigg\|\nonumber\\
&\leq \bigg\|\int_{\{|\xi|\leq\frac{r_0}{\epsilon}\}}e^{\mathrm{i}x\cdot\xi}\(\frac{1}{\epsilon}S_1(t,\epsilon\xi)\hat{f}_0-V(t,\xi)P_{0}(\mathrm{i}\tilde{v}\cdot\xi L^{-1}\hat{f}_0)\)d\xi\bigg\|\nonumber\\
&\quad+\int_{\{|\xi|\geq\frac{r_0}{\epsilon}\}}\big\|V(t,\xi)P_{0}(\mathrm{i}\tilde{v}\cdot\xi L^{-1}\hat{f}_0)\big\|d\xi+\int_{\R^3}\bigg\|\frac{1}{\epsilon}S_2(t,\epsilon\xi)\hat{f}_0\bigg\|d\xi\nonumber\\
&=:I_1+I_2+I_3.\label{fas7-0}
\ema
We estimate $I_j$, $j=1,2,3$ as follows. By Lemma \ref{3dmvpbsp10}, for any $f_0\in L^{2}$ satisfying $P_0f_0=0$, we have
\bq
S_1(t,\epsilon\xi)\hat{f}_0=\mathrm{i}\epsilon\sum^3_{j=-1}e^{\frac{-\mathrm{i}|\xi|\mu_jt}{\epsilon}-A_j|\xi|^2t+O(\epsilon|\xi|^3)t}\(\big(\tilde{v}\cdot\xi L^{-1}\hat{f}_0,E_j(\xi)\big)E_j(\xi)+O(\epsilon|\xi|^2)\),
\eq
which leads to
\bma
I_1&\leq \sum^3_{j=-1}\int_{\{|\xi|\leq \frac{r_0}{\epsilon}\}}\Big\|e^{\frac{-\mathrm{i}|\xi|\mu_jt}{\epsilon}-A_j|\xi|^2t+O(\epsilon|\xi|^3)t}\(\big(\tilde{v}\cdot\xi L^{-1}\hat{f}_0,E_j(\xi)\big)E_j(\xi)+O(\epsilon|\xi|^2)\)\nnm\\
&\quad-e^{\frac{-\mathrm{i}|\xi|\mu_jt}{\epsilon}-A_j|\xi|^2t}\big(\tilde{v}\cdot\xi L^{-1}\hat{f}_0,E_{j}(\xi)\big)E_{j}(\xi)\Big\|d\xi\nnm\\
&\quad+\sum_{j=-1,1}\bigg\|\int_{\{|\xi|\leq\frac{r_0}{\epsilon}\}}e^{\mathrm{i}x\cdot\xi}e^{\frac{-\mathrm{i}|\xi|\mu_jt}{\epsilon}-A_j|\xi|^2t}\big(\tilde{v}\cdot\xi L^{-1}\hat{f}_0,E_j(\xi)\big)E_j(\xi)d\xi\bigg\|\nnm\\
&=:I_{11}+I_{12}.\label{fas7-1}
\ema

For $I_{11}$, it holds that
\bma
I_{11}&\leq C\epsilon\int_{\{|\xi|\leq\frac{r_0}{\epsilon}\}}e^{-2\tau_2|\xi|^2t}\(|\xi|^3t\|P_{0}(\tilde{v}\cdot\xi L^{-1}\hat{f}_0)\|+|\xi|^2\|\hat{f}_0\|\)d\xi\nnm\\
&\leq C\epsilon(1+t)^{-\frac{5}{2}}\(\|f_0\|_{H^4}+\|f_0\|_{L^{2,1}}\).\label{fas7-1-1}
\ema
Then, we estimate $I_{12}$ as follows.
\bma
I_{12}&\leq\sum_{j=-1,1}\bigg\|\int_{\R^3}e^{\mathrm{i}x\cdot\xi}\hat{J}_j(t,\xi)d\xi\bigg\|+\sum_{j=-1,1}\bigg\|\int_{\{|\xi|>\frac{r_0}{\epsilon}\}}e^{\mathrm{i}x\cdot\xi}\hat{J}_j(t,\xi)d\xi\bigg\|\nnm\\
&=:I_{12}^1+I_{12}^2,\label{fas7-1-2}
\ema
where
$$
\hat{J}_j(t,\xi)=e^{\frac{-\mathrm{i}|\xi|\mu_jt}{\epsilon}-A_j|\xi|^2t}\big(\tilde{v}\cdot\xi L^{-1}\hat{f}_0,E_j(\xi)\big)E_j(\xi),\quad j=\pm1.
$$
For $I_{12}^j$, $j=1,2$, we obtain by  a similar argument as \eqref{F7-2}--\eqref{F7-3} that
\bma
I_{12}^1&\leq C\int_{\R^3}e^{-2\tau_2|\xi|^2t}|\xi|\|\hat{f}_0\|d\xi\leq C\|f_0\|_{H^{3}},\label{fas7-1-3}\\
I_{12}^2&\leq Ce^{-\frac{2\tau_2r_0^2t}{\epsilon^2}}\int_{\{|\xi|>\frac{r_0}{\epsilon}\}}|\xi|\|\hat{f}_0\|d\xi\leq Ce^{-\frac{2\tau_2r_0^2t}{\epsilon^2}}\|f_0\|_{H^3}.\label{fas7-1-4}
\ema
By \eqref{sp4}, we have
\bma
\big(\tilde{v}\cdot\xi L^{-1}\hat{f}_0,E_j\big)E_{j}&=\frac{1}{2a^2+2b^2}\big(b\hat{l}_1-j \sqrt{a^2+b^2}\hat{l}_2\cdot\omega+a \hat{l}_3\big)\(b\chi_0+a\chi_4\)\nnm\\
&\quad-j\frac{1}{2\sqrt{a^2+b^2}}\big(b \hat{l}_1-j\sqrt{a^2+b^2} \hat{l}_2\cdot\omega+a \hat{l}_3\big)\omega\cdot\chi',\label{fas7-1-5}
\ema
where
$$
(\hat{l}_1,\hat{l}_2,\hat{l}_3)= \big((\tilde{v}\cdot\xi L^{-1}\hat{f}_0,\chi_0),(\tilde{v}\cdot\xi L^{-1}\hat{f}_0,\chi'),(\tilde{v}\cdot\xi L^{-1}\hat{f}_0,\chi_4)\big).
$$
Thus,
\bma
&\big(\hat{J}_j,b\chi_0+a\chi_4\big)=\frac{b}{2}\hat{G}_{j}^{1}\hat{\mathcal{C}}_j\hat{F}_3+\frac{a}{2}\hat{G}_{j}^{1}\hat{\mathcal{C}}_j\hat{F}_5-j \frac{\sqrt{a^2+b^2}}{2}\hat{G}_{j}^{2}\hat{\mathcal{C}}_j\cdot\hat{F}_4,\label{fas7-1-6}\\
&\big(\hat{J}_j,v\chi_0\big) =-j \frac{b}{2\sqrt{a^2+b^2}} \hat{G}_{j}^{2}\hat{\mathcal{C}}_j\hat{F}_3-j\frac{a}{2\sqrt{a^2+b^2}}\hat{G}_{j}^{2}\hat{\mathcal{C}}_j\hat{F}_5+\frac{1}{2}\hat{G}_{j}^{3}\hat{\mathcal{C}}_j\cdot\hat{F}_4,\label{fas7-1-7}
\ema
where $\hat{G}_{j}^{kl}$, $\hat{\mathcal{C}}_j$ ($j=\pm1$) are given in \eqref{F7-6-1}, and
$
(\hat{F}_3,\hat{F}_4,\hat{F}_5)=(1+|\xi|)^3(\hat{l}_1,\hat{l}_2,\hat{l}_3).
$

Thus, by Lemma \ref{3dmvpbfas4}, we have
 \bma
I_{12}^1&\leq C\big\|(J_j(t),b\chi_0+a\chi_4)\big\|_{L^{\infty}_x}+C\big\|(J_j(t),\chi')\big\|_{L^{\infty}_x}\nnm\\
&\leq   C\(\|G_{j}^{1}\|_{L^{\infty}_x}+\|G_{j}^{2}\|_{L^{\infty}_x}+\|G_{j}^{3}\|_{L^{\infty}_x}\)\|\mathcal{C}_j\|_{L^1_x}\|(F_3,F_4,F_5)\|_{L^{1}_x}\nnm\\
&\leq C\(\frac{t}{\epsilon}\)^{-1}\|f_0\|_{W^{4,1}}.\label{fas7-1-8}
\ema
By combining \eqref{fas7-1-3}--\eqref{fas7-1-4} and \eqref{fas7-1-8}, we have
\bq
I_{12}\leq C\(1+\frac{t}{\epsilon}\)^{-1}\(\|f_0\|_{H^4}+\|f_0\|_{W^{4,1}}\),\label{fas7-2}
\eq
which together with \eqref{fas7-1}--\eqref{fas7-1-1} and \eqref{fas7-2} implies that
\bq\label{fas7-3}
I_1\leq C\bigg(\epsilon(1+t)^{-\frac{5}{2}}+\(1+\frac{t}{\epsilon}\)^{-1}\bigg)\(\|f_0\|_{H^4}+\|f_0\|_{W^{4,1}}\).
\eq

By Lemma \ref{3dmvpbsp10} and \eqref{lnsp10}, we have
\bma
I_2&\leq Ce^{-\frac{2\tau_2r_0^2t}{\epsilon^2}}\bigg(\int_{\{|\xi|\geq\frac{r_0}{\epsilon}\}}\frac{1}{(1+|\xi|^2)^2}d\xi\bigg)^{\frac{1}{2}}\bigg(\int_{\{|\xi|\geq\frac{r_0}{\epsilon}\}}(1+|\xi|^2)^2|\xi|^2\|\hat{f}_0\|^2d\xi\bigg)^{\frac{1}{2}}\nonumber\\
&\leq Ce^{-\frac{2\tau_2r_0^2t}{\epsilon^2}}\|f_0\|_{H^3},\label{fas7-4}\\
I_3&\leq Ce^{-\frac{\tau_1 t}{\epsilon^2}}\frac{1}{\epsilon}\bigg(\int_{\R^3}\frac{1}{(1+|\xi|^2)^2}d\xi\bigg)^{\frac{1}{2}}\bigg(\int_{\R^3}(1+|\xi|^2)^2\|\hat{f}_0\|^2d\xi\bigg)^{\frac{1}{2}}\nnm\\
&\leq C\frac{1}{\epsilon}e^{-\frac{\tau_1 t}{\epsilon^2}}\|f_0\|_{H^2}.\label{fas7-5}
\ema
Therefore, it follows from \eqref{fas7-0} and \eqref{fas7-3}--\eqref{fas7-5} that
\bmas
&\quad\bigg\|\frac{1}{\epsilon}e^{\frac{tB_{\epsilon}}{\epsilon^2}}f_0-V(t)P_{0}\big(\tilde{v}\cdot\nabla_xL^{-1}f_0\big)\bigg\|_{L^{\infty}}\nnm\\
&\leq C\bigg(\epsilon(1+t)^{-\frac{5}{2}}+\(1+\frac{t}{\epsilon}\)^{-1}+\frac{1}{\epsilon}e^{-\frac{\tau_1 t}{\epsilon^2}}\bigg)\(\|f_0\|_{H^{4}}+\|f_0\|_{W^{4,1}}\).
\emas
The proof of the lemma is completed.
\end{proof}

%
%
%
%

\section{Diffusion Limit and Optimal Convergence Rate}
\label{sect4}
\setcounter{equation}{0}

In this section, we study the diffusion limit of the solution to the nonlinear relativistic Boltzmann equation \eqref{rbe2}--\eqref{Pm3} based on the fluid approximations of the semigroup given in Section \ref{sect3}.

\subsection{Global existence}

\begin{lem}\label{rbdl-1}
For any $j,k=1,2,3$, we have
\bma
\Gamma\(v_j\chi_0,v_k\chi_0\)&=-\frac{1}{2}LP_1(v_jv_k\chi_0),\label{rbdl-1-4}\\
\Gamma\(v_0\chi_0,v_j\chi_0\)&=-\frac{1}{2}LP_1(v_jv_0\chi_0),\label{rbdl-1-5}\\
\Gamma\(v_0\chi_0,v_0\chi_0\)&=-\frac{1}{2}LP_1(v_0^2\chi_0).\label{rbdl-1-6}
\ema
\end{lem}
\begin{proof}
Since
$$
u_j+v_j=u'_j+v'_j,\quad j=1,2,3,\quad u_0+v_0=u'_0+v'_0,
$$
 we have
\bq \label{rbe3}
\left\{\bal
(u_k+v_k)(u_j+v_j)=(u'_k+v'_k)(u'_j+v'_j),\\
(u_j+v_j)(u_0+v_0)=(u'_j+v'_j)(u'_0+v'_0),\\
(u_0+v_0)(u_0+v_0)=(u'_0+v'_0)(u'_0+v'_0).
\ea\right.
\eq

By \eqref{rBQ} and \eqref{rbe3}, we have
\bmas
\Gamma\(v_j\chi_0,v_k\chi_0\)&=\frac{1}{2}\sqrt{M}\int_{\R^3}\int_{\mathbb{S}^2}v_M\sigma[v'_ju'_k+u'_jv'_k-v_ju_k-u_jv_k]M(u)dud\omega\\
&=-\frac{1}{2}\sqrt{M}\int_{\R^3}\int_{\mathbb{S}^2}v_M\sigma[u'_ju'_k+v'_jv'_k-u_ju_k-v_jv_k]M(u)dud\omega\\
&=-\frac{1}{2}LP_1(v_jv_k\chi_0),\\
\Gamma\(v_j\chi_0,v_0\chi_0\)&=\frac{1}{2}\sqrt{M}\int_{\R^3}\int_{\mathbb{S}^2}v_M\sigma[v'_ju'_0+u'_jv'_0-v_ju_0-u_jv_0]M(u)dud\omega\\
&=-\frac{1}{2}\sqrt{M}\int_{\R^3}\int_{\mathbb{S}^2}v_M\sigma[u'_ju'_0+v'_jv'_0-u_ju_0-v_jv_0]M(u)dud\omega\\
&=-\frac{1}{2}LP_1(v_jv_0\chi_0),\\
\Gamma\(v_0\chi_0,v_0\chi_0\)&=\frac{1}{2}\sqrt{M}\int_{\R^3}\int_{\mathbb{S}^2}v_M\sigma[v'_0u'_0+u'_0v'_0-v_0u_0-u_0v_0]M(u)dud\omega\\
&=-\frac{1}{2}\sqrt{M}\int_{\R^3}\int_{\mathbb{S}^2}v_M\sigma[u'_0u'_0+v'_0v'_0-u_0u_0-v_0v_0]M(u)dud\omega\\
&=-\frac{1}{2}LP_1(v_0^2\chi_0).
\emas
The proof of the lemma is completed.
\end{proof}

\begin{lem}\label{uuu}
Let  $(n,m,q)(t,x) $ to be the global solution to the iNS system \eqref{rNS-1}--\eqref{rbdl-innmq}. Then, $u(t,x,v)=n(t,x)\chi_0+m(t,x)\cdot\chi'+q(t,x)\chi_4$  can be represented by
\be
 u(t,x,v)=V(t)P_0f_0+\int^t_0V(t-s)P_{0}\big(\tilde{v}\cdot\nabla_xL^{-1}\Gamma(u,u)\big)ds. \label{uuu1}
\ee
\end{lem}
\begin{proof}
By Lemma \ref{rbdl-1}, we have
\bmas
  \big(\tilde{v}\cdot\nabla_xL^{-1}\Gamma(u,u),\chi_l\big)
&=-\frac{p_1}{2}\sum^3_{i,j,k=1} (P_1(v_i\chi_j),\tilde{v}_{k}\chi_l)\partial_{x_k}(m^im^j) \\
&\quad -\frac{p_2^2}{2} \partial_{x_l}(q^2)(P_1(v_0^2\chi_0),\tilde{v}_l\chi_l) ,\\
  \big(\tilde{v}\cdot\nabla_xL^{-1}\Gamma(u,u),E_0\big)
&=- p_1p_2\sum^3_{k=1}\(P_1(v_kv_0\chi_0),\tilde{v}_{k}E_0\)\partial_{x_k}(qm^k).
\emas
From  Lemma \ref{rbdl-ap-lp1}, it holds that
\bq\label{p2}
\left\{\bal
(P_1(v_i\chi_j),\tilde{v}_k\chi_l)=0,\quad (k,l)\neq(i,j)~\mathrm{or}~(j,i),\\
(P_1(v_i\chi_j),\tilde{v}_i\chi_j)=(P_1(v_i\chi_j),\tilde{v}_j\chi_i)=b_1,\quad 1\leq i\neq j\leq3,\\
(P_1(v_i\chi_i),\tilde{v}_i\chi_i)=3b_1-b_2,\\ (P_1(v_i\chi_i),\tilde{v}_j\chi_j)=b_1-b_2,\quad 1\leq i\neq j\leq3,\\
(P_1(v_0v_k\chi_0),\tilde{v}_kE_0)=\frac{\sqrt{a^2+b^2}}{bp_2}(P_1(v_1E_0),\tilde{v}_1E_0)=\frac{\sqrt{a^2+b^2}}{bp_2}b_3,\\
(P_1(v_0^2\chi_0),\tilde{v}_k\chi_k)=b_4,\quad k=1,2,3,
\ea\right.
\eq
where $b_i$, $i=1,2,3$ are given in Lemma \ref{rbdl-ap-lp1}. Thus, we can obtain
\bmas
 \big(\tilde{v}\cdot\nabla_xL^{-1}\Gamma(u,u),\chi'\big)&=-p_1b_1 \divx(m\otimes m )  -\frac{p_1}{2}(b_1-b_2) \Tdx|m|^2  -\frac{p_2^2}{2}b_4 \Tdx|q|^2 ,\\
 \big(\tilde{v}\cdot\nabla_xL^{-1}\Gamma(u,u),E_0\big)
&=-\frac{\sqrt{a^2+b^2}}{b}p_1b_3  \divx(qm ).
\emas
Note that $(\Tdx|m |^2)_\bot=0$ and $(\Tdx|q |^2)_\bot=0$. These together with Lemma \ref{3dmvpbfas1} implies \eqref{uuu1}.
\end{proof}

\begin{lem}\label{rbenletclm}
Let $ A_{\epsilon}=-\nu(v)-\epsilon\tilde{v}\cdot\nabla_x$. For any $d\geq0$, there exists a constant $C>0$ such that
\be
\|e^{\frac{tA_{\epsilon}}{\epsilon^2}}f_0 \|_{L^{\infty}_{v,d}(L^2_x)}\leq C e^{-\frac{\nu_0t}{\epsilon^2}}\|f_0\|_{L^{\infty}_{v,d}(L^2_x)}.\label{rbenletclm-1}
\ee

Moreover, if the function $F=F(t,x,v)$ satisfies
$$
\|\nu^{-1}F(t )\|_{L^{\infty}_{v,d}(L^2_x)}\leq C(1+t)^{-k},\quad \forall k>0,
$$
then
\be
\bigg\|\int_0^te^{\frac{(t-s)A_{\epsilon}}{\epsilon^2}}F(s)ds\bigg\|_{L^{\infty}_{v,d}(L^2_x)}\leq C\epsilon^2(1+t)^{-k}.\label{rbenletclm-2}
\ee
\end{lem}
\begin{proof}
Since $e^{\frac{tA_{\epsilon}}{\epsilon^2}}f_0$ can be represented as
$$
e^{\frac{tA_{\epsilon}}{\epsilon^2}}f_0(x,v)=e^{-\frac{\nu(v)t}{\epsilon^2}}f_0\(x-\frac{1}{\epsilon}\tilde{v}t,v\),
$$
 we can obtain \eqref{rbenletclm-1}.

Since
$$
\int_0^te^{\frac{(t-s)A_{\epsilon}}{\epsilon^2}}F(s)ds=\int_0^te^{-\frac{\nu(v)(t-s)}{\epsilon^2}}F\(x-\frac{1}{\epsilon}\tilde{v}(t-s),v,s\)ds,
$$
it follows that
\bmas
(1+|v|)^{d}\bigg\|\int_0^te^{\frac{(t-s)A_{\epsilon}}{\epsilon^2}}F(s)ds\bigg\|_{L^2_x}
&\leq \int_0^te^{-\frac{\nu(v)(t-s)}{\epsilon^2}}\nu(v)\|\nu^{-1}F(s)\|_{L^{\infty}_{v,d}(L^2_x)}ds\nnm\\
&\leq C\int_0^te^{-\frac{\nu(v)(t-s)}{\epsilon^2}}\nu(v)(1+s)^{-k}ds\nnm\\
&= C\epsilon^2\int_0^t(1+s)^{-k}de^{-\frac{\nu(v)(t-s)}{\epsilon^2}}\leq C\epsilon^2(1+t)^{-k}.
\emas
This gives \eqref{rbenletclm-2}. The proof of the lemma is completed.
\end{proof}

By virtue of Lemma \ref{3dmvpbfas2}, Lemma \ref{rbenletclm} and \eqref{nldh1}, we have following lemma.
\begin{lem}\label{rbedl5}
Let $N\geq2$. For any $\epsilon\in(0,1)$, there exists a small constant $\delta_0>0$ independent of $\mathbf{c}$ such that if $\|f_0\|_{L^{\infty}_{v,3}(H^N_x)}+\|f_0\|_{L^{2,1}}\leq\delta_0$, then the relativistic Boltzmann equation \eqref{rbe2}--\eqref{Pm3} admits a unique global solution $f_{\epsilon}=f_{\epsilon}(t,x,v)$ satisfies
\be
\|f_{\epsilon}(t)\|_{L^{\infty}_{v,3}(H^N_x)}\leq C\delta_0(1+t)^{-\frac{3}{4}},\label{rbedl5-1}
\ee
where   $C>0$ is a constant independent of $\mathbf{c}$ and $\epsilon$.

In particular, we have
\bq\label{rbedl5-3}
\|P_1f_{\epsilon}(t)\|_{H^{N-2}}\leq C\delta_0\(\epsilon(1+t)^{-\frac{5}{4}}+e^{ -\frac{\tau_1 t}{4\epsilon^2}}\),
\eq
where $\tau_1,C>0$ are two constants independent of $\mathbf{c}$ and $\epsilon$.
\end{lem}
\begin{proof}
We prove the existence of the solution $f$  by the contraction mapping theorem. Define the mapping $T$ as
\bq\label{rbethpr2-1}
(Tf_{\epsilon})(t)=e^{\frac{tB_{\epsilon}}{\epsilon^2}}f_0+\frac{1}{\epsilon}\int^t_0e^{\frac{(t-s)B_{\epsilon}}{\epsilon^2}}\Gamma(f_{\epsilon},f_{\epsilon})ds=:I_1+I_2.
\eq
For $\delta>0$, set the solution space $X$ as
$$
X=\Big\{f_{\epsilon}\in L^{\infty}((0,\infty),H^N)\,|\,\|f_{\epsilon}(t)\|_{X}=\sup_{ 0\leq s\leq t}\|f_{\epsilon}(s)\|_{L^{\infty}_{v,3}(H^N_x)}(1+s)^{\frac{3}{4}}\le \delta\Big\}.
$$

By Lemma \ref{dl-rbe-estgamfg}, 
 it holds for $0\le s\le t$ that
\bma
&\quad\|\nu^{-1}\Gamma(f_{\epsilon},f_{\epsilon})\|_{L^{\infty}_{v,3}(H^N_x)}+\|\Gamma(f_{\epsilon},f_{\epsilon})(s)\|_{L^{2,1}}\nnm\\
&\leq C\|f_{\epsilon}\|_{L^{\infty}_{v,3}(H^N_x)}\|f_{\epsilon}\|_{L^{\infty}_{v,3}(H^N_x)}+C\|f_{\epsilon}\|_{L^{2}}\|\nu f_{\epsilon}\|_{L^{2}}\nnm\\
&\leq C(1+s)^{-\frac{3}{2}}\|f_{\epsilon}(t)\|_X^2.\label{rbedl5-1-2}
\ema
 This and Lemma \ref{3dmvpbfas2} implies that
\bma
\|I_1\|_{H^N}&\leq C(1+t)^{-\frac{3}{4}}\(\|f_0\|_{H^N}+\|f_0\|_{L^{2,1}}\),\label{rbetf-h2-i1}\\
\|I_2\|_{H^N}&\leq C\int^{t}_0\((1+t-s)^{-\frac{5}{4}}+(t-s)^{-\frac{1}{2}}e^{-\frac{\tau_2(t-s)}{2}}+\frac{1}{\epsilon}e^{-\frac{\tau_1(t-s)}{\epsilon^2}}\)\|\Gamma(f_{\epsilon},f_{\epsilon})\|_{H^N\cap L^{2,1}}ds\nnm\\
&\leq C\|f_{\epsilon}(t)\|_X^2\int^{t}_0\((1+t-s)^{-\frac{5}{4}}+(t-s)^{-\frac{1}{2}}e^{-\frac{\tau_2(t-s)}{2}}+\frac{1}{\epsilon}e^{-\frac{\tau_1(t-s)}{\epsilon^2}}\)(1+s)^{-\frac{3}{2}}ds\nnm\\
&\leq C(1+t)^{-1}\|f_{\epsilon}(t)\|_X^2.\label{rbetf-h2-i2}
\ema
By taking summation of \eqref{rbetf-h2-i1}--\eqref{rbetf-h2-i2}, we obtain
\bq\label{rbetf-h2}
\|(Tf_{\epsilon})(t)\|_{H^N}\leq C(1+t)^{-\frac{3}{4}}\(\|f_0\|_{H^N}+\|f_0\|_{L^{2,1}}\)+C(1+t)^{-1}\|f_{\epsilon}(t)\|_X^2.
\eq

Note that $I_1,I_2$ satisfies
\bmas
&\partial_tI_1+\frac{1}{\epsilon}\tilde{v}\cdot\nabla_x I_1+\frac{1}{\epsilon^2}\nu(v)I_1=\frac{1}{\epsilon^2}KI_1, \quad I_1(0)=f_0,\\
&\partial_tI_2+\frac{1}{\epsilon}\tilde{v}\cdot\nabla_x I_2+\frac{1}{\epsilon^2}\nu(v)I_2=\frac{1}{\epsilon^2}KI_2+\frac{1}{\epsilon}\Gamma(f_{\epsilon},f_{\epsilon}),\quad I_2(0)=0.
\emas
Thus, we can represent $I_1,I_2$ as
\bma
I_1 &=e^{\frac{tA_{\epsilon}}{\epsilon^2}}f_0+\int^t_0\frac{1}{\epsilon^2}e^{\frac{(t-s)A_{\epsilon}}{\epsilon^2}} KI_1 ds,\label{rbenlv-2-i1}\\
I_2 &= \int^t_0e^{\frac{(t-s)A_{\epsilon}}{\epsilon^2}} \(\frac{1}{\epsilon^2}KI_2+\frac{1}{\epsilon}\Gamma(f_{\epsilon},f_{\epsilon})\) ds,\label{rbenlv-2-i2}
\ema
where $A_{\epsilon}=-\nu(v)-\epsilon\tilde{v}\cdot\nabla_x$.
By Lemma \ref{rbenletclm}, we have
\be
\|e^{\frac{tA_{\epsilon}}{\epsilon^2}}f_0\|_{L^{\infty}_{v,k}(H^N_x)}\leq Ce^{-\frac{\nu_0t}{\epsilon^2}}\|f_0\|_{L^{\infty}_{v,k}(H^N_x)},\quad k\geq0.\label{rbe-etce}
\ee
By Lemma \ref{rbegf8j2}, we can obtain
$$
\left\{\bln
&\|KI_1\|_{L^{\infty}_{v,0}(H^N_x)} \leq C\|I_1\|_{H^N}\leq C(1+t)^{-\frac{3}{4}}(\|f_0\|_{H^N}+\|f_0\|_{L^{2,1}}),  \\
&\|KI_2\|_{L^{\infty}_{v,0}(H^N_x)} \leq C\|I_2\|_{H^N}\leq C(1+t)^{-1}\|f_{\epsilon}(t)\|_X^2,
\eln\right.
$$
which, together with \eqref{rbenlv-2-i1}--\eqref{rbe-etce}, implies that
\bmas
\|I_1\|_{L^{\infty}_{v,0}(H^N_x)}& \leq C(1+t)^{-\frac{3}{4}}(\|f_0\|_{L^{\infty}_{v,3}(H^N_x)}+\|f_0\|_{L^{2,1}}),\\
\|I_2\|_{L^{\infty}_{v,0}(H^N_x)}& \leq C(1+t)^{-1}\|f(t)\|_{X}^2.
\emas

By Lemma \ref{rbegf8j2}, we can obtain that for $k>0$ and $i=1,2$,
$$
\|KI_i\|_{L^{\infty}_{v,k}(H^N_x)}\leq C\|I_i\|_{L^{\infty}_{v,k-1}(H^N_x)}.
$$
This and  \eqref{rbenlv-2-i1}--\eqref{rbenlv-2-i2}, Lemma \ref{rbenletclm} implies that
\bma
\|I_1\|_{L^{\infty}_{v,3}(H^N_x)}&= \|e^{\frac{tB_{\epsilon}}{\epsilon^2}}f_0\|_{L^{\infty}_{v,3}(H^N_x)}\leq C (1+t)^{-\frac{3}{4}}(\|f_0\|_{L^{\infty}_{v,3}(H^N_x)}+\|f_0\|_{L^{2,1}}),\label{rbenlv-4-1}\\
\|I_2\|_{L^{\infty}_{v,3}(H^N_x)}&=\bigg\|\frac{1}{\epsilon}\int^t_0e^{\frac{(t-s)B_{\epsilon}}{\epsilon^2}}\Gamma(f_{\epsilon},f_{\epsilon})ds\bigg\|_{L^{\infty}_{v,3}(H^N_x)}\leq C  (1+t)^{-1}\|f_{\epsilon}(t)\|_{X}^2. \label{gamma}
\ema
This gives
\bma
\|(Tf_{\epsilon})(t)\|_{L^{\infty}_{v,3}(H^N_x)}&\leq C_1(1+t)^{-\frac{3}{4}}(\|f_0\|_{L^{\infty}_{v,3}(H^N_x)}+\|f_0\|_{L^{2,1}})\nnm\\
&\quad+C_2(1+t)^{-\frac{3}{4}}\|f(t)\|_{X}^2.\label{rbenlv-6}
\ema
Let $\|f_0\|_{L^{\infty}_{v,3}(H^N_x)}+\|f_0\|_{L^{2,1}}\leq\delta_0$. Thus, for $f_{\epsilon}\in X$,
$$
\|(Tf_{\epsilon})(t)\|_{X}\leq C_1\delta_0+C_2\delta^2,
$$
which implies that $T$ is the mapping of $X\rightarrow X$ so long as
$$
\delta_0\leq\frac{\delta}{2C_1}\quad \mathrm{and}\quad 0<\delta<\frac{1}{2C_2}.
$$

Note that for any $f_1,f_2\in X$,
$$
\Gamma(f_2,f_2)-\Gamma(f_1,f_1)=\Gamma(f_2-f_1,f_2)+\Gamma(f_1,f_2-f_1).
$$
Thus, by  \eqref{gamma},  we can obtain
\bmas
\|T(f_2)-T(f_1)\|_{L^{\infty}_{v,3}(H^N_x)}&= \bigg\|\int^t_0\frac{1}{\epsilon}e^{\frac{(t-s)B_{\epsilon}}{\epsilon^2}}\[\Gamma(f_2,f_2)-\Gamma(f_1,f_1)\]ds\bigg\|_{L^{\infty}_{v,3}(H^N_x)}\nnm\\
&\leq C(1+t)^{-1}(\|f_1\|_X+\|f_2\|_X)\|f_2-f_1\|_X\nnm\\
&\leq 2C\delta(1+t)^{-1}\|f_2-f_1\|_X.
\emas
Thus, $T$ is a contraction mapping of $X\rightarrow X$. By the contraction mapping theorem, there exists a unique $f_{\epsilon}\in X$ such that $Tf_{\epsilon}=f_{\epsilon}$, i.e., $f_{\epsilon}$ is a unique global solution to the relativistic Boltzmann equation \eqref{rbe2}--\eqref{Pm3}.

Finally, we deal with \eqref{rbedl5-3}. It follows from Lemma \ref{3dmvpbfas2}, \eqref{nldh1} and \eqref{rbedl5-1} that
\bma
\| P_1f_{\epsilon}(t)\|_{H^k}&\leq C_1\(\epsilon(1+t)^{-\frac{5}{4}}+e^{-\frac{\tau_1t}{\epsilon^2}}\)\(\| f_0\|_{H^{1+k}}+\|f_0\|_{L^{2,1}}\)\nnm\\
&\quad+C_2\int^{t}_0\(\epsilon(1+t-s)^{-\frac{7}{4}}+\frac{1}{\epsilon}e^{-\frac{\tau_1(t-s)}{\epsilon^2}}\)\|\Gamma(f_{\epsilon},f_{\epsilon})\|_{H^{2+k}\cap L^{2,1}}ds\nnm\\
&\leq \(C_1\delta_0+C_2\delta_0^2\)\(\epsilon(1+t)^{-\frac{5}{4}}+e^{-\frac{\tau_1t}{4\epsilon^2}}\),\label{rbedl5-3-1}
\ema
where $k\leq N-2$. The proof of the lemma is completed.
\end{proof}

By virtue of Lemma \ref{3dmvpbfas3} and \eqref{nlnsdh1}, we have following lemma.

\begin{lem}\label{rbedl6}
Let $N\geq2$. There exists a small constant $\delta_0>0$ such that if $\|f_0\|_{H^N}+\|f_0\|_{L^{2,1}}\leq\delta_0$ then the iNS system \eqref{rNS-1} admits a unique global solution $(n,m,q)(t,x)\in L^{\infty}_t\big(L^2_x\big)$.
Moreover, $u(t,x,v)=n(t,x)\chi_0+m(t,x)\cdot \chi'+q(t,x)\chi_4$ has the following time-decay rate:
\bq\label{rbedl6-0}
\|u(t)\|_{H^N}\leq C\delta_0(1+t)^{-\frac{3}{4}},
\eq
where $C>0$ is a constant.
\end{lem}
\begin{proof}
We prove the existence of the solution $u$  by the contraction mapping theorem. Define the mapping $T$ as
\bq\label{rbedl6-tu}
(Tu)(t)=V(t)P_0f_0+\int^t_0V(t-s)\div_xZ(u(s))ds,
\eq
where $Z(u)=P_{0} (\tilde{v} L^{-1}\Gamma(u,u) )$.

Set the solution space
$$
\Omega=\Big\{u\in L^{\infty}((0,\infty),H^N)\,|\,\|u(t)\|_{\Omega}=\sup_{0\leq s\leq t}(1+s)^{\frac{3}{4}}\|u(s)\|_{H^N}\leq\delta\Big\}.
$$

For any $U_0\in\R^3$, we have
\bma\label{rbedl6-1}
\|\partial_x^{\alpha}V(t)\div_xU_0\|^2_{L^2}&\leq C\int_{\R^3}(\xi^{\alpha})^2e^{-2\rho_0|\xi|^2t}|\xi|^2\|\hat{U}_0\|^2d\xi\nnm\\
&\leq C\sup_{|\xi|\leq1}(\xi^{\alpha'})^2\|\hat{U}_0\|^2\int_{|\xi|\leq1}(\xi^{\alpha-\alpha'})^2|\xi|^2e^{-2\rho_0|\xi|^2t}d\xi\nnm\\
&\quad+C\sup_{|\xi|\leq1}\(|\xi|^2e^{-2\rho_0|\xi|^2t}\)\int_{|\xi|\geq1}(\xi^{\alpha})^2\|\hat{U}_0\|^2d\xi\nnm\\
&\leq C\((1+t)^{-\frac{5+2m}{2}}+t^{-1}e^{-2\rho_0t}\)\(\|\partial_x^{\alpha}U_0\|^2_{L^2}+\|\partial_x^{\alpha'}U_0\|^2_{L^{2,1}}\),
\ema
where $\alpha'\leq\alpha$, $m=|\alpha-\alpha'|$, $\rho_0=\min\{A_j,j=0,2,3\}$ and $C>0$ is a constant.

Thus, by Lemma \ref{3dmvpbfas3}, \eqref{nlnsdh1} and \eqref{rbedl6-1}, we have
\bma\label{rbedl6-2}
\|Tu(t)\|_{H^N}&\leq C(1+t)^{-\frac{3}{4}}\(\|P_0f_0\|_{H^N}+\|P_0f_0\|_{L^{2,1}}\)\nnm\\
&\quad+C\int_0^t\((1+t-s)^{-\frac{5}{4}}+(t-s)^{-\frac{1}{2}}e^{-\rho_0(t-s)}\)\nnm\\
&\qquad\times\(\|Z(u(s))\|_{H^N}+\|Z(u(s))\|_{L^{2,1}}\)ds\nnm\\
&\leq C\delta_0(1+t)^{-\frac{3}{4}}+C\|u(t)\|_{\Omega}^2(1+t)^{-\frac{3}{4}},
\ema
where we have used
$$
\|Z(u(s))\|_{H^N}+\|Z(u(s))\|_{L^{2,1}}\leq C\|u(t)\|_{\Omega}^2(1+s)^{-\frac{3}{2}}.
$$
By \eqref{rbedl6-2}, we can obtain
$$
\|Tu(t)\|_{\Omega}\leq C_3\delta_0+C_4\delta^2,
$$
which implies that $T$ is the mapping of $\Omega\rightarrow \Omega$ so long as $\delta_0\leq\frac{\delta}{2C_3}$ and $0<\delta<\frac{1}{2C_4}$ is small enough. This proves \eqref{rbedl6-0}. The existence of the solution can be proved by the contraction mapping theorem, the details are omitted. The proof of the lemma is completed.
\end{proof}

\subsection{Optimal convergence rate}
In this subsection, we will complete the proof of Theorem \ref{thm-2} about the convergence rate of the diffusion limit.
The solution $f_{\epsilon}(t)=f_{\epsilon}(t,x,v)$ to the relativistic Boltzmann equation \eqref{rbe2}--\eqref{Pm3} can be represented by
\be
f_{\epsilon}(t)=e^{\frac{tB_{\epsilon}}{\epsilon^2}}f_0+\frac{1}{\epsilon}\int_0^te^{\frac{(t-s)B_{\epsilon}}{\epsilon^2}}\Gamma(f_{\epsilon},f_{\epsilon})ds.\label{nldh1}
\ee
Let $(n,m,q)(t,x)$ be the global solution to the incompressible Navier-Stokes system \eqref{rNS-1}. Then by Lemma \ref{uuu}, $u(t,x,v)=n(t,x)\chi_0+m(t,x)\cdot\chi'+q(t,x)\chi_4$ can be represented by
\bq
u(t)=V(t)P_0f_0+\int^t_0V(t-s)\div_xZ(u) ds,\label{nlnsdh1}
\eq
where $Z(u)=P_{0} (\tilde{v} L^{-1}\Gamma(u,u) )$.

\noindent\textbf{\underline{Proof of Theorem 1.2.}} Firstly, we prove \eqref{thm-2-1}. Define
\be\label{dlth2-0}
\Lambda_{\epsilon}(t)=\sup_{0\leq s\leq t}\bigg\{\(\epsilon|\ln\epsilon|^2(1+s)^{-\frac{3}{4}}+\(1+\frac{s}{\epsilon}\)^{-1}\)^{-1}\|f_{\epsilon}(s)-u(s)\|_{L^{\infty}}\bigg\}.
\ee

By \eqref{nldh1}--\eqref{nlnsdh1}, we have
\bma\label{dlth2-1}
\|f_{\epsilon}(t)-u(t)\|_{L^{\infty}}&\leq \Big\|e^{\frac{tB_{\epsilon}}{\epsilon^2}}f_0-V(t)P_0f_0\Big\|_{L^{\infty}}+\int_0^t\bigg\|\frac{1}{\epsilon}e^{\frac{(t-s)B_{\epsilon}}{\epsilon^2}}\Gamma(f_{\epsilon},f_{\epsilon})-V(t-s)\div_xZ(u)\bigg\|_{L^{\infty}}ds\nnm\\
&=:I_1+I_2.
\ema

For $I_1$, by Lemma \ref{3dmvpbfas6}, we have
\bma\label{dlth2-2}
I_1&\leq C\bigg(\epsilon(1+t)^{-2}+\(1+\frac{t}{\epsilon}\)^{-1}\bigg)\(\|f_0\|_{H^{3}}+\|f_0\|_{W^{3,1}}\)\nnm\\
&\leq C\delta_0\bigg(\epsilon(1+t)^{-2}+\(1+\frac{t}{\epsilon}\)^{-1}\bigg).
\ema

For $I_2$, we decompose
\bma\label{dlth2-4}
I_2&\leq\int_0^t\bigg\|\frac{1}{\epsilon}e^{\frac{(t-s)B_{\epsilon}}{\epsilon^2}}\Gamma(f_{\epsilon},f_{\epsilon})-V(t-s)P_0\big(\tilde{v}\cdot\nabla_xL^{-1}\Gamma(f_{\epsilon},f_{\epsilon})\big)\bigg\|_{L^{\infty}}ds\nnm\\
&\quad+ \int_0^t\big\|V(t-s)P_0\big(\tilde{v}\cdot\nabla_xL^{-1}\Gamma(f_{\epsilon},f_{\epsilon})\big)-V(t-s)\div_xZ(u)\big\|_{L^{\infty}}ds\nnm\\
&=:I_{21}+I_{22}.
\ema

For $I_{21}$, by Lemma \ref{3dmvpbfas7} and \eqref{rbedl5-1-2}, and noticing that $P_0\Gamma(f_{\epsilon},f_{\epsilon})=0$, we have
\bma
I_{21}&\leq C\int_0^t\(\epsilon(1+t-s)^{-\frac{5}{2}}+\(1+\frac{t-s}{\epsilon}\)^{-1}+\frac{1}{\epsilon}e^{-\frac{\tau_1 (t-s)}{\epsilon^2}}\)\|\Gamma(f_{\epsilon},f_{\epsilon})\|_{H^{4}\cap W^{4,1}} ds\nnm\\
&\leq C\delta_0^2\int_0^t\(\epsilon(1+t-s)^{-\frac{5}{2}}+\(1+\frac{t-s}{\epsilon}\)^{-1}+\frac{1}{\epsilon}e^{-\frac{\tau_1 (t-s)}{\epsilon^2}}\)(1+s)^{-\frac{3}{2}}ds.\label{dlth2-4-1}
\ema
Note that for $t\leq1$,
\bq\label{dlth2-3-1-2-1}
\int^t_0\(1+\frac{t-s}{\epsilon}\)^{-1}(1+s)^{-\frac{3}{2}}ds\leq\int^t_0\(1+\frac{t-s}{\epsilon}\)^{-1}ds\leq C\epsilon|\ln\epsilon|,
\eq
and for $t\geq1$,
\bq\label{dlth2-3-1-2-2}
\int^t_0\(1+\frac{t-s}{\epsilon}\)^{-1}(1+s)^{-\frac{3}{2}}ds\leq \epsilon\int^t_0(t-s)^{-1}(1+s)^{-\frac{3}{2}}ds\leq C\epsilon(1+t)^{-1}.
\eq
Thus, we have
\be
I_{21}\leq C\delta_0^2\epsilon|\ln\epsilon|(1+t)^{-1}.\label{dlth2-4-2}
\ee

For $I_{22}$,  by Lemma \ref{3dmvpbfas3} and Lemmas \ref{rbedl5}--\ref{rbedl6}, we obtain
\bma\label{dlth2-4-7}
I_{22}&\leq \int_0^t\|V(t-s)\div_x(P_{0} (\tilde{v} L^{-1}\Gamma(f_{\epsilon},f_{\epsilon}) )-P_{0} (\tilde{v} L^{-1}\Gamma(u,u) )\|_{L^{\infty}}ds \nnm\\
&\leq C\int_0^t(1+t-s)^{-\frac{3}{4}}(t-s)^{-\frac{1}{2}}\| f_{\epsilon}-u\|_{L^{\infty}}(\| f_{\epsilon}\|_{H^2}+\|u\|_{H^2})ds\nnm\\
&\leq C\delta_0\Lambda_{\epsilon}(t)\int_0^t(1+t-s)^{-\frac{3}{4}}(t-s)^{-\frac{1}{2}}\nnm\\
&\qquad \times\(\epsilon|\ln\epsilon|^2(1+s)^{-\frac{3}{4}}+\(1+\frac{s}{\epsilon}\)^{-1}\)(1+s)^{-\frac{3}{4}}ds.
\ema
For $t\leq\epsilon$,
\bma
&\quad\int^t_{0}(1+t-s)^{-\frac{3}{4}}(t-s)^{-\frac{1}{2}}\(1+\frac{s}{\epsilon}\)^{-1}(1+s)^{-\frac{3}{4}}ds\nnm\\
&\leq C\int_0^t(t-s)^{-\frac{1}{2}}ds\leq   C\le C\(1+\frac{t}{\epsilon}\)^{-1}.\label{dlth2-4-7-2}
\ema
For $t\geq\epsilon$,
\bma
&\quad\(\int_0^{\frac{t}{2}}+\int^t_{\frac{t}{2}}\)(1+t-s)^{-\frac{3}{4}}(t-s)^{-\frac{1}{2}}\(1+\frac{s}{\epsilon}\)^{-1}(1+s)^{-\frac{3}{4}}ds\nnm\\
&\leq C\epsilon |\ln\epsilon|(1+t)^{-\frac{3}{4}}t^{-\frac12}+C\(1+\frac{t}{\epsilon}\)^{-1}(1+t)^{-\frac{3}{4}}\sqrt{t}\nnm\\
&\leq C\bigg(\epsilon|\ln\epsilon|^2(1+t)^{-1}+\(1+\frac{t}{\epsilon}\)^{-1}\bigg).\label{dlth2-4-7-6}
\ema
Thus, by \eqref{dlth2-4-7}--\eqref{dlth2-4-7-6}, we have
\be\label{dlth2-4-8}
I_{22} \leq C\(\delta_0^2+\delta_0\Lambda_{\epsilon}(t)\)\bigg(\epsilon|\ln\epsilon|^2(1+t)^{-1}+\(1+\frac{t}{\epsilon}\)^{-1}\bigg).
\ee

By combining  \eqref{dlth2-2}, \eqref{dlth2-4-2} and \eqref{dlth2-4-8}, we can obtain
\bq
\Lambda_{\epsilon}(t)\leq C\delta_0+C\delta_0^2+C\delta_0\Lambda_{\epsilon}(t),
\eq
where $C>0$ is a constant independent of $\epsilon$. By taking $\delta_0>0$ small enough, we can obtain \eqref{dlth2-0}, which
proves \eqref{thm-2-1}.

Finally, we prove \eqref{thm-2-2}. Define
\bq\label{dlth-2-2}
\Omega_{\epsilon}(t)=\sup_{0\leq s\leq t}\(\epsilon|\ln\epsilon|\)^{-1}(1+s)^{\frac{3}{4}}\(\|f_{\epsilon}(s)-u(s)\|_{L^{\infty}}\).
\eq
If $f_0$ satisfies \eqref{rbf0}, then we have by Lemma \ref{3dmvpbfas6} that
\be
I_1\leq C\delta_0\epsilon(1+t)^{-2}.\label{dlth-2-2-1}
\ee
For $I_{21}$, we have by \eqref{dlth2-4-1} that
\be\label{dlth-2-2-1j1}
I_{21}\le C\delta_0^2\eps|\ln \eps|(1+t)^{-1}.
\ee
For $I_{22}$, we have by Lemma \ref{3dmvpbfas3} and \eqref{rbedl5-1-2} that
\bma
I_{22}&\leq C\delta_0\Omega_{\epsilon}(t)\epsilon|\ln\epsilon|\int_0^t(1+t-s)^{-\frac{3}{4}}(t-s)^{-\frac{1}{2}}(1+s)^{-\frac{3}{2}}ds\nnm\\
&\quad+C\delta_0^2\int_0^t(1+t-s)^{-\frac{3}{4}}(t-s)^{-\frac{1}{2}}\(\epsilon(1+s)^{-\frac{5}{4}}+e^{-\frac{\tau_1 s}{4\epsilon^2}}\)(1+s)^{-\frac{3}{4}}ds\nnm\\
&\leq C\(\delta_0\Omega_{\epsilon}(t)+\delta_0^2\)\epsilon|\ln\epsilon|(1+t)^{-\frac{3}{4}},\label{dlth-2-2-5}
\ema
where we have used \eqref{dlth2-3-1-2-1}--\eqref{dlth2-3-1-2-2}.

By combining \eqref{dlth-2-2-1}--\eqref{dlth-2-2-5}, we have
\be\label{dlth2-5a}
\|f_{\epsilon}(t)-u(t)\|_{L^{\infty}}\leq C\(\delta_0+\delta_0^2+\delta_0\Omega_{\epsilon}(t)\) \epsilon|\ln\epsilon| (1+t)^{-\frac{3}{4}} .
\ee
Thus, we can obtain
$$
\Omega_{\epsilon}(t)\leq C\delta_0+C\delta_0^2+C\delta_0\Omega_{\epsilon}(t),
$$
where $C>0$ is a constant independent of $\epsilon$. By taking $\delta_0>0$ small enough, we can obtain \eqref{dlth-2-2}, which
proves \eqref{thm-2-2}. The proof is completed.

%
%
%
%

\section{Appendix}
\label{sect5}
\setcounter{equation}{0}

\subsection{Estimation of important parameters}

\begin{lem}[\cite{Cao-1}]\label{rbdl-bessel-kj}
It holds that
$$
K_{j+1}(z)=\frac{2j}{z}K_j(z)+K_{j-1}(z),\quad j\geq1,~~ z>0,
$$
 where $K_j(z)$ is given in \eqref{rbdl-kjzdif}. The asymptotic expansion for $K_j(z)$ takes the form
$$
K_j(z)=\sqrt{\frac{\pi}{2z}}\frac{1}{e^z}\[\sum^{n-1}_{m=0}\mathcal{A}_{j,m}z^{-m}+\gamma_{j,n}(z)z^{-n}\],\quad j\geq0,~~ z>0,~~ n\geq1,
$$
where the following additional identities and inequalities also holds:
\bmas
&K_j(z)<K_{j+1}(z),\quad j\geq0,\quad \mathcal{A}_{j,0}=1,\\
&\mathcal{A}_{j,m}=\frac{(4j^2-1)(4j^2-3^2)\cdots(4j^2-(2m-1)^2)}{m!8^m},\quad j\geq0,~~ m\geq1,\\
&|\gamma_{j,n}(z)|\leq2|\mathcal{A}_{j,n}|e^{\frac{(j^2-\frac{1}{4})}{z}},\quad j\geq0,~~ n\geq1.
\emas
\end{lem}

\begin{lem} \label{p123}
It holds that
\be
\left\{\bln
p_1&=\frac{1}{\sqrt{k_0T}}\(1+O\(\frac{k_0T}{\mathbf{c}^2}\)\),\\
p_2&=\sqrt{\frac{2}{3}}\frac{\mathbf{c}}{k_0T}\(1+O\(\frac{k_0T}{\mathbf{c}^2}\)\),\\
p_3& = \mathbf{c}+O\(\frac{k_0T}{\mathbf{c}}\),
\eln\right.
\ee where $p_i$, $i=1,2,3$ are given in \eqref{rbdl-p-123}.
\end{lem}
\begin{proof}
Since
$$  K_j\(\frac{\mathbf{c}^2}{k_0T}\)=\sqrt{\frac{\pi k_0T}{2\mathbf{c}^2}}e^{-\frac{\mathbf{c}^2}{k_0T}}\(1+\frac{4j^2-1}{8}\frac{k_0T}{\mathbf{c}^2}+\frac{(4j^2-1)(4j^2-9)}{128}\frac{k_0^2T^2}{\mathbf{c}^4}+O\(\frac{k_0^3T^3}{\mathbf{c}^6}\)\),
$$
it follows that
$$ \frac{K_3(\frac{\mathbf{c}^2}{k_0T})}{K_2(\frac{\mathbf{c}^2}{k_0T})}=1+\frac52\frac{k_0T}{\mathbf{c}^2}+\frac{15}{8}\frac{k_0^2T^2}{\mathbf{c}^4}+O\(\frac{k_0^3T^3}{\mathbf{c}^6}\).
$$
Thus,
\bma p_1&=\bigg(k_0T\frac{K_3(\frac{\mathbf{c}^2}{k_0T})}{K_2(\frac{\mathbf{c}^2}{k_0T})}\bigg)^{-\frac{1}{2}}=\frac{1}{\sqrt{k_0T}}\(1+O\(\frac{k_0T}{\mathbf{c}^2}\)\),
\\
p_3&=\mathbf{c}\frac{K_3(\frac{\mathbf{c}^2}{k_0T})}{K_2(\frac{\mathbf{c}^2}{k_0T})}-\frac{k_0T}{\mathbf{c}}= \mathbf{c}+O\(\frac{k_0T}{\mathbf{c}}\).
\ema

Note that
$$ \mathbf{c}^2\bigg(1-\frac{K_3(\frac{\mathbf{c}^2}{k_0T})^2}{K_2(\frac{\mathbf{c}^2}{k_0T})^2}\bigg)
=\mathbf{c}^2\bigg(1-\frac{K_3(\frac{\mathbf{c}^2}{k_0T})}{K_2(\frac{\mathbf{c}^2}{k_0T})}\bigg) \bigg(1+\frac{K_3(\frac{\mathbf{c}^2}{k_0T})}{K_2(\frac{\mathbf{c}^2}{k_0T})}\bigg) =-5k_0T-10\frac{k_0^2T^2}{\mathbf{c}^2}+O\(\frac{k_0^3T^3}{\mathbf{c}^4}\).
$$
Thus
$$ p_2=\(\frac{3}{2}\frac{k_0^2T^2}{\mathbf{c}^2}+O\(\frac{k_0^3T^3}{\mathbf{c}^4}\)\)^{-\frac12} =\sqrt{\frac{2}{3}}\frac{\mathbf{c}}{k_0T}\(1+O\(\frac{k_0T}{\mathbf{c}^2}\)\).
$$
The proof is completed.
\end{proof}

\begin{lem}\label{rbdl-ap-lp1-aaa}
For any $ n\ge 0$, it holds that
\bma
\int_{0}^{\infty}\frac{r^n}{\sqrt{\mathbf{c}^2+r^2}}e^{-\frac{\mathbf{c}\sqrt{\mathbf{c}^2+r^2}}{k_0T}}dr&=\frac{(2k_0T)^{\frac{n}{2}}\Gamma\(\frac{n+1}{2}\)}{\Gamma\(\frac{1}{2}\)}K_{\frac{n}{2}}\(\frac{\mathbf{c}^{2}}{k_0T}\),\label{rbdl-ap-lp1-aaa-1}
\\
\int_0^{\infty}r^n\sqrt{\mathbf{c}^2 +r^2 }e^{-\frac{\mathbf{c}\sqrt{\mathbf{c}^2 +r^2}}{k_0T}}dr
&=\frac{\(2k_0T\)^{\frac{n+2}{2}}\Gamma\(\frac{n+3}{2}\)}{\Gamma\(\frac{1}{2}\)}K_{\frac{n+2}{2}}\(\frac{\mathbf{c}^{2}}{k_0T}\)\nnm\\
&\quad+\mathbf{c}^{2}\frac{(2k_0T)^{\frac{n}{2}}\Gamma\(\frac{n+1}{2}\)}{\Gamma\(\frac{1}{2}\)}K_{\frac{n}{2}}\(\frac{\mathbf{c}^{2}}{k_0T}\).\label{rbdl-ap-lp1-aaa-2}
\ema
\end{lem}
\begin{proof}
By \eqref{rbdl-kjzdif} and $d\sqrt{\mathbf{c}^2+r^2}=\frac{r}{\sqrt{\mathbf{c}^2+r^2}}dr$, we can obtain
\bma
\int_{0}^{\infty}\frac{r^n}{\sqrt{\mathbf{c}^2+r^2}}e^{-\frac{\mathbf{c}\sqrt{\mathbf{c}^2+r^2}}{k_0T}}dr
&=\int_{0}^{\infty}r^{n-1}e^{-\frac{\mathbf{c}\sqrt{\mathbf{c}^2+r^2}}{k_0T}}d\sqrt{\mathbf{c}^2+r^2}\nnm\\
&=\mathbf{c}^{n}\int_{0}^{\infty}\(\frac{r^2}{\mathbf{c}^{2}}+1-1\)^{\frac{n-1}{2}}e^{-\frac{\mathbf{c}^2\sqrt{1+\frac{r^2}{\mathbf{c}^2}}}{k_0T}}d\sqrt{1+\frac{r^2}{\mathbf{c}^2}}\nnm\\
&=\mathbf{c}^{n}\int_{1}^{\infty}\(t^2-1\)^{\frac{n-1}{2}}e^{-\frac{\mathbf{c}^2t}{k_0T}}dt 
\nnm\\
&=\frac{(2k_0T)^{\frac{n}{2}}\Gamma\(\frac{n+1}{2}\)}{\Gamma\(\frac{1}{2}\)}K_{\frac{n}{2}}\(\frac{\mathbf{c}^{2}}{k_0T}\).
\ema
This proves \eqref{rbdl-ap-lp1-aaa-1}.

By \eqref{rbdl-kjzdif} and $d\sqrt{\mathbf{c}^2+r^2}=\frac{r}{\sqrt{\mathbf{c}^2+r^2}}dr$, we have
\bma
&\quad\int_0^{\infty}r^n\sqrt{r^2+\mathbf{c}^2}e^{-\frac{\mathbf{c}\sqrt{r^2+\mathbf{c}^2}}{k_0T}}dr\nnm\\
&=\int_0^{\infty}r^{n-1}(r^2+\mathbf{c}^2)e^{-\frac{\mathbf{c}\sqrt{r^2+\mathbf{c}^2}}{k_0T}}d\sqrt{\mathbf{c}^2+r^2}\nnm\\
&=\mathbf{c}^{n+2}\int_0^{\infty}\(\frac{r^2}{\mathbf{c}^2}+1-1\)^{\frac{n+1}{2}}e^{-\frac{\mathbf{c}^2\sqrt{\frac{r^2}{\mathbf{c}^2}+1}}{k_0T}}d\sqrt{1+\frac{r^2}{\mathbf{c}^2}}\nnm\\
&\quad+\mathbf{c}^{n+2}\int_0^{\infty}\(\frac{r^2}{\mathbf{c}^2}+1-1\)^{\frac{n-1}{2}}e^{-\frac{\mathbf{c}^2\sqrt{\frac{r^2}{\mathbf{c}^2}+1}}{k_0T}}d\sqrt{1+\frac{r^2}{\mathbf{c}^2}}\nnm\\
&=\mathbf{c}^{n+2}\int_1^{\infty}\(t^2-1\)^{\frac{n+1}{2}}e^{-\frac{\mathbf{c}^2t}{k_0T}}dt+\mathbf{c}^{n+2}\int_1^{\infty}\(t^2-1\)^{\frac{n-1}{2}}e^{-\frac{\mathbf{c}^2t}{k_0T}}dt 
\nnm\\
&=\frac{\(2k_0T\)^{\frac{n+2}{2}}\Gamma\(\frac{n+3}{2}\)}{\Gamma\(\frac{1}{2}\)}K_{\frac{n+2}{2}}\(\frac{\mathbf{c}^{2}}{k_0T}\)+\mathbf{c}^{2}\frac{(2k_0T)^{\frac{n}{2}}\Gamma\(\frac{n+1}{2}\)}{\Gamma\(\frac{1}{2}\)}K_{\frac{n}{2}}\(\frac{\mathbf{c}^{2}}{k_0T}\).
\ema
This gives \eqref{rbdl-ap-lp1-aaa-2}. The proof of the Lemma is completed.
\end{proof}

With the help of Lemma \ref{rbdl-ap-lp1-aaa}, we have the following lemma.

\begin{lem}\label{rbdl-ap-lp1}
Let $\vartheta_{ijkl}=(P_1(v_i\chi_j),\tilde{v}_k\chi_l) $ and $z_{kk}=(P_1(v_kE_0),\tilde{v}_kE_0)$ with $1\leq i,j, k,l\leq3$. The following conditions hold.
\be \label{rbdl-ap-lp1-1}
\left\{\bln
&\vartheta_{ikik} =\vartheta_{ikki} =\vartheta_{1212}=b_1,\quad &1\leq i\neq k\leq3,\\
&\vartheta_{kkkk} =\vartheta_{1111}=3b_1-b_2,\quad &1\leq k\leq3,\\
&\vartheta_{iikk} =\vartheta_{1122}= b_1-b_2,\quad &1\leq i\neq k\leq3, \\
&\vartheta_{ijkl} =0,\qquad\qquad\qquad\qquad &(k,l)\neq(i,j)~\mathrm{or}~(j,i),\\
&z_{kk}=z_{11}=b_3, \qquad\qquad\qquad  &k=1,2,3,
\eln\right.
\ee
where
\bmas
&b_1=k_0T, \quad b_2=\frac{2k_0T}{3}+\frac{\mathbf{c}ap_2}{p_1} \(\frac{6k_0T}{\mathbf{c}^2}+ \frac{K_2(\frac{\mathbf{c}^2}{k_0T})}{K_3(\frac{\mathbf{c}^2}{k_0T})}-\frac{K_3(\frac{\mathbf{c}^2}{k_0T})}{K_2(\frac{\mathbf{c}^2}{k_0T})} \),\\
&b_3= \frac{\mathbf{c}p_2ab}{p_1^2 (a^2+b^2) }-\frac{\mathbf{c}p_2^2b^2}{p_1^2 (a^2+b^2) } \(\frac{6k_0T}{\mathbf{c}}+\mathbf{c}\(\frac{K_2(\frac{\mathbf{c}^2}{k_0T})}{K_3(\frac{\mathbf{c}^2}{k_0T})}-\frac{K_3(\frac{\mathbf{c}^2}{k_0T})}{K_2(\frac{\mathbf{c}^2}{k_0T})}\)\),\\
&a=\frac{p_2}{3p_1}\( \mathbf{c}+\frac{2k_0T}{\mathbf{c}}\frac{K_2(\frac{\mathbf{c}^2}{k_0T})}{K_3(\frac{\mathbf{c}^2}{k_0T})}\),\quad b=\frac{2k_0Tp_1}{3},
\emas
and $p_j$ $(j=1,2,3)$ are given in \eqref{rbdl-p-123}.  Moreover, it holds that
\bmas
&a=\sqrt{\frac{2}{3}}\frac{\mathbf{c}^2}{3\sqrt{k_0T}}\(1+O\(\frac{k_0T}{\mathbf{c}^2}\)\),\quad b=\frac{2\sqrt{k_0T}}{3}\(1+O\(\frac{k_0T}{\mathbf{c}^2}\)\),\\
&b_2=\frac{2}{9}\mathbf{c}^2\(1+O\(\frac{k_0T}{\mathbf{c}^2}\)\),\quad b_3=\frac{2}{3}k_0T\(1+O\(\frac{k_0T}{\mathbf{c}^2}\)\),
\emas
\end{lem}
\begin{proof}
By changing variable $v_i\rightarrow-v_i$, we have
$$
(P_1(v_i\chi_j),\tilde{v}_k\chi_l)=-(P_1(v_i\chi_j),\tilde{v}_k\chi_l),\quad (k,l)\neq(i,j)~\mathrm{or}~(j,i).
$$
This implies that $\vartheta_{ijkl} =0$ for $(k,l)\neq(i,j)~\mathrm{or}~(j,i)$.

By \eqref{mmv-dif}, \eqref{chii-1}--\eqref{mami} and $P_1(v_1\chi_2)=v_1\chi_2 $, it can be verify that
\bma
(P_1(v_1\chi_2),\tilde{v}_1\chi_2)&=\mathbf{c}p_0p_1^2\int_{0}^{\infty}\frac{r^6}{\sqrt{\mathbf{c}^2+r^2}}e^{-\frac{\mathbf{c}\sqrt{\mathbf{c}^2+r^2}}{k_0T}}dr\int_0^{2\pi}\cos^2\theta d\theta\int_0^{\pi}\sin^3\varphi\cos^2\varphi d\varphi\nnm\\
&=\frac{4\pi\mathbf{c}}{15}p_0p_1^2\int_{0}^{\infty}\frac{r^6}{\sqrt{\mathbf{c}^2+r^2}}e^{-\frac{\mathbf{c}\sqrt{\mathbf{c}^2+r^2}}{k_0T}}dr\nnm\\
&=\frac{4\pi\mathbf{c}}{15}\frac{\(k_0T\frac{K_3(\frac{\mathbf{c}^2}{k_0T})}{K_2(\frac{\mathbf{c}^2}{k_0T})}\)^{-1}}{4\pi\mathbf{c}k_0TK_2(\frac{\mathbf{c}^2}{k_0T})}\frac{(2k_0T)^{3}\Gamma\(\frac{7}{2}\)}{\Gamma\(\frac{1}{2}\)}K_{3}\(\frac{\mathbf{c}^{2}}{k_0T}\)\nnm\\
&=k_0T=:b_1,\label{rbdl-ap-lp1-1-1}
\ema
where we have used $ \Gamma\(\frac{n+1}{2}\)=\frac{(n-1)!!}{2^n}\Gamma\(\frac{1}{2}\)$.

By \eqref{mmv-dif}, \eqref{chii-1}--\eqref{mami} and $P_1(v_1\chi_1)=v_1\chi_1-(v_1\chi_1,\chi_0)\chi_0-(v_1\chi_1,\chi_4)\chi_4 $, we can obtain
\bma
(P_1(v_1\chi_1),\tilde{v}_1\chi_1)
&=(v_1\chi_1,\tilde{v}_1\chi_1)-\frac{1}{p_1}(\chi_0,\tilde{v}_1\chi_1)-(v_1\chi_1,\chi_4)(\chi_4,\tilde{v}_1\chi_1)\nnm\\
&=\frac{4\pi\mathbf{c}}{5}p_0p_1^2\int_{0}^{\infty}\frac{r^6}{\sqrt{\mathbf{c}^2+r^2}}e^{-\frac{\mathbf{c}\sqrt{\mathbf{c}^2+r^2}}{k_0T}}dr-\frac{b}{p_1}-a(v_1\chi_1,\chi_4)\nnm\\
&=3b_1-b_2,\label{rbdl-ap-lp1-1-2}\\
(P_1(v_1\chi_1),\tilde{v}_2\chi_2)
&=(v_1\chi_1,\tilde{v}_2\chi_2)-\frac{1}{p_1}(\chi_0,\tilde{v}_2\chi_2)-a(v_1\chi_1,\chi_4)\nnm\\
&=\frac{4\pi\mathbf{c}}{15}p_0p_1^2\int_{0}^{\infty}\frac{r^6}{\sqrt{\mathbf{c}^2+r^2}}e^{-\frac{\mathbf{c}\sqrt{\mathbf{c}^2+r^2}}{k_0T}}dr-\frac{b}{p_1}-a(v_1\chi_1,\chi_4)\nnm\\
&=b_1-b_2,\label{rbdl-ap-lp1-1-3}
\ema
where $b_2=\frac{b}{p_1}+a(v_1\chi_1,\chi_4).$

By \eqref{mmv-dif}, \eqref{chii-1}--\eqref{mami} and Lemma \ref{rbdl-ap-lp1-aaa}, we have
\bmas
b&=(\tilde{v}_1\chi_0,\chi_1)\\
&=\mathbf{c}p_0p_1\int_{0}^{\infty}\frac{r^4}{\sqrt{\mathbf{c}^2+r^2}}e^{-\frac{\mathbf{c}\sqrt{\mathbf{c}^2+r^2}}{k_0T}}dr\int_0^{2\pi} d\theta\int_0^{\pi}\sin\varphi\cos^2\varphi d\varphi\\
&=\frac{4\pi\mathbf{c}}{3}\frac{p_1}{4\pi\mathbf{c}k_0TK_2(\frac{\mathbf{c}^2}{k_0T})}\frac{(2k_0T)^{2}\Gamma\(\frac{3}{2}\)}{\Gamma\(\frac{1}{2}\)}K_{2}\(\frac{\mathbf{c}^{2}}{k_0T}\)=\frac{2k_0Tp_1}{3},
\\
a&=(\tilde{v}_1\chi_1,\chi_4)=(\tilde{v}_1\chi_1,p_2(v_0-p_3)\chi_0)\\
&=\frac{\mathbf{c}p_2}{p_1}-p_2p_3b=\frac{p_2}{p_1}\( \mathbf{c}-\frac{2k_0T}{3}  p_1^2p_3\)\\
&=\frac{p_2}{3p_1}\bigg( \mathbf{c}+\frac{2k_0T}{\mathbf{c}}\frac{K_2(\frac{\mathbf{c}^2}{k_0T})}{K_3(\frac{\mathbf{c}^2}{k_0T})}\bigg),
\\
(v_1\chi_1,\chi_4)&=p_1p_2(v_1^2\chi_0,v_0\chi_0)-\frac{p_2p_3}{p_1}\\
&=p_0p_1p_2\int_0^{\infty}r^4\sqrt{r^2+\mathbf{c}^2}e^{-\frac{\mathbf{c}\sqrt{r^2+\mathbf{c}^2}}{k_0T}}dr\int_0^{2\pi} d\theta\int_0^{\pi}\sin\varphi\cos^2\varphi d\varphi-\frac{p_2p_3}{p_1}\\
&=\frac{p_1p_2}{3\mathbf{c}k_0TK_2(\frac{\mathbf{c}^2}{k_0T})}\(15(k_0T)^{3}K_{3}\(\frac{\mathbf{c}^{2}}{k_0T}\)+3\mathbf{c}^{2}(k_0T)^{2}K_{2}\(\frac{\mathbf{c}^{2}}{k_0T}\)\)-\frac{p_2p_3}{p_1}\\
&=\frac{5(k_0T)^2p_1p_2}{\mathbf{c}}\frac{K_3(\frac{\mathbf{c}^2}{k_0T})}{K_2(\frac{\mathbf{c}^2}{k_0T})}+\mathbf{c}k_0Tp_1p_2-\frac{p_2p_3}{p_1}\\
&=\frac{p_2}{p_1}\(\frac{5k_0T}{\mathbf{c}}+\mathbf{c}k_0Tp_1^2-p_3\)
=\frac{p_2}{p_1}\(\frac{6k_0T}{\mathbf{c}}+\mathbf{c}\(\frac{K_2(\frac{\mathbf{c}^2}{k_0T})}{K_3(\frac{\mathbf{c}^2}{k_0T})}-\frac{K_3(\frac{\mathbf{c}^2}{k_0T})}{K_2(\frac{\mathbf{c}^2}{k_0T})}\)\).
\emas
By combining \eqref{rbdl-ap-lp1-1-1}--\eqref{rbdl-ap-lp1-1-3}, we can obtain \eqref{rbdl-ap-lp1-1}.

Note that
$$(P_1(v_1E_0),\tilde{v}_1E_0)=\frac{bp_2}{\sqrt{a^2+b^2}}(P_1(v_0v_1\chi_0),\tilde{v}_1E_0).$$
By \eqref{rbdl-e0} and \eqref{p2}, we have
\bma
 &\quad(P_1(v_0v_1\chi_0),\tilde{v}_1E_0)=(v_0v_1\chi_0,\tilde{v}_1E_0)-(P_0(v_0v_1\chi_0),\tilde{v}_1E_0)\nnm\\
&=-\frac{\mathbf{c}a}{\sqrt{a^2+b^2}}(v_1\chi_0,v_1\chi_0)+\frac{\mathbf{c}b}{\sqrt{a^2+b^2}}(v_1\chi_0,v_1\chi_4)+\frac{a}{\sqrt{a^2+b^2}}(v_0v_1\chi_0,\chi_1)(\chi_1,\tilde{v}_1\chi_0)\nnm\\
&\quad-\frac{b}{\sqrt{a^2+b^2}}(v_0v_1\chi_0,\chi_1)(\chi_1,\tilde{v}_1\chi_4)\nnm\\
&=-\frac{\mathbf{c}a}{\sqrt{a^2+b^2}}(v_1\chi_0,v_1\chi_0)+\frac{\mathbf{c}b}{\sqrt{a^2+b^2}}(v_1\chi_0,v_1\chi_4)\nnm\\
&=-\frac{\mathbf{c}a}{p_1^2\sqrt{a^2+b^2}}+\frac{\mathbf{c}b}{p_1\sqrt{a^2+b^2}}(v_1\chi_1,\chi_4).
\ema

Thus, from Lemma \ref{p123}, we can obtain
\bmas
&a=\sqrt{\frac{2}{3}}\frac{\mathbf{c}^2}{3\sqrt{k_0T}}\(1+O\(\frac{k_0T}{\mathbf{c}^2}\)\),\quad b=\frac{2\sqrt{k_0T}}{3}\(1+O\(\frac{k_0T}{\mathbf{c}^2}\)\),\\
&b_2=\frac{2}{9}\mathbf{c}^2\(1+O\(\frac{k_0T}{\mathbf{c}^2}\)\),\quad b_3=\frac{2}{3}k_0T\(1+O\(\frac{k_0T}{\mathbf{c}^2}\)\).
\emas
The proof of the lemma is completed.
\end{proof}

\begin{lem}\label{rbdl-aj-est}
It holds that
\be
\left\{\bln
&A_0=-\frac{(a+bp_2p_3)^2}{a^2+b^2}\(L^{-1}P_1\tilde{v}_1\chi_0,\tilde{v}_1\chi_0\),\\
&A_{\pm1}=-\frac{(b-ap_2p_3)^2}{2a^2+2b^2}\(L^{-1}P_1\tilde{v}_1\chi_0,\tilde{v}_1\chi_0\)-\frac{1}{2}\(L^{-1}P_1\tilde{v}_1\chi_1,\tilde{v}_1\chi_1\),\\
&A_{k}=- \(L^{-1}P_1\tilde{v}_1 \chi_2,\tilde{v}_1 \chi_2\),\quad k=2,3,
\eln\right.\label{rbdl-aj-est-1}
\ee
where $p_l$ $(l=1,2,3)$ are given in \eqref{rbdl-p-123}, and $a,b$ are defined in \eqref{rbdl-a}--\eqref{rbdl-b}. Moreover, it holds that
$$A_i\sim 1 ,~~ i=0,2,3, \quad A_{\pm1}\sim \mathbf{c}^4.$$
\end{lem}
\begin{proof}
Let  $\mathbb{O}$ be a rotation in $\R^3$ with $\mathbb{O}^T\omega=(1,0,0)$. By changing variable $v\to \mathbb{O}v$, we obtain
$$A_j=-\(L^{-1}P_1(\tilde{v}\cdot\omega)E_j ,(\tilde{v}\cdot\omega)E_j \)=-\(L^{-1}P_1\tilde{v}_1F_j,\tilde{v}_1F_j\),$$
where $A_j$ $(j=-1,0,1,2,3)$ are given in \eqref{sp4}, and
\bmas
\left\{\bal
F_{\pm1}=\frac{b}{\sqrt{2a^2+2b^2}}\chi_0\mp\sqrt{\frac{1}{2}}\chi_1+\frac{a}{\sqrt{2a^2+2b^2}}\chi_4,\\
F_0=-\frac{a}{\sqrt{a^2+b^2}}\chi_0+\frac{b}{\sqrt{a^2+b^2}}\chi_4,\\
F_l=\chi_{l},\quad l=2,3.
\ea\right.
\emas
By \eqref{chii-1}, we can obtain
\bmas
 &\quad\(L^{-1}P_1\tilde{v}_1F_0,\tilde{v}_1F_0\)\\
&=\frac{a^2}{a^2+b^2}\(L^{-1}P_1\tilde{v}_1\chi_0,\tilde{v}_1\chi_0\)+\frac{b^2}{a^2+b^2}\(L^{-1}P_1\tilde{v}_1p_2(v_0-p_3)\chi_0,\tilde{v}_1p_2(v_0-p_3)\chi_0\)\\
&\quad-\frac{2ab}{a^2+b^2}\(L^{-1}P_1\tilde{v}_1\chi_0,\tilde{v}_1p_2(v_0-p_3)\chi_0\) \\
&=\frac{a^2}{a^2+b^2}\(L^{-1}P_1\tilde{v}_1\chi_0,\tilde{v}_1\chi_0\)+\frac{b^2p_2^2p_3^2}{a^2+b^2}\(L^{-1}P_1\tilde{v}_1\chi_0,\tilde{v}_1\chi_0\)+\frac{2abp_2p_3}{a^2+b^2}\(L^{-1}P_1\tilde{v}_1\chi_0,\tilde{v}_1\chi_0\)\\
&=\frac{(a+bp_2p_3)^2}{a^2+b^2}\(L^{-1}P_1\tilde{v}_1\chi_0,\tilde{v}_1\chi_0\),\\
 &\quad\(L^{-1}P_1\tilde{v}_1F_{\pm1},\tilde{v}_1F_{\pm1}\)\\
&=\frac{b^2}{2a^2+2b^2}\(L^{-1}P_1\tilde{v}_1\chi_0,\tilde{v}_1\chi_0\)+\frac{a^2}{2a^2+2b^2}\(L^{-1}P_1\tilde{v}_1p_2(v_0-p_3)\chi_0,\tilde{v}_1p_2(v_0-p_3)\chi_0\)\\
&\quad+\frac{2ab}{2a^2+2b^2}\(L^{-1}P_1\tilde{v}_1\chi_0,\tilde{v}_1p_2(v_0-p_3)\chi_0\)+\frac{1}{2}\(L^{-1}P_1\tilde{v}_1\chi_1,\tilde{v}_1\chi_1\)\\
&=\frac{b^2}{2a^2+2b^2}\(L^{-1}P_1\tilde{v}_1\chi_0,\tilde{v}_1\chi_0\)+\frac{a^2p^2_2p^2_3}{2a^2+2b^2}\(L^{-1}P_1\tilde{v}_1\chi_0,\tilde{v}_1\chi_0\)\\
&\quad-\frac{2abp_2p_3}{2a^2+2b^2}\(L^{-1}P_1\tilde{v}_1\chi_0,\tilde{v}_1\chi_0\)+\frac{1}{2}\(L^{-1}P_1\tilde{v}_1\chi_1,\tilde{v}_1\chi_1\)\\
&=\frac{(b-ap_2p_3)^2}{2a^2+2b^2}\(L^{-1}P_1\tilde{v}_1\chi_0,\tilde{v}_1\chi_0\)+\frac{1}{2}\(L^{-1}P_1\tilde{v}_1\chi_1,\tilde{v}_1\chi_1\),\\
 &\quad\(L^{-1}P_1\tilde{v}_1F_{k},\tilde{v}_1F_{k}\)=\(L^{-1}P_1\tilde{v}_1\chi_2,\tilde{v}_1 \chi_2\), \quad k=2,3.
\emas
These implies that \eqref{rbdl-aj-est-1}.

Note that
\bmas
C_1( P_1\tilde{v}_1\chi_j,P_1\tilde{v}_1\chi_j)&\ge -(L^{-1}P_1\tilde{v}_1\chi_0,\tilde{v}_1\chi_0)\ge  C_2(\nu^{-1}P_1\tilde{v}_1\chi_j,P_1\tilde{v}_1\chi_j)\\
&\ge C_3((1+|v|)^{-1}P_1\tilde{v}_1\chi_j,P_1\tilde{v}_1\chi_j).
\emas
Thus, it can be verify that
\bma
\( \frac{1}{1+|v|}P_1\tilde{v}_1v_2\chi_0,P_1\tilde{v}_1v_2\chi_0\)&=\(\frac{1}{1+|v|}\tilde{v}_1v_2\chi_0,\tilde{v}_1v_2\chi_0\)\nnm\\
&= C\mathbf{c}^2 p_0\int_0^{\infty}\frac1{1+r} \frac{r^6}{\mathbf{c}^2+r^2}e^{-\frac{\mathbf{c}\sqrt{\mathbf{c}^2+r^2}}{k_0T}}dr\nnm\\
&\geq C\mathbf{c}^2p_0\int_1^{\infty}\frac{r^5}{\mathbf{c}^2+r^2}e^{-\frac{\mathbf{c}\sqrt{\mathbf{c}^2+r^2}}{k_0T}}dr\nnm\\
&\geq C(k_0T)^{\frac32}\frac{K_{\frac52}(\frac{\mathbf{c}^2}{k_0T})}{K_2(\frac{\mathbf{c}^2}{k_0T})},
\ema
where we have used
\bmas
\int_1^{\infty}\frac{ r^n}{\mathbf{c}^2+r^2}e^{-\frac{\mathbf{c}\sqrt{\mathbf{c}^2+r^2}}{k_0T}}dr&=\int_{0}^{\infty}\frac{r^{n-1}}{\sqrt{\mathbf{c}^2+r^2}}e^{-\frac{\mathbf{c}\sqrt{\mathbf{c}^2+r^2}}{k_0T}}d\sqrt{\mathbf{c}^2+r^2}\nnm\\
&=\mathbf{c}^{n-1}\int_{1}^{\infty}\frac{(t^2-1)^{\frac{n-1}{2}}}{t}e^{-\frac{\mathbf{c}^2t}{k_0T}}dt 
\nnm\\
&\ge \mathbf{c}^{n-1}\int_{1}^{\infty} (t^2-1)^{\frac{n-1}{2}} e^{-\frac{\mathbf{c}^2t}{k_0T}}dt
\nnm\\
&=\frac{(2k_0T)^{\frac{n}{2}}\Gamma\(\frac{n+1}{2}\)}{\mathbf{c}\Gamma\(\frac{1}{2}\)}K_{\frac{n}{2}}\(\frac{\mathbf{c}^{2}}{k_0T}\).
\emas
Moreover,
$$
(  P_1\tilde{v}_1v_2\chi_0,P_1\tilde{v}_1v_2\chi_0)\leq Cp_0\mathbf{c}\int_0^{\infty}\frac{r^6}{\sqrt{\mathbf{c}^2+r^2}}e^{-\frac{\mathbf{c}\sqrt{\mathbf{c}^2+r^2}}{k_0T}}dr\leq C(k_0T)^2\frac{K_3(\frac{\mathbf{c}^2}{k_0T})}{K_2(\frac{\mathbf{c}^2}{k_0T})}.
$$
Thus,
$$-(L^{-1}P_1\tilde{v}_1v_2\chi_0,\tilde{v}_1v_2\chi_0)\sim 1.$$
Moreover,
\bmas
\frac{(a+bp_2p_3)^2}{a^2+b^2}&=\frac{(1+\frac{bp_2p_3}{a})^2}{1+\frac{b^2}{a^2}}=9\(1+O(\frac{k_0T}{\mathbf{c}^2})\),\\
\frac{(b-ap_2p_3)^2}{2a^2+2b^2}&=\frac{(p_2p_3)^2(1-\frac{b}{ap_2p_3})^2}{2+\frac{2b^2}{a^2}}=\frac{\mathbf{c}^4}{3(k_0T)^2}\(1+O(\frac{k_0^2T^2}{\mathbf{c}^4})\).
\emas
Thus, we have proved the lemma.
\end{proof}

\subsection{Estimation of $\nu(v)$ and $k(u,v)$}

In this subsection, we estimate $\nu(v)$ and $k(u,v)$. Recall that
\bma
\nu(v)&=\int_{\R^3}\int_{\S^2}\frac{\mathbf{c}g\sqrt{s}}{4u_0v_0}M(u)d\omega du,\label{rbdl-ap-nu}\\
Kf&=\int_{\R^3}\int_{\S^2}\frac{\mathbf{c}g\sqrt{s}}{4u_0v_0}\big(f(v')\sqrt{M(u')}+f(u')\sqrt{M(v')}\big)\sqrt{M(u)}d\omega du\nnm\\
&\quad-\int_{\R^3}\int_{\S^2}\frac{\mathbf{c}g\sqrt{s}}{4u_0v_0}\sqrt{M(u)}\sqrt{M(v)}f(u)d\omega du\nnm\\
&=\int_{\R^3}k_2(v,u)f(u)du-\int_{\R^3}k_1(v,u)f(u)du=\int_{\R^3}k(v,u)f(u)du.\label{rbdl-ap-K}
\ema


\begin{lem}\label{ap-est-nuv}

For $\mathbf{c}\geq1$, there exist two positive constants $\nu_0$ and $\nu_1$ are independent of $\mathbf{c}$ such that
\bq\label{rbdl-ap-hbnugj}
\nu_0w(v)\leq\nu(v)\leq\nu_1w(v),
\eq
where
\be
w(v)=
\left\{\bal
1+|v|,\quad&|v|\leq\mathbf{c},\\
\mathbf{c},\quad&|v|\geq\mathbf{c}.
\ea\right.\label{rbdl-wccc}
\ee
\end{lem}
\begin{proof}
By \eqref{rbdl-ap-nu}, we have
\be\label{rbdl-hbnueq}
\nu(v)= \pi\mathbf{c}p_0 \int_{\R^3} \frac{g\sqrt{s}}{u_0v_0}e^{-\frac{\mathbf{c}u_0}{k_0T}} du.
\ee

Since $ g^2=2(u_0v_0-u\cdot v-\mathbf{c}^2)$, it follows that (\cite{Glassey-2})
\be
\frac{|u\times v|^2+\mathbf{c}^2|u-v|^2}{u_0v_0}\leq g^2\leq |u-v|^2,\label{rbdl-hbggeq}
\ee
where we have used $|u\cdot v+\mathbf{c}^2|\le  u_0v_0.$ By  Lemma \ref{rbdl-bessel-kj}, we have for $\mathbf{c}\gg1$,
$$
K_{2}\(\frac{\mathbf{c}^2}{k_0T}\)\sim \sqrt{\frac{\pi k_0T}{2 }}\mathbf{c}^{-1}e^{-\frac{\mathbf{c}^2}{k_0T}}.
$$
This and \eqref{rbdl-kjzdif} implies that
\be
p_0=\frac{1}{4\pi\mathbf{c}k_0TK_2(\frac{\mathbf{c}^2}{k_0T})}\sim \frac1{(2\pi k_0T)^{\frac32}}e^{\frac{\mathbf{c}^2}{k_0T}}.\label{p0}
\ee

Note that (cf. \cite{Wang-2})
\bma
\mathbf{c}^2-\mathbf{c}u_0&=-\frac{|u|^2}{1+\sqrt{1+\frac{|u|^2}{\mathbf{c}^2}}}\geq-|u|^2, \label{cu1}\\
\mathbf{c}^2-\mathbf{c}u_0&\leq-\frac{|u|^2}{1+\sqrt{1+|u|^2}}\leq-\sqrt{1+|u|^2}+1\leq-|u|+1. \label{cu2}
\ema
These implies that
\be
\frac{C_1}{(2\pi k_0T)^{\frac32}}e^{-\frac{|u|^2}{k_0T}}\leq p_0e^{-\frac{\mathbf{c}u_0}{k_0T}}\leq  \frac{C_2}{(2\pi k_0T)^{\frac32}}e^{-\frac{|u|}{k_0T}},\label{rbdl-p0eu0}
\ee
where $C_1,C_2>0$ are constants independent of $\mathbf{c}$.

Thus, by \eqref{rbdl-vmgs}, \eqref{rbdl-p0eu0} and \eqref{rbdl-hbnueq}--\eqref{rbdl-hbggeq}, we have for $|v|\geq\mathbf{c}$,
\bma
\nu(v)
&\geq C\mathbf{c}\int_{\R^3}\frac{g^2}{u_0v_0}e^{-\frac{|u|^2}{k_0T}}du\nnm\\
&\geq C\mathbf{c}^3v_0^{-2}\int_{\R^3}\frac{|u-v|^2}{u_0^{2}}e^{-\frac{|u|^2}{k_0T}}du\nnm\\
&= C\mathbf{c}^3v_0^{-2}\int_{\R^3}\frac{(|u|^2-2u\cdot v + |v|^2)}{u_0^{2}}e^{-\frac{|u|^2}{k_0T}}du\nnm\\
&= C\mathbf{c}^3v_0^{-2}\int_{0}^{\infty}\int_{0}^{\pi}\frac{r^{2}-2r |v|\cos\theta+|v|^2}{\mathbf{c}^2+r^2}e^{-\frac{r^2}{k_0T}}r^2\sin\theta dr d\theta\nnm\\
&= C\mathbf{c}^3v_0^{-2}\int_{0}^{\infty}\frac{r^{4} +r^2|v|^2}{\mathbf{c}^2+r^2}e^{-\frac{r^2}{k_0T}}dr \nnm\\
&\geq \nu_4\mathbf{c}v_0^{-2}(1+|v|^{2}),\label{rbdl-hbenuvL}
\ema
where we have used
\be \label{kkk}
\int_{0}^{\infty}\frac{r^{k}}{(r^2+\mathbf{c}^2)^l}e^{-\frac{r^2}{k_0T}}dr
\geq\frac{1}{(5\mathbf{c}^2)^l}\int_{1}^{2}e^{-\frac{r^2}{k_0T}}r^kdr\geq C\mathbf{c}^{-2l},\quad k,l\ge 0.
\ee

For $|v|\leq\mathbf{c}$, we have by \eqref{rbdl-vmgs} and \eqref{rbdl-hbnueq}--\eqref{rbdl-hbggeq} that
\bma
\nu(v)
&\geq C\mathbf{c}^2\int_{\R^3}\frac{g}{u_0v_0}e^{-\frac{|u|^2}{k_0T}}du\nnm\\
&\geq C\mathbf{c}^3v_0^{-\frac{3}{2}}\int_{\R^3}\frac{|u-v|}{u_0^{\frac{3}{2}}}e^{-\frac{|u|^2}{k_0T}}du\nnm\\
&= C\mathbf{c}^3v_0^{-\frac{3}{2}}\int_{\R^3}\frac{(|u|^2-2u\cdot v + |v|^2)^{\frac{1}{2}}}{u_0^{\frac{3}{2}}}e^{-\frac{|u|^2}{k_0T}}du\nnm\\
&\geq C\mathbf{c}^3v_0^{-\frac{3}{2}}\int_{0}^{\infty}\int_{\frac{\pi}{2}}^{\pi}\frac{(r^{2}-2r |v|\cos\theta+|v|^2)^{\frac{1}{2}}}{(\mathbf{c}^2+r^2)^{\frac{3}{4}}}e^{-\frac{r^2}{k_0T}}r^2\sin\theta dr d\theta\nnm\\
&\geq C\mathbf{c}^3v_0^{-\frac{3}{2}} \int_{0}^{\infty}\frac{r^3+r^2|v|}{(\mathbf{c}^2+r^2)^{\frac{3}{4}}}e^{-\frac{r^2}{k_0T}}dr\nnm\\
&\geq \nu_2\mathbf{c}^{\frac{3}{2}}v_0^{-\frac{3}{2}}(1 +|v|),\label{rbdl-hbenuvL0}
\ema
where we have used $r|v|\cos\theta\le 0$ for $\theta\in [\pi/2,\pi]$ and \eqref{kkk}.

By \eqref{rbdl-hbnueq}, \eqref{p0} and using $g^2\le 4u_0v_0$ and $ s\le 4u_0v_0$, we have that for $|v|\geq\mathbf{c}$,
\be
\nu(v)\le 4\pi\mathbf{c}\int_{\R^3} p_0 e^{-\frac{\mathbf{c}u_0}{k_0T}}du\le C\mathbf{c}\int_{\R^3}e^{-\frac{|u|}{k_0T}}du\leq \nu_3\mathbf{c},\label{rbdl-hbenuvU}
\ee
where we have used
$$
\int_{\R^3}e^{-\frac{|u|}{k_0T}}du=4\pi\int_{0}^{\infty}e^{-\frac{r}{k_0T}}r^2dr.
$$

For $|v|\leq\mathbf{c}$, we have by \eqref{rbdl-hbggeq} and $ s\le 4u_0v_0$ that
\bma
\nu(v)
&\leq 2\pi\mathbf{c}p_0\int_{\R^3}\frac{|u-v|}{(u_0v_0)^{\frac{1}{2}}}e^{-\frac{\mathbf{c}u_0}{k_0T}}du\nnm\\
&\le C\int_{\R^3}|u-v|e^{-\frac{|u|}{k_0T}}du\leq \nu_3(1+|v|),\label{rbdl-hbenuvU-1}
\ema
where we have used
\be
\int_{\R^3}|u-v|e^{-\frac{|u|}{k_0T}}du \leq\int_{\R^3}(|u|+|v|)e^{-\frac{|u|}{k_0T}}du\leq C(1+|v|).\label{kkk-1}
\ee
By combining \eqref{rbdl-hbenuvL}--\eqref{rbdl-hbenuvU}, we can obtain \eqref{rbdl-ap-hbnugj}. The proof of the lemma is completed.
\end{proof}


Next, we estimate $k_1(u,v)$ and $k_2(u,v)$. Define
\bq\label{rbdlap-lj}
l=:\frac{u_0+v_0}{2k_0T},\quad j=:\frac{|u\times v|}{gk_0T}.
\eq
It can be verify that (cf.\cite{Glassey-2})
\bma
l^2-j^2&=\frac{(u_0v_0-u\cdot v-\mathbf{c}^2)(2\mathbf{c}^2+|u|^2+|v|^2+2u_0v_0)}{2g^2(k_0T)^2}-\frac{|u\times v|^2}{g^2(k_0T)^2}\nnm\\
&=\frac{(|u|^2+|v|^2)\frac{s}{2}+2|u|^2|v|^2-2(u\cdot v)(\frac{s}{2}+u\cdot v)}{2g^2(k_0T)^2}-\frac{|u\times v|^2}{g^2(k_0T)^2}\nnm\\
&=\frac{|u\times v|^2}{g^2(k_0T)^2}+\frac{s|u-v|^2}{4g^2(k_0T)^2}-\frac{|u\times v|^2}{g^2(k_0T)^2}\nnm\\
&=\frac{s|u-v|^2}{4g^2(k_0T)^2}.\label{rbdl-l2j2}
\ema

From \cite{Cao-1,STRIN-1,Wang-2}, we have
\bma
k_1(v,u)&=\frac{\pi\mathbf{c}p_0g\sqrt{s}}{2u_0v_0}e^{-\mathbf{c}l}\int_{0}^{\pi}\sin\theta d\theta,\label{rbdl-k1}\\
k_2(v,u)&=\frac{\pi\mathbf{c}p_0s^{\frac{3}{2}}}{4gu_0v_0}\int_0^{\infty}\frac{y(1+\sqrt{1+y^2})}{\sqrt{1+y^2}}e^{-\mathbf{c}l\sqrt{1+y^2}}I_0(\mathbf{c}jy)dy.\label{rbdl-k2}
\ema


\begin{lem}[\cite{Glassey-2}]\label{rbdlap-l1}
Let $R>r\geq0$ and consider
$$
J_1(R,r)=\int_0^{\infty}ye^{-R\sqrt{1+y^2}}I_0(ry)dy,\quad J_2(R,r)=\int_0^{\infty}y(1+y^2)^{-\frac{1}{2}}e^{-R\sqrt{1+y^2}}I_0(ry)dy.
$$
Thus,
$$
J_1(R,r)=\frac{R}{R^2-r^2}\(1+\frac{1}{\sqrt{R^2-r^2}}\)e^{-\sqrt{R^2-r^2}},\quad J_2(R,r)=\frac{1}{\sqrt{R^2-r^2}}e^{-\sqrt{R^2-r^2}}.
$$
\end{lem}

Let
\bma
&\tilde{K}_{\beta}(l,j;\mathbf{c})=\int_0^{\infty}y(1+y^2)^{\frac{\beta}{4}}e^{-\mathbf{c}l\sqrt{1+y^2}}I_0(\mathbf{c}jy)dy.\label{rbdlap-tkb}
\ema

By Lemma \ref{rbdlap-l1}, we can obtain
\begin{lem}\label{rbdlap-l2}
For $l,j$ defined in \eqref{rbdlap-lj}, $\mathbf{c}\geq1$, $ |\gamma|<1$ and $|\beta|\leq2$, we have
\bma
&\tilde{K}_{\beta}(l,j;\mathbf{c})\leq C\mathbf{c}^{-1}(l^2-j^2)^{-1-\frac{\beta}{4}}l^{1+\frac{\beta}{2}}e^{-\mathbf{c}\sqrt{l^2-j^2}}.\label{rbdlap-l2-1}
\ema
\end{lem}
\begin{proof}
For $\beta<0$, we have by \eqref{rbdlap-tkb} and Lemma \ref{rbdlap-l1} that
\bma
\tilde{K}_{\beta}(l,j;\mathbf{c})&=\int_0^{\infty}y(1+y^2)^{-\frac{|\beta|}{4}}e^{-\mathbf{c}l\sqrt{1+y^2}}I_0(\mathbf{c}jy)dy\nnm\\
&\leq \(\int_0^{\infty}y(1+y^2)^{-\frac{1}{2}}e^{-\mathbf{c}l\sqrt{1+y^2}}I_0(\mathbf{c}jy)dy\)^{\frac{|\beta|}{2}}\(\int_0^{\infty}ye^{-\mathbf{c}l\sqrt{1+y^2}}I_0(\mathbf{c}jy)dy\)^{1-\frac{|\beta|}{2}}\nnm\\
&=\(\frac{e^{-\mathbf{c}\sqrt{l^2-j^2}}}{\mathbf{c}\sqrt{l^2-j^2}}\)^{\frac{|\beta|}{2}}\(\frac{l}{\mathbf{c}l^2-\mathbf{c}j^2}\(1+\frac{1}{\mathbf{c}\sqrt{l^2-j^2}}\)e^{-\mathbf{c}\sqrt{l^2-j^2}}\)^{1-\frac{|\beta|}{2}}\nnm\\
&=\mathbf{c}^{-1}\(\frac{1}{\sqrt{l^2-j^2}}\)^{\frac{|\beta|}{2}}\(\frac{l}{l^2-j^2}\(1+\frac{1}{\mathbf{c}\sqrt{l^2-j^2}}\)\)^{1-\frac{|\beta|}{2}}e^{-\mathbf{c}\sqrt{l^2-j^2}}.\label{rbdlap-l2-1-1}
\ema

Thus, it follows from  \eqref{rbdlap-l2-1-1} that
\be
\tilde{K}_{\beta}(l,j;\mathbf{c}) \leq C\mathbf{c}^{-1}(l^2-j^2)^{-1+\frac{|\beta|}{4}}l^{1-\frac{|\beta|}{2}}e^{-\mathbf{c}\sqrt{l^2-j^2}},
\ee
which implies that \eqref{rbdlap-l2-1} for $\beta<0$.

For $\beta\ge 0$, we have
\bma
\tilde{K}_{\beta}(l,j;\mathbf{c})&=\int_0^{\infty}y(1+y^2)^{\frac{|\beta|}{4}}e^{-\mathbf{c}l\sqrt{1+y^2}}I_0(\mathbf{c}jy)dy\nnm\\
&\leq \(\int_0^{\infty}y(1+y^2)^{\frac{1}{2}}e^{-\mathbf{c}l\sqrt{1+y^2}}I_0(\mathbf{c}jy)dy\)^{\frac{|\beta|}{2}}\(\int_0^{\infty}ye^{-\mathbf{c}l\sqrt{1+y^2}}I_0(\mathbf{c}jy)dy\)^{1-\frac{|\beta|}{2}}.\label{rbdl-tkbg0-1}
\ema
Let
$$G(z)=\int_{z}^{\infty}\int_0^{\infty}y(1+y^2)^{\frac{1}{2}}e^{-r\sqrt{1+y^2}}I_0(\mathbf{c}jy)dydr.$$
By Lemma \ref{rbdlap-l1}, it holds that for $z>\mathbf{c} j$,
\bma
G(z)
&=\int_0^{\infty}yI_0(\mathbf{c}jy)\int_{z}^{\infty}e^{-r\sqrt{1+y^2}}d(r(1+y^2)^{\frac{1}{2}})dy\nnm\\
&=\int_0^{\infty}ye^{-z\sqrt{1+y^2}}I_0(\mathbf{c}jy)dy\nnm\\
&=\frac{z}{z^2-\mathbf{c}^2j^2}\(1+\frac{1}{\sqrt{z^2-\mathbf{c}^2j^2}}\)e^{-\sqrt{z^2-\mathbf{c}^2j^2}}. \label{gz}
\ema
Thus, we have by \eqref{gz} that
\bma
G'(\mathbf{c}l)&= -\int_0^{\infty}y(1+y^2)^{\frac{1}{2}}e^{-\mathbf{c}l\sqrt{1+y^2}}I_0(\mathbf{c}jy)dy\nnm\\
&=\frac{\big[\big((\mathbf{c}^2l^2-\mathbf{c}^2j^2)+3(\mathbf{c}^2l^2-\mathbf{c}^2j^2)^{\frac{1}{2}}+3\big)-\frac{(l^2-j^2)}{l^2}-\frac{\mathbf{c}(l^2-j^2)^{\frac{3}{2}}}{l^2}\big]}{\mathbf{c}^3(l^2-j^2)^{\frac{5}{2}}}l^2e^{-\mathbf{c}\sqrt{l^2-j^2}}\nnm\\
&\leq C\mathbf{c}^{-1}(l^2-j^2)^{-\frac{3}{2}}l^2e^{-\mathbf{c}\sqrt{l^2-j^2}},
\ema
which together with \eqref{rbdl-tkbg0-1} implies that
\bma
\tilde{K}_{\beta}(l,j;\mathbf{c})&\leq C\(\frac{l^2}{\mathbf{c}(l^2-j^2)^{\frac{3}{2}}}e^{-\mathbf{c}\sqrt{l^2-j^2}}\)^{\frac{\beta}{2}}\(\frac{l}{\mathbf{c}(l^2-j^2) }e^{-\mathbf{c}\sqrt{l^2-j^2}}\)^{1-\frac{\beta}{2}}\nnm\\
&\leq C\mathbf{c}^{-1}(l^2-j^2)^{-1-\frac{\beta}{4}}l^{1+\frac{\beta}{2}}e^{-\mathbf{c}\sqrt{l^2-j^2}}.
\ema
The proof is completed.
\end{proof}

\begin{lem}\label{rbdl-estk1k2}
For $\mathbf{c}\geq1$, there exists a constant $C>0$ such that
\bma
k_1(v,u)&\leq C \frac{\mathbf{c}|u-v|s^{\frac{1}{2}}}{u_0v_0} e^{-\frac{|u|+|v|}{2k_0T}},\label{rbdl-hbk1ineq}\\
k_2(v,u)&\leq C\frac{(u_0+v_0)s^{\frac{1}{2}}}{u_0v_0|u-v|}E(u,v),\label{rbdl-hbk2ineq}
\ema
where
\be
E(u,v)=e^{-\frac{|u-v|}{4k_0T}}e^{-\frac{\mathbf{c}^3}{4(u_0v_0+\mathbf{c}^2)^2 }\frac{(|u|^2-|v|^2)^2}{|u-v|^2\sqrt{4\mathbf{c}^2+|u-v|^2} }\(\mathbf{c}+\frac{|u|^2|v|^2}{(v_0+\mathbf{c})(u_0+\mathbf{c})(u_0+v_0)}\)^2}.\label{rbdl-euv}
\ee
\end{lem}
\begin{proof}
First, we prove \eqref{rbdl-hbk1ineq}. By \eqref{cu2}, \eqref{rbdl-hbggeq} and \eqref{rbdl-k1}, we have
\bma
k_1(v,u)&=\frac{\pi\mathbf{c}p_0g\sqrt{s}}{2u_0v_0}e^{-\mathbf{c}l}\int_{0}^{\pi}\sin\theta d\theta= \frac{\pi\mathbf{c}g\sqrt{s}}{u_0v_0}p_0e^{-\frac{\mathbf{c}(u_0+v_0)}{2k_0T}}\nnm\\
&\leq C\mathbf{c}\frac{ |u-v|s^{\frac{1}{2}}}{u_0v_0}e^{-\frac{|u|+|v|}{2k_0T}}.
\ema
This gives \eqref{rbdl-hbk1ineq}.

Next, we prove \eqref{rbdl-hbk2ineq}. On the one hand, we have by \eqref{rbdl-hbggeq} and \eqref{rbdl-l2j2} that
\be
l^2-j^2=\frac{s|u-v|^2}{4g^2(k_0T)^2}=\frac{ |u-v|^2+4\mathbf{c}^2\frac{|u-v|^2}{g^2}}{4(k_0T)^2}\geq\frac{4\mathbf{c}^2+|u-v|^2}{4(k_0T)^2}.\label{rbdl-l-jxjc}
\ee
This and \eqref{rbdl-p0eu0} implies that
\be
-\frac{|u-v|^2}{4}\leq \mathbf{c}^2-k_0T\mathbf{c}\sqrt{l^2-j^2}\leq \mathbf{c}^2-\frac{\mathbf{c}\sqrt{4\mathbf{c}^2+|u-v|^2}}{2}\leq  -\frac{|u-v|}{2}.\label{rbdl-p0eu-v}
\ee

On the other hand,
\bma
&\quad\mathbf{c}^2-k_0T\mathbf{c}\sqrt{l^2-j^2}\nnm\\
&=\mathbf{c}^2-\frac{\mathbf{c}\sqrt{s}}{2g}|u-v|=\frac{\mathbf{c}^4-\frac{\mathbf{c}^2(g^2+4\mathbf{c}^2)}{4g^2}|u-v|^2}{\mathbf{c}^2+\frac{\mathbf{c}\sqrt{s}}{2g}|u-v|}\nnm\\
&=\frac{-\frac{1}{4}|u-v|^2+\mathbf{c}^2-\frac{\mathbf{c}^2}{g^2}|u-v|^2}{1+\frac{\sqrt{s}}{2\mathbf{c}g}|u-v|}\leq\frac{\frac{\mathbf{c}^2}{g^2}\(g^2-|u-v|^2\)}{1+\frac{\sqrt{s}}{2\mathbf{c}g}|u-v|}.\label{rbdl-vvv-1}
\ema
By \eqref{rbdl-vmgs}, we have (cf.\cite{Wang-2})
\bma
g^2-|u-v|^2&=2u_0v_0-2u\cdot v-2\mathbf{c}^2-|u-v|^2\nnm\\
&=\frac{2(|u|^2+\mathbf{c}^2)(|v|^2+\mathbf{c}^2)-2\mathbf{c}^4}{u_0v_0+\mathbf{c}^2}-|u|^2-|v|^2\nnm\\
&=\frac{2|u|^2|v|^2+\mathbf{c}^2(|u|^2+|v|^2)-(|u|^2+|v|^2)u_0v_0}{u_0v_0+\mathbf{c}^2}\nnm\\
&=\frac{2|u|^2|v|^2(u_0v_0+\mathbf{c}^2)-(|u|^2+|v|^2)(|u|^2|v|^2+\mathbf{c}^2(|v|^2+|u|^2))}{(u_0v_0+\mathbf{c}^2)^2}\nnm\\
&=-\frac{(|u|^2v_0-|v|^2u_0)^2}{(u_0v_0+\mathbf{c}^2)^2}.\label{rbdl-vvv-2}
\ema
It can be verify that
\bma
|u|^2v_0-|v|^2u_0&=\mathbf{c}|u|^2-\mathbf{c}|v|^2+|u|^2(v_0-\mathbf{c})-|v|^2(u_0-\mathbf{c})\nnm\\
&=\mathbf{c}|u|^2-\mathbf{c}|v|^2+\frac{|u|^2|v|^2}{(v_0+\mathbf{c})}-\frac{|v|^2|u|^2}{(u_0+\mathbf{c})}\nnm\\
&=\mathbf{c}|u|^2-\mathbf{c}|v|^2+\frac{|u|^2|v|^2(|u|^2-|v|^2)}{(v_0+\mathbf{c})(u_0+\mathbf{c})(u_0+v_0)}.\label{rbdl-vvv-3}
\ema
By combining \eqref{rbdl-hbggeq} and \eqref{rbdl-vvv-1}--\eqref{rbdl-vvv-3} and $s\le 4u_0v_0$, we have
\bma
&\quad\mathbf{c}^2-k_0T\mathbf{c}\sqrt{l^2-j^2}\nnm\\
&\leq-\frac{\mathbf{c}^2(|u|^2-|v|^2)^2}{(u_0v_0+\mathbf{c}^2)^2(g^2+\frac{g\sqrt{s}|u-v|}{2\mathbf{c}})}\(\mathbf{c}+\frac{|u|^2|v|^2}{(v_0+\mathbf{c})(u_0+\mathbf{c})(u_0+v_0)}\)^2\nnm\\
&\leq-\frac{\mathbf{c}^2}{(u_0v_0+\mathbf{c}^2)^2}\frac{(|u|^2-|v|^2)^2}{|u-v|^2(1+\frac{\sqrt{s}}{2\mathbf{c}}) }\(\mathbf{c}+\frac{|u|^2|v|^2}{(v_0+\mathbf{c})(u_0+\mathbf{c})(u_0+v_0)}\)^2\nnm\\
&\leq-\frac{\mathbf{c}^3}{2(u_0v_0+\mathbf{c}^2)^2 }\frac{(|u|^2-|v|^2)^2}{|u-v|^2 \sqrt{4\mathbf{c}^2+ |u-v|^2 }}\(\mathbf{c}+\frac{|u|^2|v|^2}{(v_0+\mathbf{c})(u_0+\mathbf{c})(u_0+v_0)}\)^2.\label{rbdl-vvv-5}
\ema
Thus, we have by \eqref{p0}, \eqref{rbdl-p0eu-v} and \eqref{rbdl-vvv-5} that
$$
p_0e^{-\mathbf{c}\sqrt{l^2-j^2}}\leq CE(u,v).
$$

By \eqref{rbdl-l2j2}, \eqref{rbdl-k2}, \eqref{rbdl-p0eu-v} and Lemma \ref{rbdlap-l2}, we have that for $\mathbf{c}\geq1$,
\bma
k_2(v,u)&=\frac{\pi\mathbf{c}p_0s^{\frac{3}{2}}}{4gu_0v_0}\int_0^{\infty}\frac{y(1+\sqrt{1+y^2})}{\sqrt{1+y^2}}e^{-\mathbf{c}l\sqrt{1+y^2}}I_0(\mathbf{c}jy)dy\nnm\\
&\leq\frac{\pi\mathbf{c}p_0s^{\frac{3}{2}}}{2gu_0v_0}\int_0^{\infty}ye^{-\mathbf{c}l\sqrt{1+y^2}}I_0(\mathbf{c}jy)dy\nnm\\
&\leq C\frac{s^{\frac{3}{2}}(l^2-j^2)^{-1}l}{gu_0v_0}p_0e^{-\mathbf{c}\sqrt{l^2-j^2}}\nnm\\
&\leq C\frac{s^{\frac{1}{2}}(u_0+v_0)}{u_0v_0|u-v|}E(u,v),
\ema
which implies that \eqref{rbdl-hbk2ineq}. The proof of the lemma is completed.
\end{proof}

\begin{lem}\label{rbdl-prok1k2lem}
Letting $m>-3$, $r\geq0$, $\alpha\in\R$  and $ p, q\geq0$, it holds that
\bma
&\quad\int_{\R^3}\frac{|u-v|^{m}s^{p}u_0^{\alpha}}{(1+|u|)^{r}}e^{-\frac{|u|+|v|}{2k_0T}}du\leq C\mathbf{c}^{2p}v_0^{\alpha}e^{-\frac{|v|}{4k_0T}},\label{rbdl-prok1k2lem-1}\\
&\quad\int_{\R^3}\frac{|u-v|^{m}s^{p}(u_0+v_0)^{q}u_0^{\alpha}}{(1+|u|)^{r}}E(u,v)du\leq C\mathbf{c}^{2p}v_0^{q+\alpha}
(1+|v|)^{-r}w(v)^{-1}, \label{rbdl-prok1k2lem-2}
\ema
where $E(u,v)$ is given in \eqref{rbdl-euv}, and $w(v)$ is given in \eqref{rbdl-wccc}.
\end{lem}
\begin{proof}
 Firstly, we want to prove \eqref{rbdl-prok1k2lem-1}. Note that
\be
u_0\leq\sqrt{2}|u-v|+\sqrt{2}v_0,\quad v_0\leq\sqrt{2}|u-v|+\sqrt{2}u_0.\label{rbdl-u0v0ineq}
\ee
Thus, it holds that
\bmas &u_0^{ \alpha}\le (u_0+v_0)^{ \alpha}= v_0^ \alpha\(1+\frac{u_0}{v_0}\)^ \alpha \le  Cv_0^{ \alpha}\(\frac{|u-v|}{v_0}+ 1\)^{ \alpha },\quad  \alpha>0,\\
&(u_0+v_0)^{ \alpha}\le u_0^{ \alpha}=v_0^{ \alpha}\(\frac{v_0}{u_0}\)^{- \alpha}\le  Cv_0^{ \alpha}\(\frac{|u-v|}{u_0}+ 1\)^{- \alpha },\quad  \alpha\le 0.
\emas

By using $u_0,v_0\ge \mathbf{c}$, we can obtain
\be u_0^{ \alpha},(u_0+v_0)^{ \alpha}\le Cv_0^{ \alpha}\(\frac{|u-v|}{\mathbf{c}}+ 1\)^{| \alpha| }. \label{rbdl-uuu}\ee
Thus, we have that for $\mathbf{c}\geq1$,
\bma
&\quad\int_{\R^3}\frac{|u-v|^{m}s^{p}u_0^{\alpha}}{(1+|u|)^{r}}e^{-\frac{|u|+|v|}{2k_0T}}du\nnm\\
&\leq e^{-\frac{|v|}{4k_0T}}\int_{\R^3}|u-v|^{m}s^{p}u_0^{\alpha}e^{-\frac{|u-v|}{4k_0T}}du\nnm\\
&\leq \mathbf{c}^{2p}v_0^{\alpha}e^{-\frac{|v|}{4k_0T}}\int_{\R^3}|u-v|^{m}(|u-v|^2+1)^{p+\frac{|\alpha|}{2}}e^{-\frac{|u-v|}{4k_0T}}du\nnm\\
&\leq C\mathbf{c}^{2p}v_0^{\alpha}e^{-\frac{|v|}{4k_0T}},\label{rbdl-prok1k2lem-1-1}
\ema
where we have used
\bma
&\quad\int_{\R^3}|u-v|^{m}(|u-v|^2+1)^{p+\frac{|\alpha|}{2}}e^{-\frac{|u-v|}{4k_0T}}du\nnm\\
&=4\pi\int_{0}^{\infty}r^{m+2}(r^2+1)^{p+\frac{|\alpha|}{2}}e^{-\frac{r}{4k_0T}}du\leq C,\quad m>-3.\label{rbdl-uuu2}
\ema
This proves  \eqref{rbdl-prok1k2lem-1}.

Now, we want to show \eqref{rbdl-prok1k2lem-2}. By \eqref{rbdl-euv}, \eqref{rbdl-uuu}  and \eqref{rbdl-uuu2}, we have for $|u-v|\geq\frac{|v|}{2}$,
\bma
&\quad\int_{\{|u-v|\geq\frac{|v|}{2}\}}\frac{|u-v|^{m}s^{p}(u_0+v_0)^{q}u_0^{\alpha}}{(1+|u|)^{r}}E(u,v)du\nnm\\
&=\int_{\{|u-v|\leq\frac{|v|}{2}\}}\frac{|u-v|^{m}s^{p}(u_0+v_0)^{q}u_0^{\alpha}}{(1+|u|)^{r}}e^{-\frac{|u-v|}{4k_0T}}\nnm\\
&\qquad\times e^{-\frac{\mathbf{c}^3}{4(u_0v_0+\mathbf{c}^2)^2 }\frac{(|u|^2-|v|^2)^2}{|u-v|^2\sqrt{4\mathbf{c}^2+|u-v|^2} }\(\mathbf{c}+\frac{|u|^2|v|^2}{(v_0+\mathbf{c})(u_0+\mathbf{c})(u_0+v_0)}\)^2}du\nnm\\
&\leq \mathbf{c}^{2p}v_0^{q+\alpha}e^{-\frac{|v|}{32k_0T}}\int_{\{|u-v|\geq\frac{|v|}{2}\}}|u-v|^{m}(|u-v|^2+1)^{p+\frac{q+|\alpha|}{2}}e^{-\frac{|u-v|}{8k_0T}}du\nnm\\
&\leq C\mathbf{c}^{2p}v_0^{q+\alpha}e^{-\frac{|v|}{32k_0T}}.\label{rbdl-prok1k2lem-2-1}
\ema
For $|u-v|\leq\frac{|v|}{2}$ and $|v|\geq\mathbf{c}$, we have by \eqref{rbdl-uuu}  and \eqref{rbdl-uuu2} that
\bma
&\quad\int_{\{|u-v|\leq\frac{|v|}{2}\}}\frac{|u-v|^{m}s^{p}(u_0+v_0)^{q}u_0^{\alpha}}{(1+|u|)^{r}}E(u,v)du\nnm\\
&\le C(1+|v|)^{-r}\int_{\{|u-v|\leq\frac{|v|}{2}\}}|u-v|^{m}s^{p}(u_0+v_0)^{q}u_0^{\alpha}e^{-\frac{|u-v|}{4k_0T}}\nnm\\
&\qquad \times e^{-\frac{\mathbf{c}^3}{4k_0T(u_0v_0+\mathbf{c}^2)^2 }\frac{(|u|^2-|v|^2)^2}{|u-v|^2\sqrt{4\mathbf{c}^2+|u-v|^2} }\(\mathbf{c}+\frac{|u|^2|v|^2}{(v_0+\mathbf{c})(u_0+\mathbf{c})(u_0+v_0)}\)^2}du\nnm\\
&\leq C\mathbf{c}^{2p}(1+|v|)^{-r}v_0^{q+\alpha}\int_{\{|u-v|\leq\frac{|v|}{2}\}}|u-v|^{m}(|u-v|^2+1)^{p+\frac{q+|\alpha|}{2}}e^{-\frac{|u-v|}{4k_0T}}\nnm\\
&\qquad \times e^{-\frac{\mathbf{c}^2|u|^4|v|^4}{4k_0T(u_0v_0+\mathbf{c}^2)^2(v_0+\mathbf{c})^2(u_0+\mathbf{c})^2(u_0+v_0)^2 }\frac{\mathbf{c}(|u|^2-|v|^2)^2}{ |u-v|^2\sqrt{4\mathbf{c}^2+|u-v|^2}}}du\nnm\\
&\leq C\mathbf{c}^{2p}(1+|v|)^{-r}v_0^{q+\alpha}\int_{\{|u-v|\leq\frac{|v|}{2}\}}|u-v|^{m}(|u-v|^2+1)^{p+\frac{q+|\alpha|}{2}}e^{-\frac{|u-v|}{4k_0T}}e^{- \frac{d_1\mathbf{c}^3}{|v|^2}\frac{(|u|^2-|v|^2)^2}{ |u-v|^2\sqrt{4\mathbf{c}^2+|u-v|^2}}}du\nnm\\
&=2\pi C\mathbf{c}^{2p}(1+|v|)^{-r}v_0^{q+\alpha}\int_{0}^{\frac{|v|}{2}}r^{m+2}(r^2+1)^{p+\frac{q+|\alpha|}{2}}e^{-\frac{r}{4k_0T}}\(\int_0^{\pi}e^{- \frac{d_1\mathbf{c}^3}{|v|^2\sqrt{4\mathbf{c}^2+r^2}}(r+2|v|\cos\theta)^2}\sin\theta d\theta \)dr\nnm\\
&\leq C \mathbf{c}^{2p-1}(1+|v|)^{-r}v_0^{q+\alpha}\int_{0}^{\infty}r^{m+2}(r^2+1)^{p+\frac{q+|\alpha|}{2}}\(\frac{r^2}{\mathbf{c}^2}+4\)^{\frac14}e^{-\frac{r}{4k_0T}}dr\nnm\\
&\leq C\mathbf{c}^{2p-1}v_0^{q+\alpha}(1+|v|)^{-r},\label{rbdl-prok1k2lem-2-2}
\ema
where we have used $\frac{|v|}{2}\le |u|\le \frac{3|v|}{2}$,  and
\bma
&\quad\frac{\mathbf{c}^2|u|^4|v|^4}{4k_0T(u_0v_0+\mathbf{c}^2)^2(v_0+\mathbf{c})^2(u_0+\mathbf{c})^2(u_0+v_0)^2 }\nnm\\
&\geq\frac{\mathbf{c}^2|u|^4|v|^4}{4^3k_0T (u_0v_0+v_0^2)^2 v_0^2u_0^2(u_0+v_0)^2 }\ge d_1\frac{\mathbf{c}^2}{|v|^2},\quad |v|\geq\mathbf{c},~~ |u|\geq\frac{\mathbf{c}}{2}.\label{rbdl-prok1k2lem-2-2-aaa}
\ema

For $|u-v|\leq\frac{|v|}{2}$ and $1\leq|v|\leq\mathbf{c}$, we have by \eqref{rbdl-uuu} and \eqref{rbdl-uuu2} that
\bma
&\quad\int_{\{|u-v|\leq\frac{|v|}{2}\}}\frac{|u-v|^{m}s^{p}(u_0+v_0)^{q}u_0^{\alpha}}{(1+|u|)^{r}}E(u,v)du\nnm\\
&\leq C\mathbf{c}^{2p}v_0^{q+\alpha}(1+|v|)^{-r}\int_{\{|u-v|\leq\frac{|v|}{2}\}}|u-v|^{m}(|u-v|^2+1)^{p+\frac{q+|\alpha|}{2}}e^{-\frac{|u-v|}{4k_0T}}\nnm\\
&\qquad \times e^{-\frac{\mathbf{c}^5}{4k_0T(u_0v_0+\mathbf{c}^2)^2\sqrt{4\mathbf{c}^2+|u-v|^2}}\frac{(|u|^2-|v|^2)^2}{ |u-v|^2}}du\nnm\\
&\leq C\mathbf{c}^{2p}v_0^{q+\alpha}(1+|v|)^{-r}\int_{\{|u-v|\leq\frac{|v|}{2}\}}|u-v|^{m}(|u-v|^2+1)^{p+\frac{q+|\alpha|}{2}}e^{-\frac{|u-v|}{4k_0T}}e^{-d_2\frac{(|u|^2-|v|^2)^2}{|u-v|^2}}du\nnm\\
&=2\pi C\mathbf{c}^{2p}v_0^{q+\alpha}(1+|v|)^{-r}\int_{0}^{\frac{|v|}{2}}r^{m+2}(r^2+1)^{p+\frac{q+|\alpha|}{2}}e^{-\frac{r}{4k_0T}}\(\int_0^{\pi}e^{-d_2(r+2|v|\cos\theta)^2}\sin\theta d\theta \)dr\nnm\\
&\leq C \mathbf{c}^{2p}v_0^{q+\alpha}|v|^{-1}(1+|v|)^{-r}\int_{0}^{\infty}r^{m+2}(r^2+1)^{p+\frac{q+|\alpha|}{2}}e^{-\frac{r}{4k_0T}}dr\nnm\\
&\leq C\mathbf{c}^{2p}v_0^{q+\alpha}|v|^{-1}(1+|v|)^{-r},\label{rbdl-prok1k2lem-2-3}
\ema
where we have used 
$$
 \frac{\mathbf{c}^5}{4k_0T(u_0v_0+\mathbf{c}^2)^2\sqrt{4\mathbf{c}^2+|u-v|^2}}>d_2,\quad |u|\leq\frac{3\mathbf{c}}{2},~~ |v|\leq\mathbf{c}.
$$

For $|u-v|\leq\frac{|v|}{2}$ and $|v|\leq1$, we have by \eqref{rbdl-uuu}, \eqref{rbdl-uuu2} and \eqref{rbdl-prok1k2lem-2-3} that
\bma
&\quad\int_{\{|u-v|\leq\frac{|v|}{2}\}}\frac{|u-v|^{m}s^{p}(u_0+v_0)^{q}u_0^{\alpha}}{w(u)^{r}}E(u,v)du\nnm\\
&\leq \mathbf{c}^{2p}v_0^{q+\alpha}\int_{\{|u-v|\leq\frac{|v|}{2}\}}\frac{|u-v|^{m}(|u-v|^2+1)^{p+\frac{q+|\alpha|}{2}}}{(1+|u|)^{r}}e^{-\frac{|u-v|}{4k_0T}}du\nnm\\
&\leq C \mathbf{c}^{2p}v_0^{q+\alpha}(1+|v|)^{-r}\int_{0}^{\infty}r^{m+2}(r^2+1)^{p+\frac{q+|\alpha|}{2}}e^{-\frac{r}{4k_0T}}dr\nnm\\
&\leq C\mathbf{c}^{2p}v_0^{q+\alpha}(1+|v|)^{-r}.\label{rbdl-prok1k2lem-2-4}
\ema
This and \eqref{rbdl-prok1k2lem-2-1}--\eqref{rbdl-prok1k2lem-2-3} implies that \eqref{rbdl-prok1k2lem-2}. The proof of the lemma is completed.
\end{proof}

With help of Lemmas \ref{rbdl-estk1k2}--\ref{rbdl-prok1k2lem}, we have some integral properties of $k(v,u)$ as follows.
\begin{lem}\label{rbdl-svku1alfa}
Given $\alpha\ge 0$. There exists a constant $C>0$ independent of $\mathbf{c}$ so that
\bma
&\int_{\R^3}|k(v,u)|^2du\leq C(1+|v|)^{-1},\label{rbdl-2svku1}\\
&\int_{\R^3}|k(v,u)|(1+|u|)^{-\alpha}du\leq C(1+|v|)^{-\alpha-1}.\label{rbdl-svkualfa}
\ema
\end{lem}
\begin{proof}
By \eqref{rbdl-ap-K} and Lemmas \ref{rbdl-estk1k2}--\ref{rbdl-prok1k2lem}, we have
\bma
\int_{\R^3}|k(v,u)|(1+|u|)^{-\alpha}du&\leq \int_{\R^3}|k_1(u,v)|(1+|u|)^{-\alpha}du+\int_{\R^3}|k_2(u,v)|(1+|u|)^{-\alpha}du\nnm\\
&\leq C\mathbf{c}v_0^{-1}\int_{\R^3}|u-v|s^{\frac{1}{2}}u_0^{-1}(1+|u|)^{-\alpha}e^{-\frac{|u|+|v|}{2k_0T}}du\nnm\\
&\quad+Cv_0^{-1}\int_{\R^3}\frac{s^{\frac{1}{2}}(u_0+v_0)u_0^{-1}}{(1+|u|)^{\alpha}|u-v|} E(u,v)du\nnm\\
&\leq C\(\mathbf{c}^{2}v_0^{-2}e^{-\frac{|v|}{4k_0T}}+\mathbf{c}v_0^{-1}(1+|v|)^{-\alpha}w(v)^{-1}\)\nnm\\
&\leq C\(e^{-\frac{|v|}{4k_0T}}+(1+|v|)^{-\alpha-1}\),\label{rbdl-hbsvku1-1}
\ema
where we have used $\mathbf{c}v_0^{-1}w(v)^{-1}\leq (1+|v|)^{-1}$ for $|v|\leq \mathbf{c}$ and $\mathbf{c}v_0^{-1}w(v)^{-1}\leq v_0^{-1} $ for $|v|\geq \mathbf{c}$ with $w(v)$ given in \eqref{rbdl-wccc}. Thus, we can obtain \eqref{rbdl-svkualfa}.

By \eqref{rbdl-ap-K} and Lemmas \ref{rbdl-estk1k2}--\ref{rbdl-prok1k2lem},  we have
\bma
\int_{\R^3}|k(v,u)|^2du&\leq 2\int_{\R^3}|k_1(u,v)|^2du+2\int_{\R^3}|k_2(u,v)|^2du\nnm\\
&\leq C\mathbf{c}^2v_0^{-2}\int_{\R^3}|u-v|^2su_0^{-2}e^{-\frac{|u|+|v|}{k_0T}}du\nnm\\
&\quad+Cv_0^{-2}\int_{\R^3}\frac{s(u_0+v_0)^2u_0^{-2}}{|u-v|^2} E(u,v)^2du\nnm\\
&\leq C \(\mathbf{c}^{4}v_0^{-4}e^{-\frac{|v|}{2k_0T}}+\mathbf{c}^{2}v_0^{-2}w(v)^{-1}\)\nnm\\
&\leq C \(e^{-\frac{|v|}{2k_0T}}+(1+|v|)^{-1}\).\label{rbdl-hb2svku1-1}
\ema
This proves \eqref{rbdl-2svku1}. The proof of the lemma is completed.
\end{proof}


\begin{lem}\label{rbegf8j2}
For any $k>0$ and $f\in L^2(\R^3_v)\cap L^{\infty}_{v,k-1}$, it holds that
\bma
&\|Kf\|_{L^{\infty}_{v,0}}\leq C\|f\|_{L^2_v},\label{rbegf8j2-1}\\
&\|Kf\|_{L^{\infty}_{v,k}}\leq C\|f\|_{L^{\infty}_{v,k-1}}.\label{rbegf8j2-2}
\ema
\end{lem}
\begin{proof}
By \eqref{rbdl-ap-K} and Lemma \ref{rbdl-svku1alfa}, it follows that
\bma
\|Kf\|_{L^{\infty}_{v,0}}&\leq\sup_{v\in\R^3}\int_{\R^3}|k(v,u)f|du\nnm\\
&\leq \sup_{v\in\R^3}\(\int_{\R^3}|k(v,u)|^2du\)^{\frac12}\|f\|_{L^2_v}\leq C\|f\|_{L^2_v}.
\ema
This proved \eqref{rbegf8j2-1}.

By \eqref{rbdl-ap-K} and Lemma \ref{rbdl-svku1alfa}, we have
\bma
\|Kf\|_{L^{\infty}_{v,k}}&\leq \sup_{v\in\R^3}(1+|v|)^{k}\int_{\R^3}|k(v,u)f|du\nnm\\
&\leq \(\sup_{v\in\R^3}(1+|v|)^{k}\int_{\R^3}|k(v,u)|(1+|u|)^{-k+1}du\)\(\sup_{u\in\R^3}(1+|u|)^{k-1}|f(u)|\)\nnm\\
&\leq C\|f\|_{L^{\infty}_{v,k-1}}.
\ema
The Lemma is proved.
\end{proof}

\begin{lem}\label{dl-rbe-estgamfg}
For any given $k\geq0$, $q\geq1$ and  $l\in [0,1]$, there exists a constant $C>0$ such that 
\bma
&\|\nu^{-1}\Gamma(f,h)\|_{L^{\infty}_{v,k}(L^q_x)}\leq C\|f\|_{L^{\infty}_{v,k}(L^r_x)}\|h\|_{L^{\infty}_{v,k}(L^s_x)},\label{dl-rbe-estgamfg-1}\\
&\|\nu^{-l}\Gamma(f,h)\|_{L^{2}_{v}(L^q_x)}\leq C\(\|w^{1-l}f\|_{L^2_v(L^r_x)}\|h\|_{L^2_v(L^s_x)}+\|f\|_{L^2_v(L^r_x)}\|w^{1-l}h\|_{L^2_v(L^s_x)}\),\label{dl-rbe-estgamfg-2}
\ema
 where $r>0,s>0$ satisfies $\frac{1}{r}+\frac{1}{s}=\frac{1}{q}$.
\end{lem}

\begin{proof}
By \eqref{rbepz-1}, we have
\bmas
\Gamma(f,h)&= \int_{\R^3}\int_{\S^2}v_M \sqrt{M(u)}\big( f(v') h(u')+ f(u') h(v')- f(v) h(u)-f(u) h(v)\big)d\omega du\\
&=:I_1+I_2-I_3-I_4.
\emas

Firstly, we prove \eqref{dl-rbe-estgamfg-1} and only estimate $I_1$ as follows.
From Lemma \ref{ap-est-nuv}, one has
\bma
&\quad\|\nu^{-1}(1+|v|)^{k}I_1\|_{L^q_x}\nnm\\
&\leq\nu(v)^{-1}\int_{\R^3}\int_{\S^2}v_M\sqrt{M(u)} (1+|v|)^{k}\|f(v')h(u')\|_{L^q_x}d\omega du\nnm\\
&\leq\nu(v)^{-1}\int_{\R^3}\int_{\S^2}v_M\sqrt{M(u)}(1+|v'|)^{k}\|f(v')\|_{L^r_x}(1+|u'|)^{k}\|h(u')\|_{L^s_x}d\omega du\nnm\\
&\leq C\|f\|_{L^{\infty}_{v,k}(L^r_x)}\|h\|_{L^{\infty}_{v,k}(L^r_x)}\nu(v)^{-1}\int_{\R^3}\int_{\S^2}v_M\sqrt{M(u)}d\omega du\nnm\\
&\leq C\|f\|_{L^{\infty}_{v,k}(L^r_x)}\|h\|_{L^{\infty}_{v,k}(L^r_x)},\label{dl-rbe-estgamfg-1-1-1}
\ema
where $\frac{1}{r}+\frac{1}{s}=\frac{1}{q}$, and we have used (cf. \cite{Wang-2})
$$
(1+|v|)\leq(1+|v'|)(1+|u'|).
$$
Similarly, we can show that  $I_j$ $(j=2,3,4)$ also satisfies \eqref{dl-rbe-estgamfg-1-1-1}.

Finally, we prove \eqref{dl-rbe-estgamfg-2} and only estimate $I_1$ as follows. By Lemma \ref{ap-est-nuv}, change of the variables $(u',v')\rightarrow(u,v)$ and the Jacobian \cite{Glassey-3}
$$
\frac{\partial(u',v')}{\partial(u,v)}=-\frac{u_0'v_0'}{u_0v_0},
$$
we have that for $0\le l\le 1$,
\bma
\|\nu^{-l}I_1\|^2_{L^2_v(L^q_x)}&\leq\int_{\R^3}w(v)^{-2l}\(\int_{\R^3}\(\int_{\R^3}\int_{\S^2}v_M\sqrt{M(u)}f(v')h(u')d\omega du\)^{q} dx\)^{\frac{2}{q}}dv\nnm\\
&\leq\int_{\R^3}w(v)^{-2l}\(\int_{\R^3}\int_{\S^2}v_M\sqrt{M(u)}\(\int_{\R^3}|f(v')|^q||h(u')|^qdx\)^{\frac{1}{q}}d\omega du\)^{2}dv\nnm\\
&\leq \int_{\R^3}w(v)^{-2l}\(\int_{\R^3}\int_{\S^2}v_M^2\sqrt{M(u)}d\omega du\)\(\int_{\R^3}\int_{\S^2}\sqrt{M(u)}\|f(v')h(u')\|^2_{L^q_x}d\omega du\)dv \nnm\\
&\leq C\int_{\R^3} w(v)^{2-2l}\int_{\R^3}\int_{\S^2}\sqrt{M(u)}\|f(v')\|^2_{L^r_x}\|h(u')\|^2_{L^s_x}d\omega dudv\nnm\\
&=C\int_{\R^3}\int_{\R^3}\int_{\S^2}w(v')^{2-2l}\|f(v)\|_{L^r_x}^2\|h(u)\|^2_{L^s_x}\frac{\sqrt{M(u')}u_0'v_0'}{u_0v_0}d\omega dudv\nnm\\
&\leq C\(\|w^{1-l}f\|^2_{L^2_v(L^r_x)}\|h\|^2_{L^2_v(L^s_x)}+\|f\|^2_{L^2_v(L^r_x)}\|w^{1-l}h\|^2_{L^2_v(L^s_x)}\),\label{dl-rbe-estgamfg-2-1}
\ema
where $\frac{1}{r}+\frac{1}{s}=\frac{1}{q}$, and we have used $w(v')\leq C(w(v)+w(u))$ and
\bmas
(v_0u_0)^2&\geq\mathbf{c}^2(\mathbf{c}^2+|u|^2+|v|^2)\ge\frac{\mathbf{c}^2}{2}(2\mathbf{c}^2+|u|^2+|v|^2)\nnm\\
&\geq\frac{\mathbf{c}^2}{4}(u_0+v_0)^2=\frac{\mathbf{c}^2}{4}(u'_0+v'_0)^2\geq\frac{\mathbf{c}^2}{4}(v'_0)^2,\\
 \sqrt{M(u')}u_0'&\leq C e^{-\frac{|u'|}{2k_0T}}u_0'\leq C\mathbf{c} ,\quad \mathbf{c}\gg1.
\emas
Similarly, we can show that  $I_j$ $(j=2,3,4)$ also satisfies \eqref{dl-rbe-estgamfg-2-1}. The proof is completed.
\end{proof}

\medskip
\noindent {\bf Conflict of interest:} The authors declared that they have no conflicts of interest to this work.

\medskip
\noindent {\bf Data availability:} No data was used for the research described in the article.

\medskip
\noindent {\bf Acknowledgements:} The research of this work was supported  the special foundation for Guangxi Ba Gui Scholars, and the National Natural Science Foundation of China  (No.  12671272).


\end{document}